\documentclass[12pt]{amsart}
\usepackage{mathtools} 
\usepackage{dsfont} 
\usepackage{amssymb}
\usepackage[pdftex]{graphicx} 
\usepackage{color}
\usepackage[colorlinks,citecolor=blue,linkcolor=blue,urlcolor=blue]{hyperref} 
\usepackage{cleveref}

\usepackage[utf8]{inputenc} 
\usepackage[english]{babel} 

\usepackage{geometry} 
\usepackage{tikz} 
\usetikzlibrary{calc,intersections}
\usepackage{cite} 
\usepackage{comment}
\usepackage{float} 

\newtheorem{theorem}{Theorem}

\newtheorem{lemma}[theorem]{Lemma}
\newtheorem{coro}[theorem]{Corollary}
\newtheorem{definition}[theorem]{Definition}
\newtheorem{notation}[theorem]{Notation}
\newtheorem{proposition}[theorem]{Proposition}
\newtheorem{example}[theorem]{Example}
\newtheorem{observation}[theorem]{Observation}
\newtheorem{Remark}[theorem]{Remark}

\newtheorem{question}[theorem]{Question}

\newtheorem{q}[theorem]{Open question}
\newtheorem{fakethm}{Theorem}
\newtheorem{fakeprop}[fakethm]{Proposition}

\numberwithin{theorem}{section} 
\numberwithin{equation}{section} 
\numberwithin{figure}{section} 

\newcommand*{\R}{\mathbb{R}}
\newcommand*{\C}{\mathbb{C}}

\newcommand*{\Q}{\mathbb{Q}}
\renewcommand*{\H}{\mathbb{H}}
\newcommand*{\Z}{\mathbb{Z}}
\newcommand*{\N}{\mathbb{N}}

\DeclareMathOperator{\PGL}{\mathrm{PGL}}
\DeclareMathOperator{\PSL}{\mathrm{PSL}}
\DeclareMathOperator{\Gr}{\mathrm{Gr}}

\newcommand{\abs}[1]{\left\lvert#1\right\rvert}

\title{Plane geometry of $q$-rationals and Springborn Operations}

\author{Perrine Jouteur, Olga Paris-Romaskevich and Alexander Thomas}

\address{Université de Reims Champagne-Ardenne, CNRS, LMR, UMR 9008, Reims, France}
\email{perrine.jouteur@univ-reims.fr}

\address{CNRS, ICJ UMR 5208, École Centrale de Lyon, INSA Lyon, Université Claude Bernard Lyon 1, Université Jeann Monnet, 69622 Villeurbanne, France}
\email{paro@math.univ-lyon1.fr}

\address{CNRS, ICJ UMR 5208, École Centrale de Lyon, INSA Lyon, Université Claude Bernard Lyon 1, Université Jeann Monnet, 69622 Villeurbanne, France}
\email{athomas@math.univ-lyon1.fr}

\begin{document}

\begin{abstract}
We study the geometry of $q$-rational numbers, introduced by Morier-Genoud and Ovsienko, for positive real $q$. In particular, we construct and analyse the deformed Farey triangulation and the deformed modular surface. We interpret every $q$-rational geometrically as a circle, similar to the famous Ford circles.
Further, we define and study new operations on $q$-rationals, the Springborn operations, which can be seen as a quadratic version of the Farey addition. Geometrically, the Springborn operations correspond to taking the homothety centers of a pair of two circles. As an application, we derive a formula for the $q$-deformed midpoint of two Farey neighbors and we consider a new $q$-deformation of Markov numbers.
\end{abstract}

\maketitle

\begin{center}
   \includegraphics[width=\textwidth, trim={0 50 0 50}, clip]{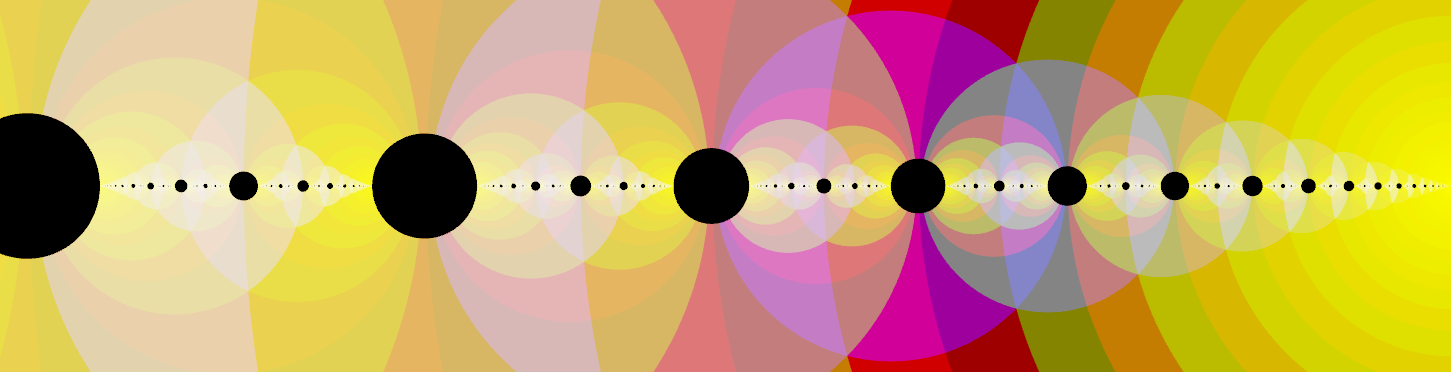}
\end{center}

\setcounter{tocdepth}{1}
\tableofcontents

\section{Introduction}

\noindent \textbf{Context and motivations.}
In their seminal paper \cite{MGO-2020}, Morier-Genoud and Ovsienko introduced the notion of a $q$-deformed rational number, using a deformed version of the continued fraction development. These $q$-rationals have astonishing positivity and convergence properties.
Equivalently, one can define $q$-rationals via a deformation of the Farey addition $$\frac{a}{b}\oplus_F \frac{c}{d}=\frac{a+c}{b+d},$$ where $\tfrac{a}{b}, \tfrac{c}{d}$ are two rationals with Farey determinant $\lvert ad-bc\rvert =1$. 
Still another equivalent description of $q$-rationals can be obtained by deforming the action of the modular group $\PSL_2(\Z)$ on the hyperbolic plane $\mathbb{H}^2$ by Möbius transformations (fractional linear transformations). The usual generators $S$ and $T$ of $\PSL_2(\Z)$ get deformed as follows:
\begin{equation*}
\begin{array}{c c c}
T = \begin{pmatrix}
		1 & 1\\
		0 & 1\\
		\end{pmatrix} & \overset{q}{\rightsquigarrow} & 
T_q := \begin{pmatrix}
		q & 1\\
		0 & 1\\
		\end{pmatrix},\\[20pt]	
S = \begin{pmatrix}
		0 & -1\\
		1 & 0\\
		\end{pmatrix} & \overset{q}{\rightsquigarrow} &
	 S_q := \begin{pmatrix}
		0 & -1\\
		q & 0\\
		\end{pmatrix}.\\
	\end{array}
\end{equation*}

Soon after, another version of $q$-rationals was discovered, first noticed in \cite[Remark 3.2]{MGO-2022} and introduced formally by Bapat-Becker-Licata \cite{BBL}. It is called the left version (while the first version is then called the right version), and has similar positivity and convergence properties. In their work, Bapat-Becker-Licata stipulate the idea of thinking about a $q$-rational number as a hyperbolic geodesic, with endpoints given by the left and right $q$-deformation. For this geometric picture to hold, it is necessary to specialize the formal parameter $q$ to be a positive real number\footnote{For negative real numbers, the matrices $T_q$ and $S_q$ have negative determinant, and the representation can be non-faithful.}.

The starting point of this work is a deepened investigation of the geometry of $q$-rational numbers, using the geometry of the hyperbolic plane and the symmetries of the Farey triangulation. By doubling the hyperbolic geodesics, using complex conjugation, we advocate for thinking of a $q$-rational as a \emph{disk} in $\C$. 
These disks, indexed by $\mathbb{QP}^1$, are ordered along the real line the same way as corresponding rationals. The figure below the abstract, drawn using Shadertoy \cite{shader}, shows some of these disks.

Natural objects associated to rationals and the modular group action on them are the Farey triangulation and the modular surface. Recently, Simon \cite[Section 3.2]{Simon} has described a deformation of the modular surface, which becomes a hyperbolic orbifold with a unique funnel. We independently arrive at the same description, and go further, making the link to $q$-rationals. The preimage of the unique geodesic around the funnel in $\H^2$, seen as universal cover, is the set of hyperbolic geodesics associated to $q$-rationals.

The visualization of the $q$-rationals via disks allowed us to notice a completely unexpected property: for many pairs of disks, the intersection point of the inner or outer common tangents (the inner or outer homothety center of the two disks) lies at the boundary of another disk! If the two initial disks are indexed by rational numbers $\tfrac{a}{b}$ and $\tfrac{c}{d}$, the inner homothety point lies on the boundary of the disk associated to 
\begin{equation}\label{Eq:Springborn-intro}
    \frac{a}{b}\oplus_S \frac{c}{d} = \frac{ab+cd}{b^2+d^2},
\end{equation}
where the fraction on the right hand side might not be reduced. This operation can be seen as a quadratic version of the Farey addition. In a different context (Diophantine approximations of rationals numbers), this operation has been recently studied by Springborn \cite{Springborn}. This is why we call it the \emph{Springborn addition}. A similar operation, the Springborn difference, corresponds to the outer homothety center.

The proof of this property led to the study of involutive symmetries of the classical Farey triangulation and their deformations. Finding the reduced expression of the $q$-deformed Springborn operation uses recent results on $q$-rationals, notably when the denominator is palindromic \cite{Ren, Kogiso}. 

\medskip
\noindent \textbf{Summary and results.}
In Section \ref{Sec:q-rationals} we recall the construction and main properties of $q$-rationals, notably through their symmetry group $\PSL_2(\Z)$ or the extended modular group $\PGL_2(\Z)$. For a rational number $\tfrac{a}{b}\in\Q$, there are two versions: the right $q$-rational $[\tfrac{a}{b}]_q^\sharp = \tfrac{A^\sharp}{B^\sharp}$ and the left $q$-rational $[\tfrac{a}{b}]_q^\flat = \tfrac{A^\flat}{B^\flat}$, where $A^\sharp, B^\sharp, A^\flat, B^\flat\in\Z[q]$.

In Section \ref{Sec:Farey-tiling}, we study the Farey triangulation and its symmetries. The \emph{Farey determinant} of two rational numbers $\tfrac{a}{b}$ and $\tfrac{c}{d}$, denoted by $d_F(\tfrac{a}{b},\tfrac{c}{d})$, is given by $\lvert ad-bc\rvert$. The Farey triangulation is simply the union of all hyperbolic geodesics between pairs of rationals of Farey determinant 1.
Using the $q$-deformed action of the modular group on $\H^2$, we describe the $q$-deformed Farey tesselation and the deformed modular surface. Concretely, we compute the fundamental domain of the deformed $\mathrm{PGL}_2(\Z)$-action \cite{Jouteur}:

\begin{fakeprop}[Proposition \ref{Prop:q-fund-domain}]
The fundamental domain of the $q$-deformed action of $\PGL_2(\Z)$ on $\mathbb{H}^2$ is a hyperbolic ``triangle'' open towards infinity, with two vertices given by $\tfrac{i}{\sqrt{q}}$ and $\sigma=\frac{1+i\sqrt{3}}{2}$.
It is a deformation of the triangle with vertices $i, \sigma$ and $\infty$.
\end{fakeprop}

\begin{figure}[h!]
    \centering
    \includegraphics[height=4.5cm]{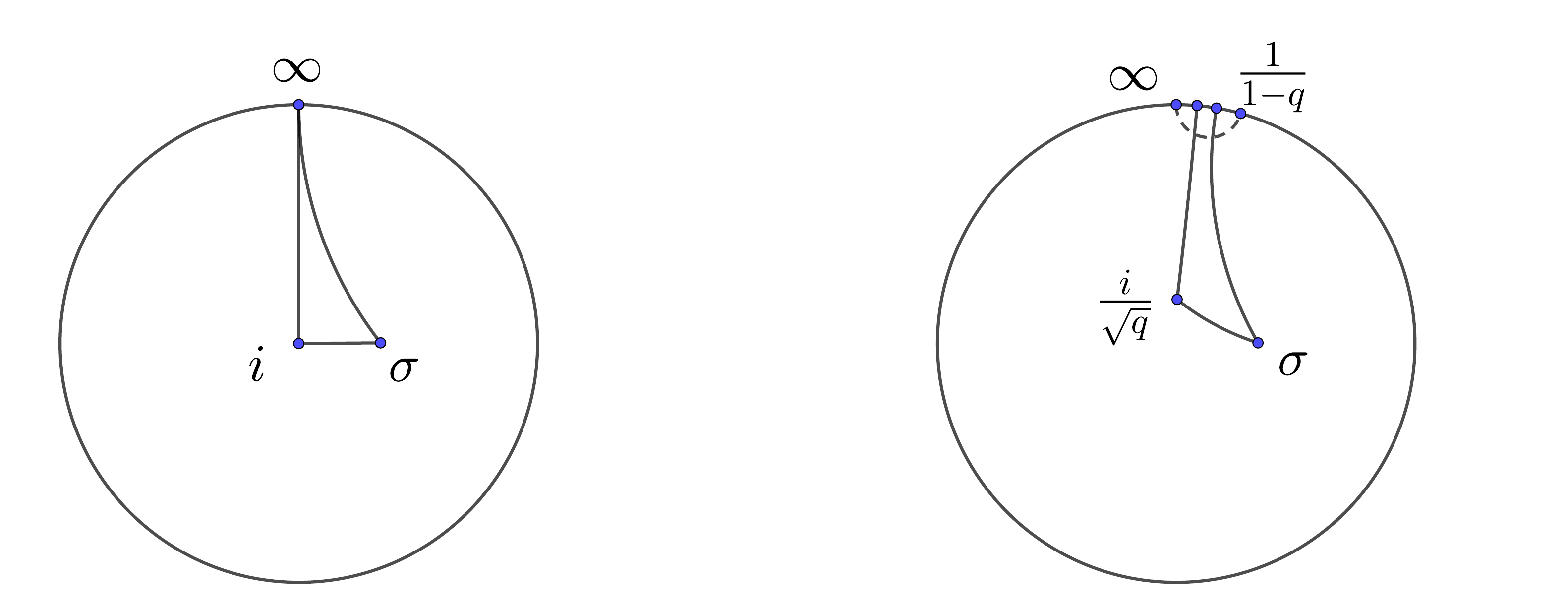}
    \caption{Fundamental domain of $q$-deformed action of $\PGL_2(\Z)$ on $\mathbb{H}^2$.}
    \label{fig:fund-domain-disk-model}
\end{figure}

As a consequence, together with Poincaré's theorem on fundamental polygons, this allows to check that the group which acts is indeed $\mathrm{PGL}_2(\Z)$ (Corollary \ref{Coro:no-extra-relations}).

The deformed modular surface is an orbifold with a unique funnel.
The preimage of the geodesic around the funnel is the set of $q$-rationals, seen as hyperbolic geodesics (the endpoints are the right and left $q$-rationals). We denote by $[\tfrac{a}{b}]$ the geodesic associated to $\tfrac{a}{b}\in \Q$ or the full disk in $\C$. 
From this geometric viewpoint we get that the natural order on $\Q$ is preserved under the deformation to $q$-rationals thought as disks. We also reprove Etingof's gap formula \cite[Prop. 4.6]{Etingof} in Corollary \ref{Coro:Etingof-gap-formula}, using the ergodicity of the geodesic flow on a closed hyperbolic surface.

In Section \ref{Sec:q-Farey} we define a deformed version of the Farey determinant of a pair of rationals $(\tfrac{a}{b},\tfrac{c}{d})\in\Q^2$.
Since there are right and left versions for $q$-rationals, we get four possible notions of $q$-Farey determinants, denoted by $d_F^{\triangle\square}$ with $\triangle, \square\in\{\sharp,\flat\}$, see Definition \ref{Def:q-Farey-det}.

\begin{fakethm}[Theorem \ref{Thm:link-left-right-dF} and Proposition \ref{Prop:positivity-q-dF}]
The four $q$-Farey determinants of any pair $(\tfrac{a}{b},\tfrac{c}{d})\in\Q^2$ have positive integer coefficients: $d_F^{\triangle\square}\in\N[q]$. Moreover, up to multiplication by some monomial in $q$ (denoted further by $\equiv_q$), they are related by
$$  d_F^{\flat \sharp }(q) \equiv_{q}  d_F^{\sharp  \flat}(q^{-1}) \;\;\text{ and }\;\;   d_F^{\flat\flat}(q)  \equiv_q d_F^{\sharp \sharp }(q^{-1}).$$
\end{fakethm}

The proof of this theorem needs the computation of special values of numerators and denominators of $q$-rationals at the value $q=\sigma = \tfrac{1+i\sqrt{3}}{2}$, see Lemma \ref{Lem:special-values}. This contributes to other special values computed for instance in \cite[Section 1.4]{MGO-2020} for $q=-1$, \cite[Section 7]{Kogiso} for $q=\sigma^2$ and \cite{Leclere_these} for roots of unity of order at most 5, see also \cite{Finitness}.

We also generalize the $q$-Farey addition, introduced in the original paper on $q$-rationals \cite[Section 2.5]{MGO-2020} for pairs of rationals of Farey determinant 1, to pairs which are at graph distance 2 in the Farey triangulation, see Theorem \ref{Thm:q-Farey-operations}.

In Section \ref{Sec:Springborn-classic}, we introduce the Springborn operations \eqref{Eq:Springborn-intro}, give geometric interpretations and analyse how to iterate them. An important notion is that of inner and outer regular pairs, for which a reduced expression of the Springborn operation can be given. We concentrate on inner regular pairs in this introduction. A pair of rationals $(\tfrac{a}{b},\tfrac{c}{d})$ is called \emph{inner regular}, if $\mathrm{gcd}(ab+cd,b^2+d^2,a^2+c^2)=d_F(\tfrac{a}{b},\tfrac{c}{d})$. It is easy to see that any pair of Farey determinant 1 or 2 is regular. A geometric characterization is given in Theorem \ref{Thm:charact-regularity}, which states that a pair $(\tfrac{a}{b},\tfrac{c}{d})$ is inner regular if and only if there is an orientation-preserving involution in $\PGL_2(\Z)$ exchanging $\tfrac{a}{b}$ and $\tfrac{c}{d}$.

The core of the paper is Section \ref{Sec:q-Springborn}, in which we compute the $q$-version of the Springborn operations, given by the coordinates of homothetic centers of two circles associated to $q$-rationals.
We denote by $i([\tfrac{a}{b}],[\tfrac{c}{d}])$ the inner homothety center of the two circles $[\tfrac{a}{b}]$ and $[\tfrac{c}{d}]$.\footnote{We also consider the outer homothety center, giving similar results, which are omitted in this introduction.}

Our main result links the Springborn operation to the homothety centers of the associated $q$-disks:
\begin{fakethm}[Theorem \ref{Thm:main}]
Let $(\tfrac{a}{b},\tfrac{c}{d})\in\Q^2$ be inner regular. Then
$$\left[\frac{a}{b}\oplus_S\frac{c}{d}\right]^\sharp _q = i\left(\left[\frac{a}{b}\right],\left[\frac{c}{d}\right]\right).$$
\end{fakethm}

The inner homothety center can be computed explicitely:
\begin{fakethm}[Proposition \ref{prop:homotheticformulas} and Theorem \ref{Thm:q-gcd}]
Let $\left(\frac{a}{b},\frac{c}{d}\right)\in\mathbb{Q}^2$ be an inner regular pair. Then there are explicit integers $\varepsilon_1,\varepsilon_2$ such that
$$
i\left(\left[\frac{a}{b}\right],\left[\frac{c}{d}\right]\right) = \frac{q^{\varepsilon_2}A^\sharp B^{\flat}+q^{\varepsilon_1}C^{\flat} D^\sharp }{q^{\varepsilon_2}B^\sharp B^{\flat}+q^{\varepsilon_1}D^\sharp D^{\flat}},$$
and the greatest common divisor of numerator and denominator is (up to a monomial in $q$) given by $d_F^{\sharp\flat}$.
\end{fakethm}

As application of the two theorems above, we derive a formula for the mid-point of two Farey neighbors:
\begin{fakethm}[Theorem \ref{Coro:q-midpoint}]
Consider two rational numbers $0 < \tfrac{a}{b} < \tfrac{c}{d}$ of Farey determinant 1. Then there is an explicit integer $\varepsilon$ such that the left $q$-version of their mid-point is represented as a reduced fraction by:
$$\left[\frac{1}{2}\left(\frac{a}{b} + \frac{c}{d}\right)\right]_q^\flat = \frac{A^\sharp D^\flat + q^{\varepsilon}C^\sharp B^\flat}{B^\sharp D^\flat + q^{\varepsilon}B^\flat D^\sharp}.$$
\end{fakethm}

Another application shows that the size of the $q$-disks associated to rationals decreases within the Farey tree, see Theorem \ref{Thm:diam-decrease}.

In the final Section \ref{Sec:Markov}, we study combinatorial interpretations for the $q$-Springborn operations in the case of Markov fractions, introduced by Springborn in \cite{Springborn}. They are rational numbers with a Markov number in the denominator, obtained by iteration of the Springborn sum on inner regular pairs. The $q$-Springborn sum gives a new notion of $q$-Markov numbers, and we deduce from our main result a deformed Markov equation they satisfy, see Theorem \ref{thm:qdeformedequations}.

\medskip
\noindent \textbf{Notation.} 
\begin{itemize}
    \item For a rational number $\tfrac{r}{s}\in\Q$, we denote by $R^\sharp, S^\sharp \in\Z[q]$ the numerator and denominator of the right $q$-rational associated to $\tfrac{r}{s}$. Similarly, we denote by $R^\flat, S^\flat \in\Z[q]$ the numerator and denominator of its left $q$-version.
    \item For a pair of reduced fractions $(\tfrac{a}{b},\tfrac{c}{d})\in\Q^2$, we often use $d_F=\lvert ad-bc\rvert$, their Farey determinant.
    \item We put $\sigma=\frac{1+i\sqrt{3}}{2}$, the 6th root of unity and fixed point of $TS=T_q S_q$, the order 3 element of the modular group.
    \item For two polynomials $A,B\in \Z[q,q^{-1}]$, we write $A\equiv_q B$ if there is an integer $k\in\Z$ such that $A(q)=\pm q^k B(q)$.
\end{itemize}

\medskip
\noindent \textbf{Acknowledgements.}
We had many fruitful discussions and exchanges of ideas, for which we want to thank Pierre-Louis Blayac, Vladimir Fock, Summer Haag, Cyril Lecuire, Julien Marché, Théo Marty, Sophie Morier-Genoud, Valentin Ovsienko, Serge Parmentier, Ivan Rasskin, Charles Reid, Barbara Schapira, Bruno Sevennec, Christopher-Lloyd Simon, Boris Springborn, Katherine Stange and Andrei Zabolotskii. For the mathematical illustrations, A. T. is very grateful to Steve Trettel for his course on Shadertoy, to Sebastian Manecke for his improvement of the illustration and to the semester program about mathematical illustration at IHP in 2026. 

We are grateful to our institutions, Université de Reims Champagne-Ardenne and Université Claude Bernard Lyon 1, as well as to the Institut Henri Poincaré, where we regularly met for discussions.

O. P.-R. has been supported by the ANR grant GALS ANR-23-CE40-0001. 
A. T. has been supported by a BQR grant of Université Claude Bernard Lyon 1.

\medskip
\noindent \textbf{AI disclosure.}
The first version of this paper has been written without any use of AI. A referee sent us a document with comments from ChatGPT $6$ Astra, which helped us improving the paper.

\section{From $q$-integers to $q$-rationals to $q$-reals}\label{Sec:q-rationals}
\subsection{Definition(s) of $q$-rationals}

Let $q$ be a formal parameter. Recall the classical $q$-integers that deform any positive integer $n\in \Z_{\geq 0}$ as follows
$$[n]_q := \frac{1-q^n}{1-q}=1+q+\ldots+q^{n-1}\in \Z[q].$$

One of various illustrations of the interest of $q$-integers is the fact that they generalize classical combinatorial objects in the setting of vector spaces. 

\begin{example} 
Define $[n]!_q:=[n]_q \cdot [n-1]_q \cdot \ldots [1]_q$ and $\binom{n}{k}_q:=\frac{[n]!_q}{[k]!_q \cdot [n-k]!_q}$. Then, for the finite field $\mathbb{F}_q$, the $q$-binomial coefficient counts the number of $k$-dimensional subspaces of $\mathbb{F}_q^n$: $\left|\Gr(k,n)(\mathbb{F}_q)\right|=\binom{n}{k}_q$.
\end{example}

Morier-Genoud and Ovsienko in \cite{MGO-2020} extended this $q$-deformation to all rational numbers. Their construction relies on a deformation of the standard action of the modular group $\PSL_2(\Z)$ on the projective line $\mathbb{QP}^1$ by fractional linear transformations. It is defined via the following deformation of the two generators of the modular group $\PSL_2(\Z)= \left<T,S ~|\; S^2=(TS)^3=1\right>$:
\begin{equation}\label{eq:q_representation}
\begin{array}{c c c}
T = \begin{pmatrix}
		1 & 1\\
		0 & 1\\
		\end{pmatrix} & \overset{q}{\rightsquigarrow} & 
T_q := \begin{pmatrix}
		q & 1\\
		0 & 1\\
		\end{pmatrix},\\[20pt]	
S = \begin{pmatrix}
		0 & -1\\
		1 & 0\\
		\end{pmatrix} & \overset{q}{\rightsquigarrow} &
	 S_q := \begin{pmatrix}
		0 & -1\\
		q & 0\\
		\end{pmatrix}.\\
	\end{array}
\end{equation}

At this formal level, the matrices $T_q$ and $S_q$ are elements of $\PGL_2(\Z[q,q^{-1}])$. One checks that as Möbius transformations, we still have $S_q^2=(T_qS_q)^3 = 1$. Hence the group generated by $T_q$ and $S_q$ is still $\PSL_2(\Z)$.

Given a matrix $M\in \PSL_2(\Z)$, its $q$-analogue $M_q$ is obtained by replacing $T$ by $T_q$ and $S$ by $S_q$ in the expansion of $M$ as a finite product of generators $T$ and $S$. This does not depend on the choice of representing $M$ as product of $T$ and $S$.

\begin{notation}
By convention, the reduced form of a rational number $x \in \Q$ is a fraction $x = \frac{a}{b}$ with $a$ and $b$ two coprime integers, and such that $b\in \Z_{\geq 0}$. Any rational number admits a unique such reduced form. We extend to $\mathbb{QP}^1$ by putting $\infty=\frac{1}{0}$. Throughout the document, we assume that rational numbers are always given in reduced form, if not stated explicitly otherwise. 
\end{notation}

\begin{definition}[\cite{MGO-2020, BBL}]\label{def:q-numbers}
The \emph{right $q$-version}  of a number $\frac{a}{b} \in \mathbb{QP}^1$ is a rational function in $q$ with integer coefficients defined as
\begin{equation}\label{defi:right_q_numbers}
\left[\frac{a}{b}\right]^{\sharp}_q := M_q \cdot \frac{1}{0} = \frac{A^{\sharp}_{a/b}(q)}{B^{\sharp}_{a/b}(q)} \in \Z(q), 
\end{equation}
\noindent where $M\in \PSL_2(\Z)$ is any map such that $M\cdot \frac{1}{0} = \frac{a}{b}$ and where $\PSL_2(\Z)$ acts by M\"obius transformations. Its \emph{left $q$-version} is defined as
\begin{equation}\label{defi:left_q_numbers}
\left[\frac{a}{b}\right]^{\flat}_q := M_q \cdot \frac{1}{1-q} = \frac{A^{\flat}_{a/b}(q)}{B^{\flat}_{a/b}(q)} \in \Z(q).
\end{equation}
We often omit the subscript $a/b$ while dealing with these polynomials when the corresponding rational is fixed along the argument.
By convention, $A^{\square}$ and $B^{\square}$ are coprime polynomials in $q$, and $B^{\square}$ has positive leading coefficient, for $\square \in \{\sharp,\flat\}$. 
\end{definition}

The left version of $q$-rationals was first mentioned in \cite[Remark 3.2]{MGO-2022}, and then studied by Bapat-Becker-Licata in \cite{BBL}. The reason for such naming is that for any $q \in (0,1)$, the left $q$-version of any rational number is strictly smaller than its right $q$-version (see Proof of Proposition 2.14 in \cite{BBL}). The choice of $\frac{1}{0}$ and $\frac{1}{1-q}$ in Definition \ref{def:q-numbers} comes from the fact that these are exactly the two fixed points of $T_q$.

\begin{example} In addition to the usual $q$-integers $[n]_q=[n]_q^\sharp$, we give some more examples:
\begin{itemize}
    \item For $n\in \N_{>0}$, we have $[n]_q^\flat = 1+q+q^2+...+q^{n-2}+q^n$.
    \item $\left[\frac{1}{2}\right]_q^\sharp = \frac{q}{1+q}$ and $\left[\frac{1}{2}\right]_q^\flat = \frac{q^2}{1+q^2}$.
    \item $\left[\frac{7}{5}\right]_q^\sharp = \frac{q^4 + 2q^3 + 2q^2 + q + 1}{q^3 + 2q^2 + q + 1}$ and $\left[\frac{7}{5}\right]_q^\flat = \frac{q^5 + q^4 + 2q^3 + q^2 + q + 1}{q^4 + q^3 + q^2 + q + 1}$.
\end{itemize}
\end{example}

Various equivariance properties of $q$-rationals are summarized in the following:
\begin{proposition}[\cite{Leclere-Morier-Genoud,Jouteur}]\label{Prop:equivariance}
For any $x\in\mathbb{QP}^1$ and $\square\in\{\sharp,\flat\}$, we have
\begin{enumerate}
    \item[(i)] $[x+1]_q^\square = q[x]_q^\square+1$ and $\left[-\frac{1}{x}\right]_q^\square = \frac{-1}{q[x]_q^\square}$,
    \item[(ii)] $\left[\frac{1}{x}\right]_q^\square = \frac{1}{[x]_{q^{-1}}^\square}$ and $[-x]_q^\square = -\frac{1}{q}[x]_{q^{-1}}^\square$.
\end{enumerate}
\end{proposition}
\begin{proof}
Item (i) expresses the equivariance with respect to $T$ and $S$, so follows directly from Definition \ref{def:q-numbers}. For item (ii), we also need the negation map $N_q$ (Equation \eqref{Eq:N_q-equivariance} below) and the duality $g_q$ (Theorem \ref{Prop:duality} 
below). We get
$$\left[\tfrac{1}{x}\right]_q^\sharp = [NSx]_q^\sharp = N_qS_q[x]_q^\flat = N_qS_qg_{q^{-1}}[x]^\sharp_{q^{-1}}=\frac{1}{[x]_{q^{-1}}^\sharp}.$$
The same holds for $[\tfrac{1}{x}]_q^\flat$. Finally, 
$$[-x]_q^\square = [Nx]_q^\square = N_qg_{q^{-1}}[x]_{q^{-1}}^\square = -\frac{1}{q}[x]_{q^{-1}}^\square.$$
\end{proof}

Even though any $q$-rational can be explicitly computed in a finite time, numerous arithmetic patterns of its coefficients represent a challenge to understand. We state several results already obtained in this direction, starting with

\begin{proposition}[\textbf{\emph{Positivity property}},  \cite{MGO-2020,BBL}]\label{Prop:cst-sign-coeffs}
For any number $\tfrac{a}{b}\in\mathbb{Q}\backslash\{0\}$, the four corresponding polynomials $A^\sharp , A^\flat, B^\sharp , B^\flat$ have integer coefficients of constant sign. Moreover, this sign is positive if $\tfrac{a}{b}>0$.
\end{proposition}

\begin{proof}
If $\tfrac{a}{b} \geq 1$, there is a combinatorial interpretation of all coefficients of the four polynomials, as proven in \cite[Theorem 4]{MGO-2020} and \cite[Corollary A.2]{BBL}, see also \cite[Proposition 2.4]{Leclere-Morier-Genoud}. Indeed, these coefficients count some special subsets of the graph dual to a part of the Farey triangulation encoding the continued fraction convergents of $\frac{a}{b}$. So they are all  positive.

If $0\leq \tfrac{a}{b}<1$, from Proposition \ref{Prop:equivariance} the identity $\left[\tfrac{1}{x}\right]^\square_q = \frac{1}{[x]^\square_{q^{-1}}}$, reduces the argument to the previous case.

For $\tfrac{a}{b}<0$, again from Proposition \ref{Prop:equivariance} the identity $[-x]^\square_q = -\tfrac{1}{q}[x]^\square_{q^{-1}}$ reduces the argument to the two previous cases.
\end{proof}

\begin{Remark}
The only two polynomials with coefficients of different signs are the numerator of $[0]_q^\flat = \frac{q-1}{q}$ and the denominator of $[\infty]_q^\flat=\frac{1}{1-q}$.
\end{Remark}

In our investigations, we will need two results giving conditions under which the numerators of two $q$-rationals are identical or palindromic. Recall that a polynomial $B(q)$ is called palindromic if there is an integer $N$ such that $B(q)= B(q^{-1}) \cdot q^{N}$.

\begin{theorem}[Theorem 1.2 in \cite{Ren}]\label{thm:palindromicity}
The polynomial $B^{\sharp}_{a/b}$ is palindromic if and only if $a^2 \equiv 1 \mod \; b$, and the polynomial $B^{\flat}_{a/b}$ is palindromic if and only if $a^2 \equiv -1 \mod \; b$.
\end{theorem}

\begin{theorem}[Theorem 3.5 in \cite{Kogiso}]\label{thm:equal-denominators}
Let $p$ be a positive integer. For irreducible fractions $\frac{a}{p}$ and $\frac{b}{p}$ with $ab \equiv -1 \mod p$, the denominators are equal: $B^{\sharp}_{a/p}(q)=B^{\sharp}_{b/p}(q)$.
\end{theorem}
The authors also conjecture that the inverse holds if $p$ is prime and $a \neq b \mod p$.

\smallskip
An important feature of $q$-rationals is the existence of a transition map that exchanges the left version with the right one while replacing $q$ by $q^{-1}$ :

\begin{theorem}[\cite{Thomas, Jouteur}]\label{Prop:duality}
For any rational $\tfrac{a}{b}\in \mathbb{QP}^1$, we have
$$g_q\left(\left[\frac{a}{b}\right]_q^\sharp \right)= \left[\frac{a}{b}\right]^\flat_{q^{-1}}   \;\;\text{ and } \;\; g_q\left(\left[\frac{a}{b}\right]_q^\flat\right) = \left[\frac{a}{b}\right]^\sharp _{q^{-1}},   $$
where $g_q(x) = \frac{1+(x-1)q}{1+(q-1)x}$ is called the \emph{transition map}. In other words, there is a global map that exchanges left and right $q$-versions of rational numbers : a composition of the transition map (acting on rational fractions in $q$) and that of the duality map on the set of parameters : $q \mapsto \frac{1}{q}$.
\end{theorem}

\begin{proof}
The first statement is proven in \cite[Theorem 2.7]{Thomas}, the second follows from the first together with a remark that $g_q \circ g_{q^{-1}} = \mathrm{id}$. This statement also follows from \cite[Theorem 1.7]{Jouteur}, by applying the argument of the theorem to the identity matrix.
\end{proof}

\subsection{Extension of the deformation from $\PSL_2(\Z)$ to $\PGL_2(\Z)$.}

The $q$-action of $\PSL_2(\Z)$ was extended to that of $\PGL_2(\Z)$ in \cite{Jouteur}, showing that under this larger $q$-action the left and right versions of $q$-rationals form a single orbit. 

Since $\PGL_2(\Z)=\left<T,S,N |\; S^2=(TS)^3=1, N^2=(NT)^2=(NS)^2=1\right>$, where $N : x \mapsto -x$ is the negation map, the extension is defined by $q$-deforming $N$ into
\begin{equation}
\label{defi:Nq}
N_q  = \begin{pmatrix}-1 & 1-q^{-1}\\q-1 & 1\\\end{pmatrix}.
\end{equation}
Moreover, for any matrix $M \in \PGL_2(\Z)$ with $\det M =-1$ and $x\in \mathbb{QP}^1$, the following relations hold (see Theorem 1.5 in \cite{Jouteur}):
\begin{equation}\label{Eq:N_q-equivariance}
M_q\cdot [x]_q^{\sharp} = [M\cdot x]_q^{\flat} \;\; \text{ and }\;\;
M_q\cdot [x]_q^{\flat} = [M\cdot x]_q^{\sharp}.
\end{equation}

For any $q \in \R_{>0}$, the map $N_q$ sends the upper half-plane to the lower half-plane. In order to define (and deform) $\PGL_2(\Z)$ inside the group $\mathrm{Isom}(\H^2)$, one should compose $N$ (and $N_q$) with the complex conjugation $c : z \mapsto \bar{z}$, and define $\PGL_2(\Z)$ as generated by this composition and its products with $T$ and $S$.
Putting
$$s_1 = SNc \;, \; s_2 = Nc \;, \; s_3 = TNc,$$
we get another well-known presentation of $\PGL_2(\Z)$, seen as Coxeter group:
\begin{equation}\label{Eq:presentation-pgl2}
    \PGL_2(\Z) = \langle s_1,s_2,s_3\mid s_i^2 = (s_1s_2)^2=(s_1s_3)^3 = 1 \rangle.
\end{equation}

\subsection{From $q$-rationals to $q$-reals and beyond}

The $q$-rational numbers satisfy convergence properties leading to the definition of $q$-irrational numbers, see \cite[Thm 1]{MGO-2022}. In the formal setting, the convergence takes place as Laurent series (expanded around $q=0$) in the $q$-adic topology. We only state the version of this result with $q$ specialized to a positive real value (smaller than 1), see \cite[Thm 2.11]{BBL}. 

\begin{theorem}[\textbf{\emph{Convergence property}}, \cite{MGO-2022, BBL}]\label{Thm:convergence-q-numbers}
Let $q\in (0,1)$ and $(x_n)_{n\in \mathbb{N}}$ be a convergent sequence of rational numbers with limit $\ell$. 
\begin{itemize}
    \item If $\ell$ is irrational, then both $([x_n]_q^\sharp)_{n\in \N}$ and $([x_n]_q^\flat)_{n\in \N}$ converge to the same real number, denoted by $[\ell]_q$.
    \item If $(x_n)_{n\in \mathbb{N}}$ converges to $\ell\in\Q$ from the right (resp. from the left), both $([x_n]_q^\sharp)_{n\in \N}$ and $([x_n]_q^\flat)_{n\in \N}$  converge to $[\ell]_q^\sharp$ (resp. to $[\ell]_q^\flat$).
\end{itemize}
\end{theorem}

Note that, unlike the case of rational numbers, the left and right versions of $q$-irrationals coincide. These properties become visualized in our geometric approach which represent $q$-rationals as circles, see for instance Figure \ref{fig:qrationals}.

The geometric arguments in \cite{BBL}, as well those that we develop in this work, show that for any $x \in \R \setminus \Q$, its quantification $[x]_q \in \R$ is well-defined for any $q \in \R_{+}$.

\begin{question}
For any fixed $x \in \R \setminus \mathbb{Q}$, consider the function on the half-line $ [x] : \R_+ \rightarrow \R$, defined by $[x] : q \mapsto [x]_q$. This function is continuous in $q \in \R_+$, and analytic near $q=0$. What can one say about its behavior outside its convergence radius? What about its smoothness class, does it depend in any way on the Diophantine properties of $x$? 
\end{question}

For an irrational number $x\in \R\setminus \Q$, the expression for $[x]_q$ is given by a formal Laurent series. The question of convergence of this series was raised by Leclere, Morier-Genoud, Ovsienko and Veselov in \cite{Leclere-radius}, where this convergence was proven for any rational $x \in \Q_{>1}$ in the disk centered at $0$ of radius $R_1:=3-2 \sqrt{2}$. It was also conjectured in \cite{Leclere-radius} that for any $x \in \R_{>1}$ the corresponding Laurent series defining its $q$-analogue $[x]_q$ converges  in the disk of radius $R_{\star}:=\frac{3-\sqrt{5}}{2}$ centered at $0$. This was proven for all $x \in \Q$ in \cite{Elzenaar2024BoundingDS} using the theory of Kleinian groups, and recently for all $x\in\mathbb{R}$ in \cite{Etingof-Ovsienko}. Then, in \cite{Etingof} the above conjecture is proven for all real $x$, in a weaker form (in the disk of radius $R_2=2-\sqrt{3}<R_{\star}$).

\smallskip

Finally, one could ask : what about $q$-complex numbers? Two different approaches were proposed by Ovsienko in \cite{Ovsienko} (quantization of Gaussian integers) and Etingof \cite{Etingof} (quantization of all complex numbers), and we are currently working on another (geometric) approach that we will present in the upcoming work. 

In the following, we deal exclusively with $q$-rationals for $q \in \R_+$.

\section{Hyperbolic geometry and deformed Farey tesselation}\label{Sec:Farey-tiling}

For a fixed positive real $q$, we propose a geometric picture of $q$-rationals, using the Farey triangulation of the hyperbolic plane.

\subsection{Classical Farey triangulation}

We present two constructions of the classical Farey triangulation. The first uses geodesics connecting certain pairs of rationals, the second uses the representation of $\mathrm{PGL}_2(\Z)$ as a reflection group.

\begin{definition}\label{defi:farey_det}
The \emph{Farey determinant} $d_F:\Q\times \Q\to \mathbb{N}$ is the function which to a given pair of rational numbers in reduced form $\left(\frac{a}{b},\frac{c}{d}\right)$ associates the quantity
$$
d_F\left(\frac{a}{b},\frac{c}{d}\right) = \abs{ad-bc}.
$$
We can extend this definition to a pair in $\mathbb{QP}^1$ by writing $\infty=\tfrac{1}{0}$.

The \emph{Farey sum} and the \emph{Farey difference} of the pair are  
\[
\frac{a}{b}\oplus_F \frac{c}{d} := \frac{a+c}{b+d} \text{ and } \frac{a}{b}\ominus_F \frac{c}{d} := \frac{a-c}{b-d}.
\]
\end{definition}
Note that the resulting fractions of the Farey operations are not necessarily reduced. These expressions are reduced if the pair is of Farey determinant 1.

\begin{proposition}\label{Prop:invariance-Farey}
The Farey determinant is invariant under the $\mathrm{SL}_2(\Z)$-action, and the Farey operations are equivariant under the same action. The pairs of Farey determinant $n$ belong to $n$ orbits, each having a representative given by $\left(\frac{1}{0},\frac{k}{n}\right)$ with some $k \in [0,n)$.
\end{proposition}
Note that the Farey operations are preserved as a pair $(\frac{a}{b}\oplus_F \frac{c}{d}, \frac{a}{b}\ominus_F \frac{c}{d})$. The individual operations can be exchanged. This comes from the fact that $\frac{a}{b}\oplus_F \frac{c}{d}=\frac{a}{b}\ominus_F \frac{-c}{-d}$.
\begin{proof}
The invariance of the Farey determinant is clear from the matrix perspective: 
$ d_F \left (\tfrac{a}{b},\tfrac{c}{d}\right)= \left| 
\mathrm{det}M
\right|,$ with $M=\left(\begin{smallmatrix}a & c \\ b& d\end{smallmatrix}\right)$. The Farey operations correspond to $M\binom{1}{1}$ and $M\binom{1}{-1}$. Since the $\mathrm{SL}_2(\Z)$-action preserves the space of reduced fractions, its action on  pairs is simply by matrix multiplication. Hence the Farey determinant is unchanged, while the Farey operations commute with the $\PSL_2(\Z)$-action.

Since $\mathrm{SL}_2(\Z)$ acts transitively on $\mathbb{QP}^1$, we can send one element of a pair to $\infty = \tfrac{1}{0}$. The other element is then of the form $\tfrac{k}{n}$, where $n$ is the Farey determinant of the original pair. Using the action of $T\in\mathrm{Stab}(\infty)$, we can impose $k\in\{0,1,...,n-1\}$.
\end{proof}

\begin{lemma}\label{Lemma:dF-for-Farey-op}
Let $\frac{a}{b}$ and $\frac{c}{d}$ be two rational numbers in reduced form. Then  $$d_F\left(\frac{a}{b},\frac{c}{d}\right)\in\{1,2\} \;\; \Longrightarrow \;\; d_F\left(\frac{a}{b}\oplus_F\frac{c}{d},\frac{a}{b}\ominus_F\frac{c}{d}\right)\in\{2,1\}.$$

\end{lemma}

\begin{proof} Let us denote by $d_F$ the Farey determinant between $\frac{a}{b}$ and $\frac{c}{d}$, and suppose without loss of generality that $\frac{c}{d} < \frac{a}{b}$. Note the following relations : $$a(b+d) - b(a+c) = d_F, \; \textit{and} \; \; b(a-c) - a(b-d) = d_F.$$ 

Let us suppose first that $d_F = 1$. Then the fractions $\frac{a+c}{b+d}$ and $\frac{a-c}{b-d}$ are reduced, and the Farey determinant between those is given by 
$$
d_F\left(\frac{a+c}{b+d},\frac{a-c}{b-d}\right) = \abs{(a+c)(b-d) - (a-c)(b+d)} = 2\abs{ad-bc} = 2.
$$
Now, suppose that $d_F = 2$. Then $\gcd(a+c,b+d) = \gcd(a-c,b-d) = 2$, and the Farey determinant we are interested in is given by 
$$
d_F\left(\frac{a+c}{b+d},\frac{a-c}{b-d}\right) = \abs{\frac{a+c}{2} \cdot \frac{b-d}{2} - \frac{a-c}{2} \cdot \frac{b+d}{2}} = \frac{2d_F}{4} = 1.
$$
\end{proof}

\begin{definition}
The \emph{Farey triangulation} $\mathcal{F}$ is a tesselation of the hyperbolic plane by ideal triangles, delimited by all hyperbolic geodesics between two rationals in $\mathbb{QP}^1$ of Farey determinant $1$.
\end{definition}

\begin{figure}[h!]
\centering
\begin{tikzpicture}[scale=1.8]
    \draw (0,0) circle (1);
    \draw (0,1)--(0,-1);
    \draw (0,1) arc (180:270:1);
    \draw (0,-1) arc (180:90:1);
    \draw (0,1) arc (0:-90:1);
    \draw (0,-1) arc (0:90:1);
      \node at(0,1.3) {$\frac{1}{0}$};
    \node at(0,-1.3) {$\frac{0}{1}$};
    \node at(1.2,0) {$\frac{1}{1}$};
    \node at(-1.3,0) {$-\frac{1}{1}$};
     \node at(1.0,0.8) {$\frac{2}{1}$};
    \node at(1.0,-0.8) {$\frac{1}{2}$};
    \node at(-1,-0.8) {$-\frac{1}{2}$};
    \node at(-1.0,0.8) {$-\frac{2}{1}$};
    
    \node at(0.7,1) {\tiny{$\frac{3}{1}$}};
     \node at(-0.7,1) {\tiny{$-\frac{3}{1}$}};
    \node at(0.7,-1) {\tiny{$\frac{1}{3}$}};
     \node at(-0.7,-1) {\tiny{$-\frac{1}{3}$}};
  
   \node at({12/13+0.1},{5/13+0.1}) {\tiny{$\frac{3}{2}$}};
     \node at({-12/13-0.15},{5/13+0.1}) {\tiny{$-\frac{3}{2}$}};
    \node at ({12/13+0.1},{-5/13-0.1}) {\tiny{$\frac{2}{3}$}};
     \node at ({-12/13-0.15},{-5/13-0.1}) {\tiny{$-\frac{2}{3}$}};
  
    \draw (1,0) arc[radius=0.3333, start angle=270, end angle=126.87];
    \draw (-0.8,0.6) arc (53.13:-90:0.3333);
    
    \draw (0,1) arc (180:306.87:0.5);
    \draw (0,1) arc (0:-126.87:0.5);
    
    \draw (1, 0) arc (90:233.13:0.3333);
    
    \draw (0.8, -0.6) arc (53.13:180:0.5);
    
    \draw (0, -1) arc (0:126.87:0.5);
    
    \draw (-1, 0) arc (90:-53.13:0.3333);
    
\draw (0,1) arc (180:323.13:0.3333);
\draw (0,1) arc (0:-143.13:0.3333);
    
    \draw (0.6,0.8) arc (143.13:306.87:0.1414);
    \draw (-0.8, 0.6) arc (233.13:396.87:0.1428);

    \draw (0.6,-0.8) arc (36.87:180:0.333);
    \draw (-0.6,-0.8) arc (143.13:0:0.333);
    	
 \draw (0.8, -0.6) arc (53.13:216.87:0.14286);
 \draw (-0.8, -0.6) arc (126.87:-36.87:0.14286);
    
\draw (0.8, 0.6) arc (126.87:292.62:0.125);
\draw (-0.8, 0.6) arc (53.13:-112.62:0.125);

\draw (1, 0) arc (270:112.62:0.2);
\draw (-1, 0) arc (270:427.38:0.2); 

\draw (1, 0) arc (90:247.38:0.2);
\draw (-1, 0) arc (90:-67.38:0.2);
	
\draw (0.8, -0.6) arc (233.13:67.38:0.125);
\draw (-0.8, -0.6) arc (306.87:472.62:0.125);
\end{tikzpicture}
\label{fig:farey_tesselation}
\caption{The ideal triangles of the classical Farey triangulation $\mathcal{F}$ attached to rationals with denominator at most $3$. Note that the rational points are not equidistant on the boundary.}
\end{figure}

The second construction of the Farey triangulation starts from the presentation of $\PGL_2(\Z)$ as a subgroup of $\mathrm{Isom}(\H^2)$ given by (see  \eqref{Eq:presentation-pgl2}):
 $$\PGL_2(\Z) = \langle s_1,s_2,s_3\mid s_i^2 = (s_1s_2)^2=(s_1s_3)^3 = 1 \rangle.$$

\begin{proposition}
Consider the hyperbolic triangle $\Delta_0$ with vertices $i, \sigma=\tfrac{1+i\sqrt{3}}{2}$ and $\infty$. The reflection group generated by its three sides is $\PGL_2(\Z)$.
\end{proposition}
\begin{proof}
We work in the upper half plane model of $\mathbb{H}^2$. A direct computation shows that the fixed point set of $s_1$ is the unit circle, which is the geodesic between $i$ and $\sigma$. The fixed point set of $s_2$ is the imaginary axis, the geodesic between $i$ and $\infty$, and the fixed point set of $s_3$ is $\mathrm{Re}(z)=\tfrac{1}{2}$, the geodesic between $\sigma$ and $\infty$.
\end{proof}

The tesselation generated by the reflection group and the base triangle $\Delta_0$ is a refinement of the Farey tesselation, where each Farey triangle gets subdivided into six smaller triangles.

\begin{figure}[H]
\centering
\begin{tikzpicture}[scale=1.8]
    \draw (0,0) circle (1);
    \draw (0,1)--(0,-1);
    \draw (0,1) arc (180:270:1);
    \draw (0,-1) arc (180:90:1);
    \draw (0,1) arc (0:-90:1);
    \draw (0,-1) arc (0:90:1);
      \node at(0,1.3) {$\frac{1}{0}$};
    \node at(0,-1.3) {$\frac{0}{1}$};
    \node at(1.2,0) {$\frac{1}{1}$};
    \node at(-1.3,0) {$-\frac{1}{1}$};
     \node at(1.0,0.8) {$\frac{2}{1}$};
    \node at(1.0,-0.8) {$\frac{1}{2}$};
    \node at(-1,-0.8) {$-\frac{1}{2}$};
    \node at(-1.0,0.8) {$-\frac{2}{1}$};
    
    \node at(0.7,1) {\tiny{$\frac{3}{1}$}};
     \node at(-0.7,1) {\tiny{$-\frac{3}{1}$}};
    \node at(0.7,-1) {\tiny{$\frac{1}{3}$}};
     \node at(-0.7,-1) {\tiny{$-\frac{1}{3}$}};
  
   \node at({12/13+0.1},{5/13+0.1}) {\tiny{$\frac{3}{2}$}};
     \node at({-12/13-0.15},{5/13+0.1}) {\tiny{$-\frac{3}{2}$}};
    \node at ({12/13+0.1},{-5/13-0.1}) {\tiny{$\frac{2}{3}$}};
     \node at ({-12/13-0.15},{-5/13-0.1}) {\tiny{$-\frac{2}{3}$}};
    
    \draw (1,0) arc[radius=0.3333, start angle=270, end angle=126.87];
    \draw (-0.8,0.6) arc (53.13:-90:0.3333);
    
    \draw (0,1) arc (180:306.87:0.5);
    \draw (0,1) arc (0:-126.87:0.5);
    
    \draw (1, 0) arc (90:233.13:0.3333);
    
    \draw (0.8, -0.6) arc (53.13:180:0.5);
    
    \draw (0, -1) arc (0:126.87:0.5);
    
    \draw (-1, 0) arc (90:-53.13:0.3333);
    
\draw (0,1) arc (180:323.13:0.3333);
\draw (0,1) arc (0:-143.13:0.3333);
    
    \draw (0.6,0.8) arc (143.13:306.87:0.1414);
    \draw (-0.8, 0.6) arc (233.13:396.87:0.1428);

    \draw (0.6,-0.8) arc (36.87:180:0.333);
    \draw (-0.6,-0.8) arc (143.13:0:0.333);
    	
 \draw (0.8, -0.6) arc (53.13:216.87:0.14286);
 \draw (-0.8, -0.6) arc (126.87:-36.87:0.14286);
    
\draw (0.8, 0.6) arc (126.87:292.62:0.125);
\draw (-0.8, 0.6) arc (53.13:-112.62:0.125);

\draw (1, 0) arc (270:112.62:0.2);
\draw (-1, 0) arc (270:427.38:0.2); 

\draw (1, 0) arc (90:247.38:0.2);
\draw (-1, 0) arc (90:-67.38:0.2);
	
\draw (0.8, -0.6) arc (233.13:67.38:0.125);
\draw (-0.8, -0.6) arc (306.87:472.62:0.125);

\definecolor{mygreen}{RGB}{167,62,75}
\draw[mygreen] (-1,0)--(1,0);
  
\draw[mygreen](0, -1) arc (180:126.87:2);
\draw[mygreen] (0, -1) arc (0:53.13:2);
\draw[mygreen](0, 1) arc (180:233.13:2);
\draw[mygreen] (0, 1) arc (0:-53.13:2);
\draw[mygreen] (1, 0) arc (90:216.87:0.5);
\draw[mygreen]  (0, -1) arc (180:67.38:0.6667); 
\draw[mygreen] (1, 0) arc (270:143.13:0.5);
\draw[mygreen] (0, 1) arc (180:292.62:0.6667);
\draw[mygreen] (0, 1) arc (0:-112.62:0.6667);
\draw[mygreen] (-1, 0) arc (270:396.87:0.5);
\draw[mygreen] (-1, 0) arc (90:-36.87:0.5);
\draw[mygreen] (0, -1) arc (0:112.62:0.6667);
    
\end{tikzpicture}
\label{fig:farey_tesselation_refined}
\caption{The subdivided Farey triangulation, consisting in all geodesics between rationals of Farey determinant 1 (black) or 2 (red), up to the denominator at most $3$. This Figure is a refinement of the Figure \ref{fig:farey_tesselation}. The points appearing as intersections of three red lines belong to the orbit of $\sigma$ under the action of the modular group, while the intersections between red and black lines to the orbit of $i$.}
\end{figure}

\subsection{Symmetries of classical Farey triangulation}

We are now studying the symmetries of the Farey triangulation $\mathcal{F}$. In particular, we describe the set of involutive symmetries. The material of this paragraph is classical.

\begin{proposition}
The group of symmetries of the Farey triangulation $\mathcal{F}$ is $\mathrm{PGL}_2(\Z)$.
\end{proposition}

\begin{proof}
The group $\mathrm{PGL}_2(\Z)$ is generated by three reflections, and these reflections are symmetries of $\mathcal{F}$. Hence $\mathrm{PGL}_2(\Z)$ is included in the symmetry group of $\mathcal{F}$. On the other hand, any symmetry of $\mathcal{F}$ is an isometry of $\mathbb{H}^2$ which preserves $\mathbb{QP}^1$, the set of vertices of $\mathcal{F}$. Such an isometry is in $\mathrm{PGL}_2(\mathbb{Q})$. Finally, a matrix in $\mathrm{GL}_2(\mathbb{Q})$ which preserves the Farey determinant (which is necessary to preserve the set of edges in $\mathcal{F}$) is necessarily of determinant $\pm 1$, hence in $\mathrm{GL}_2(\Z)$. Since the action is projective, we get $\PGL_2(\Z)$.
\end{proof}

In order to describe involutive symmetries of $\mathcal{F}$, it will appear natural to introduce Farey determinants for pairs of points in the quadratic imaginary field $\Q[i]$. This ring is Euclidean, so reduced fractions are well-defined.

\begin{definition}\label{Def:dF-in-Qi}
For two reduced fractions $\tfrac{r}{s}$ and $\tfrac{r'}{s'}$ in $\Q[i]$, we define their \emph{Farey determinant} by
$$d_F\left(\frac{r}{s}, \frac{r'}{s'}\right)=\lvert rs'-r's\rvert.$$
\end{definition}

In the sequel, we will only use the case for a pair of complex conjugates.

\begin{proposition}\label{Prop:d_F-complex}
For any matrix $M\in\mathrm{GL}_2(\mathbb{Z})$, we have 
$$d_F(Mz, M\bar{z}) =  d_F(z,\bar{z}).$$
\end{proposition}

In particular, we see that the Farey determinant is preserved under conjugacy.

\begin{proof}
It is sufficient to prove the proposition for $z=i$. Indeed, if $z=\frac{\alpha+i\beta}{\gamma+i\delta}$ is a reduced fraction in $\Q[i]$, then we can replace $M$ by $M\left(\begin{smallmatrix}\alpha & \beta\\ \gamma & \delta\end{smallmatrix}\right)$ to reduce to $z=i$.

For $M=\left(\begin{smallmatrix}a & b\\ c & d\end{smallmatrix}\right)$, a direct computation gives
$$d_F(Mi,-Mi) = \lvert(a+bi)(c-di)-(c+di)(a-bi)\rvert = 2 \lvert ad-bc\rvert.$$
Since $d_F(i,-i)=\lvert 2i\rvert =2$ and $\det M = 1$, the latter equals $d_F(i,-i)$.
\end{proof}

We can now describe the set of involutive symmetries of the Farey triangulation.

\begin{proposition}\label{Prop:inversion-sym}
The set of orientation-reversing involutions in $\mathrm{PGL}_2(\Z)$ is the set of inversions in circles with endpoints in $\mathbb{QP}^1$ of Farey determinant 1 or 2.

The set of orientation-preserving involutions in $\mathrm{PGL}_2(\Z)$ is the set of rotations of $\pi$ around a point $z\in \mathbb{Q}[i]\cap\mathbb{H}^2$ such that $d_F(z,\bar{z})=2$.
\end{proposition}

\begin{proof}
Consider a pair of distinct points in $\mathbb{QP}^1$ with Farey determinant 1 or 2. Then the associate geodesic is an edge in the refined Farey triangulation. Hence the inversion with respect to this geodesic is a symmetry of $\mathcal{F}$.

Conversely, consider an orientation-reversing involution $I$ which is a symmetry of $\mathcal{F}$. From general structure theory of isometries of the hyperbolic plane, we know that $I$ is an inversion with respect to some geodesic $\gamma$.

If $\gamma$ does never intersect the interior of a triangle of $\mathcal{F}$, then $\gamma$ has to be one edge of $\mathcal{F}$, so by definition its endpoints are of Farey determinant 1.

Consider then the case when $\gamma$ intersects the interior of some triangle. We can assume that this triangle is the base triangle $\Delta_0$, by conjugating $I$ if necessary by an element of $\PGL_2(\Z)$ (by definition this action is transitive on triangles of $\mathcal{F}$). Since $I$ is a symmetry of $\mathcal{F}$, it has to fix $\Delta_0$. Hence $\gamma$ is one of the three perpendicular bisectors of $\Delta_0$, which are geodesics with endpoints of Farey determinant 2 (see Figure \ref{fig:farey_tesselation_refined}).

Finally, consider an orientation-preserving involution $I'$ in $\mathrm{PGL}_2(\Z)$, different from identity. Any such involution is the rotation around a point $z$ with angle $\pi$. 
We can conjugate $I'$ inside $\mathrm{PGL}_2(\Z)$ so that the fixed point $z$ lies inside the fundamental triangle $\Delta_0$. If $z$ lies strictly inside $\Delta_0$, then $I'(0)\in (1,\infty)$ and $I'(\infty)\in (0,1)$. But there is no edge in $\mathcal{F}$ between any pair of points in these two intervals. Hence $z$ lies on the boundary of $\Delta_0$. Using the rotation $TS$, we can suppose that $z$ lies on the geodesic between $0$ and $\infty$, so $I'$ exchanges $0$ and $\infty$. Hence $I'$ sends $\Delta_0$ to the triangle $(-1,0,\infty)$. Therefore $z$ is the intersection of the geodesic between 0 and $\infty$ with the one between $1$ and $-1$. Thus $z=i$, for which we have $d_F(i,\bar{i}) = \lvert 2i\rvert = 2$. Since the action of $\PGL_2(\Z)$ does not change the Farey determinant by Proposition \ref{Prop:d_F-complex}, we see that $d_F(z,\bar{z})=2$. 
Conversely, any point $z\in \Q[i]\cap \H^2$ with $d_F(z,\bar{z})=2$ can be conjugated to $i$, so the rotation around $z$ with angle $\pi$ is a symmetry of $\mathcal{F}$.
\end{proof}

\begin{proposition}\label{Prop:exchange-under-involution}
If two distinct rational numbers $\tfrac{a}{b}$ and $\tfrac{c}{d}$ are exchanged under an orientation-reversing involution in $\PGL_2(\Z)$, then the fixed points (in $\R$) of this involution are $\tfrac{a+c}{b+d}$ and $\tfrac{a-c}{b-d}$.

If two distinct rational numbers $\tfrac{a}{b}$ and $\tfrac{c}{d}$ are exchanged under an orientation-preserving involution in $\PSL_2(\Z)$, then the fixed points (in $\C$) of this involution are $\tfrac{a+ci}{b+di}$ and $\tfrac{a-ci}{b-di}$.
\end{proposition}
Note that the fractional expressions of the fixed points need not to be reduced.

\begin{proof}
We can lift the rationals to vectors $v=\binom{a}{b}$ and $w=\binom{c}{d}$. Using the scaling freedom, we can assume that $Mv=w$. For an orientation-reversing, we have $M^2=\mathrm{id}$. Hence $Mw=v$, and the eigenvectors are $v\pm w$.
For an orientation-preserving involution, we have $M^2=-\mathrm{id}$. Hence $Mw=-v$ and the eigenvectors are $v\pm iw$.
\end{proof}

\subsection{$q$-Farey tesselation and $q$-modular surface}

Using the second viewpoint on the Farey tesselation above, we now $q$-deform the construction, using the deformed generators $(S_q,T_q,N_q)$ and complex conjugation $c$.

We consider $q\in (0,1)$, since we can get $q>1$ from the transition map \ref{Prop:duality}.

\begin{proposition}\label{Prop:q-fund-domain}
The fundamental domain of the $q$-deformed action of $\PGL_2(\Z)$ on $\mathbb{H}^2$ in the upper half-plane model, is a quadrilateral with vertices $\tfrac{i}{\sqrt{q}}, \sigma, P_1:=\frac{1+\sqrt{q^2-q+1}}{1-q}$ and $P_2:=\frac{1+\sqrt{q-1+q^{-1}}}{1-q}$.
\end{proposition}

Note that one side of the quadrilateral lies at the boundary at infinity of $\mathbb{H}^2$, see Figure \ref{fig:fundamental_quadrilateral_PGL}. 

\begin{proof}
In analogy to the classical case ($q=1$) we consider the combinations $s_1=S_qN_qc$, $s_2=N_qc$ and $s_3=T_qN_qc$, which satisfy the same relations than in Equation \eqref{Eq:presentation-pgl2}. We determine the fixed point sets of each of the transformations.

Consider first $$z=S_qN_qc(z)=\frac{(q-1)\bar{z}+1}{q\bar{z}+1-q}.$$
Writing $z=x+iy$ the previous equation is equivalent to $q(x^2+y^2)+2(1-q)x-1=0$, which in turn can be written as
\begin{equation}\label{Eq:q-unit-circle}
    \left(x-\frac{q-1}{q}\right)^2+y^2 = \frac{q^2-q+1}{q^2},
\end{equation}
which is the equation of a circle. This is the $q$-deformed unit circle $\mathcal{C}_{1,q}$.

Similarly, $z=N_qc(z) = \frac{-\bar{z}+1-1/q}{(q-1)\bar{z}+1}$ gives $(q-1)(x^2+y^2)+2x+\tfrac{1-q}{q}=0$, which is equivalent to
\begin{equation}\label{Eq:q-imag-axis}
\left(x-\frac{1}{1-q}\right)^2+y^2 = \frac{q^2-q+1}{q(1-q)^2},
\end{equation}
a circle describing the $q$-deformed imaginary axis, which we denote by $\mathrm{Im}_q$.

Finally, $z=T_qN_qc(z) = \frac{-\bar{z}+q}{(q-1)\bar{z}+1}$ gives $(q-1)(x^2+y^2)+2x-q=0$ which is equivalent to
\begin{equation}\label{Eq:q-critical-axis}
\left(x-\frac{1}{1-q}\right)^2+y^2 = \frac{q^2-q+1}{(1-q)^2},
\end{equation}
a circle describing the $q$-deformed line $\mathrm{Re}(z)=\tfrac{1}{2}$, denoted by $\mathcal{C}_{2,q}$.

Finally, we notice that the three fixed geodesics form a quadrilateral. The computation of the vertices is direct.
\end{proof}

\begin{figure}[h!]
\begin{tikzpicture}[scale=0.7]
\def\q{0.45}
\begin{scope}
 \clip (-5,0) rectangle (5,5);
\draw[blue, thick,fill=blue!10] ({1/(1-\q)},0) circle [radius={sqrt((\q*\q - \q + 1)/(\q*(1-\q)*(1-\q))}];

\draw[blue,thick,fill=white] ({(1 + sqrt((1/\q + \q - 1)))/(1 - \q)},0)arc(0:-180:{sqrt((\q*\q - \q + 1)/(\q*(1-\q)*(1-\q))});

\draw[blue, thick,fill=white] ({(\q-1)/(\q)},0) circle [radius={sqrt((\q*\q - \q + 1)/(\q*\q))} ];

\draw[blue, thick,fill=white] ({1/(1-\q)},0) circle [radius={sqrt((\q*\q - \q + 1)/((1-\q)*(1-\q)))} ];

\draw[blue, thick] ({1/(1-\q)},0) circle [radius={sqrt((\q*\q - \q + 1)/(\q*(1-\q)*(1-\q)))}];

\draw[blue, thick] ({(\q-1)/(\q)},0) circle [radius={sqrt((\q*\q - \q + 1)/(\q*\q))} ];
\end{scope}

\draw[->] (-3.3,0) -- (5,0) node[right] {$x$};
\draw[->] (0,0) -- (0,3.1) node[above] {$y$};
\coordinate (A) at (
  { (1 + sqrt((\q*\q - \q + 1))) / (1 - \q) },
  0
);
\fill (A) circle (3pt) node[above left] {$P_1$};
\coordinate (B) at (
  { (1 + sqrt((1/\q + \q - 1))) / (1 - \q) },
  0
);
\fill (B) circle (3pt) node[above right] {$P_2$};

\coordinate (I) at (0,{1/sqrt(\q)});
\fill (I) circle (3pt) node[above left] {$\frac{i}{\sqrt{q}}$};
\coordinate (S) at (0.5,{sqrt(3)/2});
\fill (S) circle (3pt) node[right]
{$\sigma$};

\end{tikzpicture}
\caption{Fundamental quadrilateral for the action of $\PGL_2(\Z)$ on $\mathbb{H}^2$.
}\label{fig:fundamental_quadrilateral_PGL}
\end{figure}

\begin{coro}\label{Coro:no-extra-relations}
For $q\in \R_+$, the group generated by $T_q, S_q$ and $N_qc$ is isomorphic to $\mathrm{PGL}_2(\Z)$.
\end{coro}
\begin{proof}
It is clear that the map $T_q \mapsto T, S_q \mapsto S, N_qc \mapsto Nc$ is a surjective group homomorphism. We only have to show that its kernel is trivial, i.e. that there are not more relations among the generators than expected. This follows from Poincaré's theorem on fundamental polygons (see for instance \cite{Maskit}). It asserts that given a hyperbolic polygon $\mathcal{P}$ in which all angles are of the form $\pi/m$, $m\in\N_{\geq 2}$, then the group $G$ generated by reflections in all sides of $\mathcal{P}$ is a Coxeter group with explicit presentation, and its action on $\H^2$ has $\mathcal{P}$ as fundamental domain.

We apply this theorem to our fundamental domain we found in the previous proposition. We only have to check the angles. At the vertex $\tfrac{i}{\sqrt{q}}$, the angle is $\tfrac{\pi}{2}$ since this vertex is the fixed point of $S_q$, which is of order 2. Similarly, at the vertex $\sigma$, the angle is $\tfrac{\pi}{3}$, since this vertex is the fixed point of $T_qS_q$, which is of order 3.
\end{proof}

\begin{Remark}
Note that the previous Corollary is not true for all real $q$. In the case of $q=-1$ for example, the group generated by $T_{-1}$, $S_{-1}$ and $N_{-1}c$ is finite. More precisely, the corresponding special linear group has presentation $$\left<T_{-1},S_{-1} | T_{-1}^2=S_{-1}^2=(T_{-1}S_{-1})^3=1\right>,$$ so it is the classical dihedral group of order $6$, that is, the symmetric group $S_3$. Furthermore, the corresponding general linear group is a semidirect product $S_3\rtimes \Z/2\Z$. Indeed, $S_3$ is its normal subgroup of index $2$ and it is easily checked that the element $N_{-1}c$ is an antiholomorphic involution such that $N_{-1}c T N_{-1}c = S, N_{-1}c S N_{-1}c = T$.
\end{Remark}

Since one side of the fundamental quadrilateral is along the boundary $\partial_\infty\H^2$, geometrically we have a funnel. There is a unique geodesic which cuts this funnel. In our case, it is given by the vertical line $x=\frac{1}{1-q}$, which is orthogonal both to $\mathrm{Im}_q$ and $\mathcal{C}_{2,q}$. This geodesic links precisely $[\infty]_q^\flat$ and $[\infty]_q^\sharp$. 

\begin{definition}\label{def:discs}
Consider $q\in (0,1)$. To each $x\in \mathbb{QP}^1$, we associate the unique hyperbolic geodesic between $[x]_q^\flat$ and $[x]_q^\sharp $. We denote this geodesic by $[x]$. The union of all these geodesics is denoted by $\mathcal{Q}$.

We denote by $\mathbb{H}_q^2$ the hyperbolic plane in the upper half-plane model from which we removed all (Euclidean) half-disks bounded by the geodesics $[x]$, where $x\in \mathbb{QP}^1$. Doubling the construction using complex conjugation is shown in Figure \ref{fig:qrationals}.
\end{definition}

\begin{figure}[H]
    \centering
\begin{tikzpicture}[scale=1.5]
\fill[gray!30] (1.8182,-1.6) rectangle (3.4,1.6);
\draw[darkgray, line width=1.5pt] (1.8182,-1.6) -- (1.8182,1.6);
\node[font=\large] at (2.9,0.2) {$\infty$};
\filldraw[fill=gray!30, draw=darkgray, line width=1pt]
  (-3.58025,0) circle (1.35803);
\node[font=\large] at (-3.58025,0.25) {$-1$};
\filldraw[fill=gray!30, draw=darkgray, line width=1pt]
  (-1.69035,0) circle (0.15765);
\node[font=\footnotesize]  at (-1.69035,0.25) {\small{$-1/2$}};
\filldraw[fill=gray!30, draw=darkgray, line width=1pt]
  (-0.61111,0) circle (0.61111);
\node[font=\large] at (-0.61111,0.25) {$0$};
\node[font=\large, above] at (-0.61111, 1) {$\mathbb{H}_q$};
\node[font=\large, below] at (-0.61111, -1) {$\overline{\mathbb{H}}_q$};
\filldraw[fill=gray!30, draw=darkgray, line width=1pt]
  (0.725,0) circle (0.275);
\node[font=\small] at (0.725,0.10) {$1$};
\filldraw[fill=gray!30, draw=darkgray, line width=1pt]
  (1.32625,0) circle (0.12375);
\node at (1.32625,0) {\scriptsize{$2$}};
\filldraw[fill=gray!30, draw=darkgray, line width=1pt]
  (0.23935,0) circle (0.07095);
\node[font=\footnotesize]  at (0.23935,0.17) {\small{$1/2$}};
\draw[->, line width=1pt] (-5.2,0) -- (3.5,0) node[right] {$x$};
\draw[->, line width=1pt] (0,-1.5) -- (0,1.5) node[above] {$y$};
\end{tikzpicture}
    \caption{Representation of the $q$-disks corresponding to the numbers $-1, 1/2, 0, 1/2, 1, 2$ and $\infty$. Here, the parameter $q$ is specified to $q=0.45$. The ``$q$-disk'' corresponding to $\infty$ is a half-plane bordered by a vertical line $x=\frac{1}{1-q}$.}\label{fig:qrationals}
\end{figure}

The convergence properties of $q$-rationals and $q$-irrationals from Theorem \ref{Thm:convergence-q-numbers} can be well-understood in from Definition \ref{def:discs} and visualization of the Figure \ref{fig:qrationals}.

\begin{proposition}\label{Prop:q-Farey-topo}
For $x, y\in \mathbb{QP}^1$ with $x\neq y$, the geodesics $[x]$ and $[y]$ are disjoint.
\end{proposition}

\begin{proof}
This follows from the geometric viewpoint: denote by $\gamma$ the unique geodesic cutting the funnel of the $q$-deformed modular surface, and by $\pi:\H^2\to \H^2/\mathrm{PGL}_2(\Z)$ the universal covering map. Then $\mathcal{Q}=\pi^{-1}(\gamma)$, hence all geodesics in $\mathcal{Q}$ are pairwise disjoint.
\end{proof}

\begin{Remark}
The previous proposition also follows from the order preservation of $q$-rationals, see Corollary \ref{Coro:well-orderedness}.
\end{Remark}

\begin{Remark}
The order preservation of $q$-rationals might come as a surprise on the first sight, since rational numbers are already dense in $\R$. For $x\in\Q$, the interval $[[x]_q^\flat,[x]_q^\sharp]$ does not contain $x$ in general. The construction of $q$-rationals, step by step via Farey addition, is similar to that of a Cantor set.
\end{Remark}

Recall from hyperbolic geometry that for two disjoint geodesics $\gamma_1, \gamma_2$, there is a unique third geodesic orthogonal to both of them. This third geodesic realizes the minimal distance between $\gamma_1$ and $\gamma_2$. 
For $x,y\in\mathbb{QP}^1$, we denote by $\gamma([x],[y])$ the geodesic between $[x]$ and $[y]$ (which are disjoint by Proposition \ref{Prop:q-Farey-topo}).

From Proposition \ref{Prop:inversion-sym} and the $\mathrm{PGL}_2(\Z)$-symmetry of $\mathcal{Q}$, we get:
\begin{coro} \label{coro:symmetry_12}
If $d_F(x,y)\in \{1,2\}$, the inversion with respect to $\gamma([x],[y])$ is a symmetry of $\mathcal{Q}$.
\end{coro}

We analyze more in detail the quotient $\mathbb{H}^2_q/\mathrm{PSL}_2(\Z)$, which is the convex core of the deformed modular surface.
It is an orbifold with boundary. The orbifold point $i/\sqrt{q}$ is of order 2, while $\sigma$ is of order 3. The deformed modular surface has recently been described independently by Simon \cite[Section 3.2]{Simon}.

\begin{proposition}\label{Prop:q-modular-surface}
The boundary curve of the convex core of the $q$-deformed modular surface $\mathbb{H}^2_q/\mathrm{PSL}_2(\Z)$ is a geodesic of length $\lvert \ln(q)\rvert$, and the area of the convex core is $\pi/3$.
\end{proposition}
Note that for $q=1$, we recover the modular surface with a cusp. We also see the symmetry between $q$ and $q^{-1}$. If we use $\mathrm{PGL}_2(\Z)$ instead of $\PSL_2(\Z)$, then the length of the boundary geodesic and the area have to be divided by 2, since the fundamental domain for $\mathrm{PGL}_2(\Z)$ is half the one for $\mathrm{PSL}_2(\Z)$.

\begin{proof}
The area is simply given by Gauss--Bonnet, since the boundary is geodesic. For the length of the boundary, we do a direct computation.
The vertices of the fundamental quadrilateral of $\mathrm{PSL}_2(\Z)$ acting on $\H^2_q$ are given by $A=\sigma, B=S_q(\sigma), C=\frac{q+i\sqrt{q^2-q+1}}{q(1-q)}, D=T_q(C)$, see Figure \ref{fig:fundamental_quadrilateral}.
Since $C$ and $D$ lie on the same vertical line, their hyperbolic distance is given by $\lvert\ln(y(C)/y(D))\rvert$, where $y(C)$ denotes the $y$-coordinate of $C$. Since $D=T_q(C)$, we get $y(D)=qy(C)$. Therefore the length of the boundary geodesic is $\lvert\ln(q)\rvert$.
\end{proof}

\begin{figure}
\begin{tikzpicture}[scale=1.6]
\def\q{0.45}
\fill[gray!30] (1.8182,0) rectangle (3,3.505); 
\draw[darkgray, line width=1.5pt] (1.8182,0) -- (1.8182,3.5);
\draw[color=purple, line width=1.5pt] ({1/(1-\q)},{(sqrt(\q*\q-\q+1))/(\q*(1-\q))}) -- ({1/(1-\q)},{(sqrt(\q*\q-\q+1))/(1-\q)});

[line width=0.4mm, mygreen]

\node[font=\large] at (2.5,1) {$\infty$}; 
\draw[->] (-3,0) -- (3.5,0) node[right] {$x$}; 
\draw[->] (0,0) -- (0,3.1) node[above] {$y$};
\draw[->] (0,0) -- (0,3.1) node[above] {$y$};

\draw[purple, line width=0.4mm] ({-1/(2*(0.45))},{sqrt(3)/(2*(0.45))}) arc (86.7:26.6:1.928);


\draw[purple, line width=0.4mm]({-1/(2*(0.45))},{sqrt(3)/(2*(0.45))}) arc (147:90:3.505);

\draw[purple, line width=0.4mm]({1/2}, {sqrt(3)/2}) arc (147:90:1.578);

\coordinate (A) at (0.5,{sqrt(3)/2});
\fill (A) circle (1pt) node[right] {$A$};

\coordinate (Q) at (1.818,0);
\fill (Q) circle (1pt) node[above right]{$\frac{1}{1-q}$};

\draw[gray, dotted, line width=1.5pt] (1.8182,0) -- ({-1/(2*\q)},{sqrt(3)/(2*\q)});

\coordinate (B) at ({-1/(2*\q)},{sqrt(3)/(2*\q)});
\fill (B) circle (1pt) node[above left] {$B$};

\coordinate (C) at ({1/(1-\q)},{(sqrt(\q*\q-\q+1))/(\q*(1-\q))}); \fill (C) circle (1pt) node[above right] {$C$};

\coordinate (D) at ({1/(1-\q)},{(sqrt(\q*\q-\q+1))/(1-\q)}); \fill (D) circle (1pt) node[above right] {$D$};
\end{tikzpicture}
\caption{Fundamental quadrilateral for the action of $\PSL_2(\Z)$ on $\mathbb{H}_q^2$. The angles at $C$ and $D$ are equal to $\frac{\pi}{2}$. The point $B$ is obtained as the intersection of the circle centered at $\frac{1}{1-q}$  passing by $C$ and of the (Euclidean) line connecting $\frac{1}{1-q}$ and $A$.}\label{fig:fundamental_quadrilateral}
\end{figure}

Etingof \cite[Prop. 4.6]{Etingof} defines the jump of a rational $x\in\mathbb{Q}$ as the length 
\begin{equation}\label{Eq:diam-def}
\ell_q(x):= \lvert [x]^\sharp _q-[x]^\flat_q\rvert.
\end{equation}
Using analytic considerations, Etingof computes the sum of jumps of all rationals for $q$ in some open subset of $\C$ (Section 6 of his paper). From our geometric perspective, the jump corresponds to the diameter of the $q$-disk $[x]$. We get a similar result for $q\in \R_{>0}$, which is both more restrictive (we do not include complex $q$) and more general (since Etingof considers only $\mathrm{Re}(q)<1$):

\begin{coro}\label{Coro:Etingof-gap-formula}
Let $q\in \R_{>0}$. If $q\in (0,1)$, then for any $x, y\in \mathbb{QP}^1$ with $x<y$, we have
$$\sum_{x<z<y, z\in\Q} \ell_q(z)=[y]^\flat_q -[x]^\sharp_q.$$
If $q=1$, then all jumps are zero. If $q>1$, then 
$$\sum_{x<z<y, z\in\Q} \ell_q(z)=[y]^\sharp_q -[x]^\flat_q.$$
\end{coro}

In particular, for $x=1$, $y=\infty$, and $q\in (0,1)$, this sum equals $[\infty]^\flat_q-[1]^\sharp _q=\tfrac{1}{1-q}-1=\frac{q}{1-q}$, in concordance with \cite[Prop. 6.1]{Etingof}. 
We are grateful to Pierre-Louis Blayac who suggested the ergodicity arguments in the following proof.

\begin{proof}
We can reduce to $q\in (0,1)$ using the transition map \ref{Prop:duality}. We then only need to prove that the complement $\Lambda$ to the union 
$$
\bigcup_{x \in \mathbb{QP}^1} \left(
[x]_q^{\flat}, [x]_q^{\sharp}
\right)
$$
is of Lebesgue measure zero in $\mathbb{RP}^1$. Since the intervals above are ordered, we then get the result.

From the hyperbolic geometry viewpoint, the complement $\Lambda$ is the limit set of the fundamental group $\Gamma$ of the deformed modular surface ($\Gamma\cong \PSL_2(\Z)$), which is convex cocompact.
Sullivan \cite{Sullivan1979} then shows that if $\Lambda$ was of positive Lebesgue measure, then the action of $\Gamma$ on the boundary of $\H^2$ would be ergodic with respect to the Lebesgue measure. Hence $\Lambda$ would be of full measure, but its complement has strictly positive measure, a contradiction.

Another argument is to use the ergodicity of the geodesic flow on a closed hyperbolic orbifold. First, cut the deformed modular surface along the unique geodesic along its funnel and double it. Passing to a finite cover, we get a closed surface. Ergodicity of its geodesic flow implies that almost all geodesics in the original modular surface go into the funnel. Hence the preimage of the geodesic around the funnel in the boundary $\partial\H^2$ has full Lebesgue measure.
\end{proof}

For rational numbers, the jump can be computed algebraically. For that, we need the following:

\begin{definition}\label{defi:epsilon}
To a rational number $x=\tfrac{a}{b}$ with even-length continued fraction expansion $\frac{a}{b} = [\alpha_1,\alpha_2,\cdots,\alpha_{2s}]$, we associate the integer $\varepsilon(\tfrac{a}{b})$ given by
\begin{equation}
\varepsilon(x)=\begin{cases}
\abs{\alpha_1} + \alpha_2 -2 \text{ if } x \in \Z_{\leq 0},\\
\abs{\alpha_1}+\alpha_2+...+\alpha_{2s}-1 \text{ else.}
\end{cases}
\end{equation}
\end{definition}

\begin{proposition}\label{prop:qradius}
Let $x = a/b$ be a rational number, and denote by $A^{\square}/B^{\square}$ its left and right $q$-deformations, for $\square\in \{\sharp,\flat\}$. Then  
\begin{equation}\label{Eq:q-diam}
\ell_q(x) = \abs{1-q}\frac{q^{\varepsilon(x)}}{B^{\sharp}(q)B^{\flat}(q)}.
\end{equation}
\end{proposition}

\begin{proof} Denote $L = TST$. Suppose $x \notin \Z_{\leq 0}$, and let $M = T^{\alpha_1}L^{\alpha_2}\cdots T^{\alpha_{2s-1}}L^{\alpha_{2s}}$. Then, by Proposition 4.3 of \cite{MGO-2020} and Proposition 2.11 of \cite{BBL},
$$
    \det\left(M_q \begin{pmatrix}
    1 & 1 \\
    0 & 1-q\\
    \end{pmatrix}\right) = q^{2\min(0,\alpha_1)}\det \begin{pmatrix}
    qA^{\sharp} & A^{\flat}\\
    qB^{\sharp} & B^{\flat}\\
    \end{pmatrix} = q^{2\min(0,\alpha_1) + 1}(A^{\sharp}B^{\flat} - A^{\flat}B^{\sharp}).
$$
\noindent But $\det(M_q) = q^{\alpha_1+\alpha_2+\cdots+\alpha_{2s}}$, and hence the formula for $\ell_q(x)$. The proof is analogue for the case $x\in \Z_{\leq 0}$.
\end{proof}

As a consequence of this explicit formula for the diameter, we see that for $q$ close to 1, the Taylor expansion of Equation \eqref{Eq:q-diam} gives
\begin{equation}\label{Eq:diam-near-1}
\ell_q\left(\frac{r}{s}\right)
 \underset{q\to 1}{\sim} \frac{\lvert 1-q\rvert}{s^2}.
\end{equation}
In other words: the radius of the disk associated to $\tfrac{r}{s}\in\Q$ is $\lvert 1-q\rvert$ times the radius of the Ford circle. The important difference is that the boundaries of $q$-disks are orthogonal to the real axis, while Ford circles are tangent to it.

\section{Deformed Farey determinants and operations}\label{Sec:q-Farey}

In this section, we define a $q$-deformation of the Farey determinant, and extend the $q$-Farey addition from \cite[Section 2.5]{MGO-2020}, where it is defined only between pairs of Farey determinant 1, to a more general case.

\subsection{$q$-Farey determinant}

To define a $q$-deformed version of the Farey determinant, there are four possibilities, since there are two versions, left and right, for a $q$-rational:

\begin{definition}\label{Def:q-Farey-det}
For $(\tfrac{a}{b},\tfrac{c}{d})$ any pair of rationals with $q$-deformation $\left(\tfrac{A^\square}{B^\square},\tfrac{C^\triangle}{D^\triangle}\right)$, with $\square, \triangle\in\{\sharp,\flat\}$, the \emph{deformed Farey determinant} is
$$d_F^{\square\triangle}:= A^\square D^\triangle-B^\square C^\triangle $$
if $(\square,\triangle)\neq (\flat ,\flat)$.
Else, 
$$d_F^{\flat\flat}:=\frac{A^\flat D^\flat-B^\flat C^\flat}{q^2-q+1}.$$
\end{definition}

To analyze the properties of these four $q$-Farey determinants, we need the following lemma about special values of $q$-rationals at the point $q=\sigma$. These properties complete previous works on special values: \cite[Section 1.4]{MGO-2020} for $q=-1$, \cite[Section 7]{Kogiso} for $q=\sigma^2$ and \cite{Leclere_these} for roots of unity of order at most 5.

\begin{lemma}\label{Lem:special-values}
Denote by $\mathbb{U}_6$ the set of 6-th roots of unity. For all $\frac{a}{b}\in\mathbb{Q}$, we have the following:
\begin{enumerate}
    \item $B^\sharp (\sigma)+(\sigma-1)A^\sharp (\sigma) \in \mathbb{U}_6$,
    \item $B^\flat(\sigma)+(\sigma-1)A^\flat(\sigma)=0$,
    \item $A^\flat(\sigma)$ and $B^\flat(\sigma)$ are in $\mathbb{U}_6$.
\end{enumerate}
\end{lemma}
\begin{proof}
All the identities hold for $\tfrac{a}{b}=0$. It is then sufficient to show that the properties are invariant under $T$ and $S$, since the modular group acts transitively on $\mathbb{QP}^1$.\\
\noindent For $\square\in\{\sharp ,\flat\}$, the expression $B^\square(q)+(q-1)A^\square(q)$ transforms under $T$ into $$B^\square(q)+(q-1)(qA^\square +B^\square) = q(B^\square(q)+(q-1)A^\square(q))$$
modulo a power of $q$. Hence, evaluating at $q=\sigma$ does not change the fact of being zero or being in $\mathbb{U}_6$.\\ 
\noindent Similarly, the transformation under $S$ gives $qA^\square(q)-(q-1)B^\square(q)$. Evaluating at $q=\sigma$ and using $\sigma^2-\sigma+1=0$, this gives $-\sigma^2(B^\square(\sigma)+(\sigma-1)A^\square(\sigma))$. Again, this preserves the fact of being zero or in $\mathbb{U}_6$. This finishes the proof of (1) and (2).\\
\noindent For (3), note that $A^\flat(\sigma)$ transforms under $T$ into
$$\sigma A^\flat(\sigma)+B^\flat(\sigma)=A^\flat(\sigma)$$
using the identity from (2). Under $S$ it transforms into $-B^\flat(\sigma)=\sigma^2 A^\flat(\sigma)$, where we used again (2). Therefore, $A^\flat(\sigma)\in\mathbb{U}_6$. Finally, since $B^\flat(\sigma)=-\sigma^2 A^\flat(\sigma)$, we also get $B^\flat(\sigma)\in\mathbb{U}_6$.
\end{proof}

\begin{notation}
We will write $A \equiv_{q} B$ to mean that there is an integer $k\in\Z$ and $\nu\in\{-1,1\}$ such that $A(q) = \nu q^k B(q)$.
\end{notation}

\begin{proposition}\label{Prop:invariance-q-Farey-det}
The four $q$-Farey determinants are invariant under the $\PSL_2(\Z)$-action, up to some monomial factor in $q$ and a sign.
\end{proposition}

\begin{proof}
It is sufficient to compute the change of $d_F^{\triangle\square}$ under $T$ and $S$ (modulo some power of $q$):
\begin{align*}
A^\square D^\triangle-B^\square C^\triangle &\xmapsto{T} (qA^\square+B^\square)D^\triangle-B^\square(qC^\triangle+D^\triangle) = q(A^\square D^\triangle-B^\square C^\triangle)    \\
A^\square D^\triangle-B^\square C^\triangle &\xmapsto{S} -qB^\square C^\triangle + qA^\square D^\triangle = q(A^\square D^\triangle-B^\square C^\triangle).
\end{align*}
\end{proof}

The four $q$-Farey determinants behave nicely under the duality $q\mapsto q^{-1}$, showing that there are essentially only two $q$-Farey determinants:
\begin{theorem}\label{Thm:link-left-right-dF}
For any two rational numbers, $d_F^{\flat\flat}$ is a polynomial and we have 
\begin{align*}
    d_F^{\flat \sharp }(q) & \equiv_{q}  d_F^{\sharp  \flat}(q^{-1}); \\
    d_F^{\flat\flat}(q) & \equiv_q d_F^{\sharp \sharp }(q^{-1}), 
\end{align*}
\end{theorem}

The proof idea is to use the invariance under $\PSL_2(\Z)$ to put one fraction to $\infty$. Then, the transition map from Proposition \ref{Prop:duality} concludes.

\begin{proof}
We can use the $\mathrm{PSL}_2(\mathbb{Z})$-symmetry from Proposition \ref{Prop:invariance-q-Farey-det} to send the first rational to $\infty$. The second one is then sent to some rational $\tfrac{r}{s}$.\\

\noindent To prove the first identity, we have to compare $d_F^{\flat\sharp }(q) = S^\sharp (q)-(1-q)R^\sharp (q)$ to $d_F^{\sharp \flat}(q^{-1})=S^\flat(q^{-1})$, since $[\infty]^\sharp =\tfrac{1}{0}$ and $[\infty]^\flat=\tfrac{1}{1-q}$.\\

\noindent The transition map of $q$-rationals gives
\begin{equation}\label{Eq:duality-bemol}
\frac{R^\flat(q^{-1})}{S^\flat(q^{-1})}=\left[\frac{r}{s}\right]^\flat_{q^{-1}}=g_q\left(\left[\frac{r}{s}\right]^\sharp _q\right)=\frac{(1-q)S^\sharp (q)+qR^\sharp (q)}{S^\sharp (q)+(q-1)R^\sharp (q)}.    
\end{equation}
In order to identify the denominators of the two sides, we have to check that the fraction on the right hand side is reduced:
\begin{align*}
\gcd((1-q)S^\sharp +qR^\sharp ,S^\sharp +(q-1)R^\sharp ) &= \gcd((q^2-q+1)R^\sharp ,S^\sharp +(q-1)R^\sharp ) \\
&= \gcd(R^\sharp ,S^\sharp +(q-1)R^\sharp )= 1,
\end{align*}
where we used an elementary operation in the first line, then by point (1) of Lemma \ref{Lem:special-values}, the fact that $\gcd(q^2-q+1,S^\sharp +(q-1)R^\sharp )=1$ ($\sigma$ is a root of $q^2-q+1$), and finally that $\gcd(R^\sharp ,S^\sharp )=1$. This proves
$$d_F^{\flat \sharp }(q) = S^\sharp (q)+(q-1)R^\sharp (q) = q^{\alpha}S^{\flat}(q^{-1}) \equiv_q d_F^{\sharp\flat}(q^{-1}).$$

\medskip
\noindent For the second case, we notice that the polynomial $S^\flat(q)+(q-1)R^\flat(q)$ evaluates to 0 at $q=\sigma$, by Lemma \ref{Lem:special-values} point (2), hence is divisible by $q^2-q+1=(q-\sigma)(q-\bar{\sigma})$.
We then have to compare $d_F^{\sharp \sharp }(q^{-1})=S^\sharp (q^{-1})$ to $(q^2-q+1)d_F^{\flat\flat}(q)=S^\flat(q)+(q-1)R^\flat(q)$. \\
\noindent We use the transition map again to get
\begin{equation}\label{Eq:duality-diese}
\frac{R^\sharp (q^{-1})}{S^\sharp (q^{-1})}=\left[\frac{r}{s}\right]^\sharp _{q^{-1}}=g_q\left(\left[\frac{r}{s}\right]^\flat_q\right)=\frac{(1-q)S^\flat(q)+qR^\flat(q)}{S^\flat(q)+(q-1)R^\flat(q)}.
\end{equation}
The fraction on the right hand side is not reduced this time since
\begin{align*}
\gcd((1-q)S^\flat+qR^\flat,S^\flat+(q-1)R^\flat) &= \gcd((q^2-q+1)R^\flat,S^\flat+(q-1)R^\flat) \\
&= (q^2-q+1)\gcd\left(R^\flat,\frac{S^\flat+(q-1)R^\flat}{q^2-q+1}\right) \\
&= (q^2-q+1)\gcd(R^\flat,S^\flat+(q-1)R^\flat) \\
&= q^2-q+1,
\end{align*}
where we used point (2) of Lemma \ref{Lem:special-values} in the second line. For the third line, point (3) of Lemma \ref{Lem:special-values} implies $\gcd(R^\flat,q^2-q+1)=1$. Therefore,
$$ d_F^{\sharp \sharp }(q^{-1})= S^\sharp (q^{-1}) \equiv_q \frac{S^\flat(q)+(q-1)R^\flat(q)}{q^2-q+1} =d_F^{\flat\flat}(q).$$
\end{proof}

As a consequence, we get a positivity property of all $q$-Farey determinants:
\begin{proposition}\label{Prop:positivity-q-dF}
For $\tfrac{a}{b}> \tfrac{c}{d}$, we have $d_F^{\square\triangle}\in\mathbb{N}[q]$ for $\square,\triangle\in\{\sharp ,\flat\}$.
\end{proposition}

\begin{proof}
We show that for any pair of distinct rationals, the $q$-Farey determinants have constant sign coefficients (this property is invariant under $T$ and $S$). The proposition then follows, since the sign can be determined by evaluating at $q=1$.\\
~\\
\noindent We have already seen that the Farey determinants do not change under $T$ and $S$ (modulo a power of $q$).
If $\triangle = \sharp $, we can send $\tfrac{c}{d}$ to $\tfrac{0}{1}$, then $d_F^{\square \sharp }=A^\square$, which has constant sign coefficients by Proposition \ref{Prop:cst-sign-coeffs}. This proves the proposition for $d_F^{\sharp \sharp }$ and $d_F^{\flat\sharp }$.
The remaining two cases follow from the relations between the $q$-Farey determinants given in Theorem \ref{Thm:link-left-right-dF}.
\end{proof}

Note that even in the case $(\sharp \sharp )$, this provides a new proof of the positivity of the polynomial $A^\sharp D^\sharp -B^\sharp C^\sharp $. The original proof in \cite[Section 4.7]{MGO-2020} uses the Farey triangulation.

As a consequence, that the order on $q$-rationals inherits from the natural order on rationals : 
\begin{coro}[Order preservation]\label{Coro:well-orderedness}
For $\tfrac{c}{d}< \tfrac{a}{b}$ and $q\in(0,1)$, we have
$$\left[\frac{c}{d}\right]^{\flat}_q<\left[\frac{c}{d}\right]^{\sharp}_q<\left[\frac{a}{b}\right]^{\flat}_q<\left[\frac{a}{b}\right]^{\sharp}_q.$$
\end{coro}

Another consequence arises by combining Theorem \ref{Thm:link-left-right-dF} with results about palindromic denominators in $q$-rationals from \cite{Kogiso, Ren}.

\begin{coro}\label{Prop:id-flat-sharp}
Consider a rational number $\frac{k}{n} \in \Q$.\\
$(i)$ If $n$ divides $k^2+1$, then $N^\sharp +(q-1)K^\sharp  \equiv_q N^{\flat}$.\\
$(ii)$ If $n$ divides $k^2-1$, then $N^\flat+(q-1)K^\flat\equiv_q (q^2-q+1)N^\sharp $.
\end{coro}

\begin{proof}
From \cite[Theorem 1.2]{Ren} (see Theorem \ref{thm:palindromicity}), we know that $n\mid k^2+1$ implies that $N^\flat$ is palindromic. Hence by the denominator of Equation \eqref{Eq:duality-bemol}, we see that $$N^\flat(q)=q^{\mathrm{deg}(N^\flat)}N^\flat(q^{-1})\equiv_q (N^\sharp +(q-1)K^\sharp ).$$
\noindent Similarly, from \cite[Corollary 3.8]{Kogiso} (see Theorem \ref{thm:equal-denominators}), we know that $n\mid k^2-1$ implies that $N^\sharp $ is palindromic. Hence the denominator of Equation \eqref{Eq:duality-diese} concludes.
\end{proof}

\subsection{$q$-Farey operations}

The $q$-deformed Farey operations are an equivalent way to define $q$-rationals. They were introduced in \cite[Section 2.5]{MGO-2020}. Let us quickly recall them here.

\begin{proposition}[\cite{MGO-2020,ren_2024}]
Let $(\tfrac{a}{b},\tfrac{c}{d})$ be a pair of rationals with Farey determinant 1, with $\frac{a}{b} < \frac{c}{d}$. Put $\tfrac{r}{s} = \tfrac{a}{b}\oplus_F \tfrac{c}{d}=\tfrac{a+c}{b+d}$. Then
$$R^\sharp = A^\sharp + q^\alpha C^\sharp \;\;\text{ and }\;\; S^\sharp = B^\sharp+q^\alpha D^\sharp,$$
where 
$$\alpha= \left \{ \begin{array}{cl}
\varepsilon(\tfrac{a}{b})-\varepsilon(\tfrac{c}{d})+1 & \text{ if  } \varepsilon(\tfrac{a}{b})\geq \varepsilon(\tfrac{c}{d}) \\
1 &\text{ else  },
\end{array} \right.$$
where $\varepsilon$ is defined in Definition \eqref{defi:epsilon}. A similar formula holds for left $q$-rationals.
\end{proposition}

Pairs of Farey determinant 1 correspond to points in the Farey graph $\mathcal{F}$ which are at graph distance 1 (they have an edge between them). We will generalize the above proposition to pairs of rationals of Farey graph distance 2.

\begin{Remark}
Note that the Farey graph distance is not equal in general to the Farey determinant. They coincide for instance when they both are equal to $1$.
\end{Remark}

Pairs of rationals of Farey graph distance 2 arise naturally in any Farey sequence, i.e. a sequence of all reduced fractions in a given interval, with denominator smaller or equal to a fixed number $m$, and ordered increasingly, see Chapter III of \cite{Hardy_Wright}. For a Farey sequence $(f_n)$, we have $d_F(f_n,f_{n+1}) = 1$ for all $n$, and also 
\begin{equation}\label{Eq:Farey-sequences}
    f_{n}=f_{n-1}\oplus_F f_{n+1}.
\end{equation}
Hence, the Farey graph distance between $f_n$ and $f_{n+2}$ is at most 2.

\begin{example}
The Farey sequence in the intveral $[0,\tfrac{1}{2}]$ for $m=6$ is given by
$$\left\{\frac{0}{1},\frac{1}{6}, \frac{1}{5}, \frac{1}{4}, \frac{1}{3}, \frac{2}{5}, \frac{1}{2}\right\}.$$
For a more complicated example, consider the interval $[\tfrac{151}{227},\tfrac{153}{229}]$ with $m=234$, which gives 
$\{\tfrac{151}{227},\tfrac{153}{230}, \tfrac{155}{233},\tfrac{2}{3}, \tfrac{155}{232}, \tfrac{153}{229}\}.$
In both examples, the relation \eqref{Eq:Farey-sequences} holds.
\end{example}

There is a nice characterization of pairs with Farey graph distance at most 2:
\begin{proposition}
Consider a pair of rationals $(\tfrac{a}{b},\tfrac{c}{d})$ and denote by $d_F$ their Farey determinant. The pair is of Farey graph distance at most 2 inside $\mathcal{F}$ if and only if $\mathrm{gcd}(a+c,b+d) = d_F$ or $\mathrm{gcd}(a-c,b-d) = d_F$.
\end{proposition}
In particular, the reduced expression of the Farey sum or the Farey difference of a pair of Farey graph distance 2 is explicit.

\begin{proof}
Assume that $\mathrm{gcd}(a+c,b+d) = d_F$. We then compute the Farey determinant between $\tfrac{a}{b}$ and $\tfrac{a+c}{b+d}$:
$$d_F\left(\frac{a}{b}, \frac{a+c}{b+d}\right) = \frac{a(b+d)-b(a+c)}{d_F} = 1.$$
The same computation shows that $d_F(\tfrac{c}{d}, \tfrac{a+c}{b+d})=1$. Hence their graph distance is at most 2. A similar argument works for the assumption that $\mathrm{gcd}(a-c,b-d) = d_F$.

Conversely, assume that the graph distance is at most 2. If the distance is 1, there is nothing to prove since $d_F=1$ and the Farey operations give reduced fractions. If the distance is 2, we know that there is a rational number $\tfrac{x}{y}$ such that $d_F(\tfrac{x}{y}, \tfrac{a}{b})=1=d_F(\tfrac{x}{y},\tfrac{c}{d})$. This leads to a system in $(x,y)$ with solution $x=\frac{\pm a\pm c}{d_F}$ and $y=\frac{\pm b \pm d}{d_F}$. Since $\tfrac{x}{y}$ is reduced, we get that either $\mathrm{gcd}(a+c,b+d)=d_F$ or $\mathrm{gcd}(a-c,b-d)=d_F$.
\end{proof}

\begin{theorem}\label{Thm:q-Farey-operations}
Let $(\tfrac{a}{b}, \tfrac{c}{d})$ be a pair of rationals, and denote by $d_F$ their Farey determinant. Assume that $\mathrm{gcd}(a+c,b+d)=d_F$. Then there are integers $\alpha, \beta$ such that 
$$\left[\frac{a}{b}\oplus_F \frac{c}{d}\right]_q^\sharp = \frac{(q^\alpha A^\sharp+q^\beta C^\sharp)/d_F^{\sharp\sharp}}{(q^\alpha B^\sharp+q^\beta D^\sharp)/d_F^{\sharp\sharp}},$$
where the fraction on the right hand side is reduced. 

\noindent Similarly, if $\mathrm{gcd}(a-c,b-d)=d_F$, there are integers $\gamma, \delta$ such that
$$\left[\frac{a}{b}\ominus_F \frac{c}{d}\right]_q^\sharp = \frac{(q^\gamma A^\sharp-q^\delta C^\sharp)/d_F^{\sharp\sharp}}{(q^\gamma B^\sharp-q^\delta D^\sharp)/d_F^{\sharp\sharp}}.$$
\end{theorem}

The idea of the proof is to show that the property is invariant under the $\PSL_2(\Z)$-action, which allows to reduce to the pair $(\tfrac{1}{0}, \tfrac{-1}{n})$. Finding an explicit formula for the exponents of $q$ when $d_F \geq 2$ seems challenging.

\begin{proof}
Consider first the Farey sum. Put $\tfrac{r}{s}=\tfrac{a}{b}\oplus_F \tfrac{c}{d}$, i.e. $r=(a+c)/d_F$ and $s=(b+d)/d_F$. We have to show that there are integers $\alpha$ and $\beta$ such that
\begin{align}\label{Eq:alpha-beta}
 d_F^{\sharp\sharp}R^\sharp &= q^\alpha A^\sharp+q^\beta C^\sharp ; \\ \nonumber
 d_F^{\sharp\sharp}S^\sharp &= q^\alpha B^\sharp+q^\beta D^\sharp.
\end{align}
This implies in particular that $\mathrm{gcd}(q^\alpha A^\sharp+q^\beta C^\sharp,q^\alpha B^\sharp+q^\beta D^\sharp)=d_F^{\sharp\sharp}$. A similar statement has to be shown for the Farey difference.

Let us show that the equations \eqref{Eq:alpha-beta} are invariant under the deformed $\PSL_2(\Z)$-action. The crucial fact is that the classical Farey operations are equivariant with respect to any M\"obius transformation $M\in\PSL_2(\Z)$, see Proposition \ref{Prop:invariance-Farey}. Hence the action of $S$ or $T$ has the same effect on $\tfrac{a}{b}, \tfrac{c}{d}$ and $\tfrac{r}{s}$.

By Proposition \ref{Prop:invariance-q-Farey-det}, we know that the $q$-Farey determinant is invariant, up to some power of $q$ which can be absorbed in the integers $\alpha$ and $\beta$.

The transformation $S_q$ sends any $q$-rational $(X^\sharp, Y^\sharp)$ to $(-Y^\sharp, qX^\sharp)$ or $(-Y^\sharp/q, X^\sharp)$ (depending on whether $q$ divides the denominator $Y^\sharp$). Applying this to $(A^\sharp, B^\sharp)$, $(C^\sharp, D^\sharp)$ and $(R^\sharp, S^\sharp)$, we see that Equations \eqref{Eq:alpha-beta} can be satisfied after the action of $S_q$, changing $\alpha$ or $\beta$ by 1 if necessary.

Similarly, the transformation $T_q$ sends any $q$-rational $(X^\sharp, Y^\sharp)$ to $(qX^\sharp+Y^\sharp, Y^\sharp)$ or $(X+Y^\sharp/q, Y^\sharp/q)$. Applying this to $(A^\sharp, B^\sharp)$, $(C^\sharp, D^\sharp)$ and $(R^\sharp, S^\sharp)$, we see that Equations \eqref{Eq:alpha-beta} can be satisfied after the action of $T_q$, changing $\alpha$ or $\beta$ by 1 if necessary.

\medskip
Using the $\PSL_2(\Z)$-invariance, we can reduce to the case where $\tfrac{a}{b}=\infty=\tfrac{1}{0}$. Then $\tfrac{c}{d}=\tfrac{k}{n}$ with $n=d_F$. By the assumption on the greatest common divisor, we see that $k\equiv -1 \!\mod n$. Using $T$, we can assume $k=-1$.
Then $\alpha=\beta=0$ satisfy Equations \eqref{Eq:alpha-beta}. Indeed, $[\infty]^\sharp_q = \tfrac{1}{0}$ and $[-\tfrac{1}{n}]^\sharp_q = \frac{-1}{q[n]_q^\sharp}$. Hence $d_F^{\sharp\sharp}=q[n]_q^\sharp$. Finally, $\tfrac{1}{0}\oplus_F\tfrac{-1}{n}=\tfrac{0}{1}$, with deformation $[0]_q^\sharp = \tfrac{0}{1}$.

For the Farey difference, we have $k\equiv 1 \!\mod n$, so we can assume $k=1$. Using $[\tfrac{1}{n}]_q^\sharp=\tfrac{q^{n-1}}{[n]_q}$, we get $d_F^{\sharp\sharp}=[n]_q$. Then we see that $\gamma=n-1$ and $\delta=0$ satisfy the analog of Equations \eqref{Eq:alpha-beta} for the Farey difference.
\end{proof}

\begin{Remark}
There are five more identities, proven similarly to that of Theorem \ref{Thm:q-Farey-operations}. Under the same assumptions, we have
$$\left[\frac{a}{b}\oplus_F \frac{c}{d}\right]_q^\sharp = \frac{(q^\alpha A^\triangle+q^\beta C^\square)/d_F^{\triangle\square}}{(q^\alpha B^\triangle+q^\beta D^\square)/d_F^{\triangle\square}},$$
where $(\triangle,\square) \in \{(\sharp,\sharp), (\sharp,\flat), (\flat, \sharp)\}$. Note that the values  $\alpha, \beta \in \Z$ change depending on the case $\{\triangle, \square\}$ . Similarly,
$$\left[\frac{a}{b}\oplus_F \frac{c}{d}\right]_q^\flat = \frac{(q^\alpha A^\triangle+q^\beta C^\square)/d_F^{\triangle\square}}{(q^\alpha B^\triangle+q^\beta D^\square)/d_F^{\triangle\square}},$$
where $(\triangle,\square) \in \{(\sharp,\sharp), (\flat, \flat)\}$.

\noindent Finally, under the assumption $\mathrm{gcd}(a-c,b-d)=d_F$, we get the same five identities using the Farey difference instead of the Farey sum.
\end{Remark}

\section{Classical Springborn operations}\label{Sec:Springborn-classic}

In this section, we introduce a quadratic version of the Farey operations, which we call the Springborn operations. The Springborn sum has been used by Springborn in \cite{Springborn}. Our motivation comes from the geometric picture, seeing $q$-rationals as disks.

\subsection{Motivation: homothetic centers}\label{Sec:homothetic-centers}

Consider two circles $C_1$ and $C_2$. Denote by $M_1$ and $M_2$ their centers and by $r_1$ and $r_2$ their radii.

\begin{definition}\label{Def:i-and-e}
The \emph{inner homothety center} of $C_1$ and $C_2$, denoted by $i(C_1,C_2)$, is the fixed point of the unique homothety with negative factor exchanging $C_1$ and $C_2$. Similarly, the \emph{outer homothety center}, denoted by $e(C_1,C_2)$, is the fixed point of the unique homothety with positive factor exchanging $C_1$ and $C_2$.
\end{definition}

In the case when $C_1$ and $C_2$ have disjoint interiors, $i(C_1,C_2)$ (resp. $e(C_1,C_2)$) is the intersection of the inner (resp. outer) common tangents of $C_1$ and $C_2$, see Figure \ref{fig:homothetic_center}.

\begin{figure}[H]
    \centering
    \begin{tikzpicture}[scale=0.6]
    \draw (-1,0) -- (16,0);
    \draw[line width=1pt] (2,0) circle (2);
    \draw[line width=1pt] (8,0) circle (1);
    \node at(2,1) {$C_1$};
    \node at(8,0.5) {$C_2$};
    \node[color=orange] at(14,0) {$\bullet$};
    \node[color=purple] at(6,0) {$\bullet$};
    \node at(2.333,1.972) (O1) {\rotatebox{45}{~}};
    \node at(3,1.732) (I1) {};
    \node at(8.167,0.986) (O2) {\rotatebox{45}{~}};
    \node at(7.5,0.866) (I2) {\rotatebox{50}{~}};
    \draw[color=orange] (-1,2.535) -- (16,-0.338);
    \draw[color=orange] (-1,-2.535) -- (16,0.338);
    \draw[color=purple] (-1,-4.041) -- (10,2.31);
    \draw[color=purple] (-1,4.041) -- (10,-2.31);
    \end{tikzpicture}
    \caption{Inner (in purple) and outer (in orange) homothety centers.}
    \label{fig:homothetic_center}
\end{figure}

\begin{observation}\label{Conj-real-line}
For many pairs $\left(\frac{a}{b}, \frac{c}{d}\right) \in \Q^2$, we find 
$$
    i\left(\left[\frac{a}{b}\right],\left[\frac{c}{d}\right]\right) = \left[\frac{ab+cd}{b^2+d^2}\right]_q^\sharp  \text{ and }
   e\left(\left[\frac{a}{b}\right],\left[\frac{c}{d}\right]\right) = \left[\frac{ab-cd}{b^2-d^2}\right]_q^\flat \, .
$$
\end{observation}

One of the main result of our work is to prove this observation for pairs satisfying an arithmetic condition, that we call regular pairs, see Theorem \ref{Thm:main}.

\begin{Remark}
The formula $\frac{ab+cd}{b^2+d^2}$ first appears in the study of Ford circles \cite{Ford}. It seems that the first time it was used as an iterative operation was only recently by Springborn \cite{Springborn} in the context of Diophantine approximation of rationals.
Springborn defines the notion of Markov fractions, which can be obtained as iterations of this formula on the pair $(\frac{0}{1},\frac{1}{1})$. 
Then, this formula was explored in the subsequent work of Veselov \cite[Equation (1)]{Veselov} on Markov fractions. We study $q$-deformed Markov fractions in Section \ref{Sec:Markov}.
\end{Remark}

\subsection{Springborn operations and regular pairs}

Motivated by Observation \ref{Conj-real-line}, we define the Springborn sum and difference as follows:

\begin{definition}
Let $\frac{a}{b}$ and $\frac{c}{d}$ be two rationals in reduced form (or equal to $\infty=\tfrac{1}{0}$). We define the \emph{Springborn sum} and \emph{Springborn difference} respectively by 
\[
\frac{a}{b}\oplus_S \frac{c}{d} = \frac{ab+cd}{b^2+d^2} \;\;\text{ and }\;\; \frac{a}{b}\ominus_S \frac{c}{d} = \frac{ab-cd}{b^2-d^2}.
\]
\end{definition}

Note that these expressions are in general not reduced. 

\begin{Remark}
The Springborn operations are \emph{not} equivariant with respect to M\"obius transformations.
\end{Remark}

We define now a family of pairs for which the reduced form of Springborn operations is easy to compute. 

\begin{definition}\label{Def:reg}
A pair $(\tfrac{a}{b},\tfrac{c}{d})\in\mathbb{Q}^2$ with Farey determinant $d_F=\abs{ad-bc}$ is said to be \emph{inner regular} if $$\gcd(ab+cd,b^2+d^2,a^2+c^2)=d_F.$$
Similarly, the pair is \emph{outer regular} if $\gcd(ab-cd,b^2-d^2,a^2-c^2)=d_F.$
\end{definition}

The following lemma will be useful to check regularity.
\begin{lemma}\label{Lemma:simplify-Springborn-cond}
For two rational numbers $\tfrac{a}{b}, \tfrac{c}{d}$, the quantity $\gcd(a^2+c^2, b^2+d^2, ab+cd)$ divides $d_F$. If moreover $\mathrm{gcd}(b,d) = 1$, then $(\tfrac{a}{b},\tfrac{c}{d})$ is inner regular if and only if $d_F \mid b^2 + d^2.$ A similar statement holds for the outer regular case.
\end{lemma}

\begin{proof} We have the following identities:
\begin{equation}
\label{eq:technicalrelations}
\begin{cases}
b(ab+cd) = a(b^2+d^2) - d (ad-bc)\\
a(ab+cd) = b(a^2+c^2) + c (ad-bc),\\
(a^2+c^2)(b^2+d^2) = (ab+cd)^2 + d_F^2.
\end{cases}
\end{equation}
The third one implies the first statement of the lemma.

Now suppose $\gcd(b,d) = 1$. We only have to prove that $d_F\mid b^2+d^2$ implies inner regularity. By the first relation in \eqref{eq:technicalrelations}, we see that $d_F$ divides $b(ab+cd)$, but $\gcd(b,d) = 1$ so $d_F$ and $b$ are coprime, thus $d_F$ divides $\gcd(b^2+d^2,ab+cd)$. By the second relation in \eqref{eq:technicalrelations}, we see that $d_F$ divides $a^2+c^2$. We then get $\mathrm{gcd}(a^2+c^2,b^2+d^2,ab+cd)=d_F$, hence the inner regularity condition is satisfied.
\end{proof}

Let us now see some special cases :

\begin{example}
When $d_F=1$, the pair is both inner and outer regular by the last part of Lemma \ref{Lemma:simplify-Springborn-cond} ($d_F=1$ implies $\mathrm{gcd}(b,d)=1$).
\end{example}

\begin{example}
In the special case $d_F=2$, the pair is also inner and outer regular. Indeed, if $\gcd(b,d)=1$, then again by Lemma \ref{Lemma:simplify-Springborn-cond} it is sufficient to show that $2\mid b^2+d^2$. If $b$ is even and $d$ is odd, then $a$ is also odd, which contradicts $d_F=ad-bc=2$. Hence both $b$ and $d$ are odd, so $2\mid b^2+d^2$. Now if $\gcd(b,d)>1$, then $\gcd(b,d)=2$ and $a,c$ are both odd. Put $b=2b'$ and $d=2d'$. It is then easy to see that $ab'+cd'$ is odd, hence $\gcd(ab+cd,b^2+d^2)/2=\gcd(ab'+cd',b'^2+d'^2)=1$, where the last equality comes from Lemma \ref{Lemma:simplify-Springborn-cond}. Therefore the pair is inner regular. Similar arguments show the same for the outer case.
\end{example}

\begin{proposition}\label{Prop:represents-reg}
The set of inner regular pairs is invariant under the modular group action. A set of representatives is given by $(\tfrac{1}{0},\tfrac{k}{n})$ where $n\mid k^2+1$, and $0\leq k<n$.

Similarly, the set of outer regular pairs is invariant under $\PSL_2(\Z)$ and a set of representatives is given by $(\tfrac{1}{0},\tfrac{k}{n})$ where $n\mid k^2-1$, and $0\leq k<n$.
\end{proposition}
\begin{proof}
Applying $S$ to a pair $(\tfrac{a}{b},\tfrac{c}{d})$ exchanges $a$ with $b$ and $c$ with $d$. Hence it does not change the regularity condition. 
Applying $T$ changes $a$ into $a+b$ and $c$ into $c+d$, but keeps $b$ and $d$ the same. By elementary operations in the greatest common divisor, we see that regularity is preserved under $T$.

By Proposition \ref{Prop:invariance-Farey}, the Farey determinant $d_F$ is invariant under $\PSL_2(\Z)$-action with representative $(\tfrac{1}{0},\tfrac{k}{n})$, with $0\leq k<n$, for a pair of Farey determinant $n$. Inner regularity implies that $\gcd(kn,n^2,1+k^2)=n$. This is equivalent to $n\mid k^2+1$.
Similarly, outer regulartiy implies that $\gcd(kn,n^2,-1+k^2)=n$, which is equivalent to $n\mid k^2-1$.
\end{proof}

\subsection{Characterisation of regular pairs}

Regular pairs can be characterized using involutions, which will make the link to the symmetries of the Farey tesselation.

\begin{theorem}\label{Thm:charact-regularity}
A pair $(\tfrac{a}{b},\tfrac{c}{d})\in\Q^2$ is inner (resp. outer) regular if and only if there exists an orientation-preserving (resp. orientation-reversing) involution in $\PGL_2(\Z)$ exchanging $\tfrac{a}{b}$ with $\tfrac{c}{d}$.
\end{theorem}

The main ingredient of the proof is a description of horocycles in the hyperbolic plane $\H^2$ via points in the plane. It is well-known that horocycles correspond to non-zero vectors in the positive null cone of $\mathbb{R}^{1,2}$ (using the hyperboloid model of $\H^2$). This cone can be described via the plane $\R^2$ where one identifies $(x,y)$ with $(-x,-y)$. The following proposition summarizes this description. We refer to \cite[Section 5]{Springborn_2018} for details and the proof.

\begin{proposition}\label{Prop:horocycle-bij}
The map
\begin{align*}
    \R^2\backslash\{0\}/\pm\mathrm{id} &\longrightarrow \{\text{horocycles in } \H^2\} \\
    (x,y) &\longmapsto C\left(\frac{x}{y}+\frac{i}{2y^2},\frac{1}{2y^2}\right)
\end{align*}
is a bijection, where we use the upper half-plane model for the hyperbolic plane and $C(z,r)$ denotes the Euclidean circle with center $z$ and radius $r$. Further, this bijection is equivariant with respect to the $\PSL_2(\R)$-action, acting by projective linear action on the right and by Möbius transformations on the left.
\end{proposition}
Note that primitive integer points corresponds to Ford circles.

\medskip
We are now ready for the proof of Theorem \ref{Thm:charact-regularity}.
We are grateful to Boris Springborn for the argument, which simplified considerably our original proof.

\begin{proof}[Proof of Theorem \ref{Thm:charact-regularity}]
Suppose the existence of an orientation-preserving involution $I$ in $\mathrm{PGL}_2(\Z)$ exchanging two rationals $\tfrac{a}{b}$ and $\tfrac{c}{d}$. Since $I$ is a symmetry of the Farey triangulation, we get that the Ford circles associated to $\tfrac{a}{b}$ and $\tfrac{c}{d}$ are exchanged.

By the correspondence \ref{Prop:horocycle-bij}, these Ford circles are represented by $[(a,b)]$ and $[(c,d)]$ and the involution $I$ corresponds to some linear map $M_I\in\mathrm{SL}_2(\Z)$ exchanging $[(a,b)]$ and $[(c,d)]$. Since $M_I$ has positive determinant, we get
$$\binom{a}{b} \xmapsto{M_I} \binom{c}{d} \xmapsto{M_I} \binom{-a}{-b}.$$
This determines $M_I$ completely:
\begin{equation}\label{Eq:matrix-exchange}
    M_I = \frac{1}{ad-bc}\begin{pmatrix}ab+cd & -a^2-c^2\\ b^2+d^2 & -ab-cd\end{pmatrix}.
\end{equation}
Therefore, $M_I\in\mathrm{SL}_2(\Z)$ implies that $ad-bc$ divides each of the terms $ab+cd, b^2+d^2$ and $a^2+c^2$. By Lemma \ref{Lemma:simplify-Springborn-cond}, this implies that $\gcd(ab+cd, b^2+d^2, a^2+c^2)=ad-bc$.

Conversely, if $(\tfrac{a}{b},\tfrac{c}{d})$ is inner regular, then the matrix given by Equation \eqref{Eq:matrix-exchange} is in $\mathrm{SL}_2(\Z)$ and defines an involution exchanging $\tfrac{a}{b}$ with $\tfrac{c}{d}$.

A similar argument holds for an outer regular pair, with the only difference that $M_I\in\mathrm{PGL}_2(\Z)$ exchanges the vectors $\binom{a}{b}$ and $\binom{c}{d}$. Hence it is given by
\begin{equation}\label{Eq:matrix-exchange2}
M_I = \frac{1}{ad-bc}\begin{pmatrix}-ab+cd & a^2-c^2\\ -b^2+d^2 & ab-cd\end{pmatrix}.
\end{equation}
\end{proof}

\begin{Remark}
The image $M_I. \infty $ is the Springborn sum or difference of a regular pair. This will be used to get geometric interpretations of Springborn operations and in the proof of our main Theorem \ref{Thm:main} later on.  
\end{Remark}

\begin{Remark}\label{Rk:reg-characterizations}
Using Proposition \ref{Prop:inversion-sym} and \ref{Prop:exchange-under-involution}, one can also prove that outer regularity is equivalent to either of the following:
\begin{enumerate}
    \item $\mathrm{gcd}(a+c,b+d)\,\mathrm{gcd}(a-c,b-d) \in\{d_F,2d_F\}$,
    \item $d_F\left(\frac{a+c}{b+d},\frac{a-c}{b-d}\right) \in \{1,2\}$.
\end{enumerate}
Similarly, inner regularity is equivalent to either of the following, where we use Definition \ref{Def:dF-in-Qi}:
\begin{enumerate}
    \item[(1')] $\mathrm{gcd}(a+ic,b+id)\,\mathrm{gcd}(a-ic,b-id)=d_F$,
    \item[(2')] $d_F\left(\frac{a+ic}{b+id},\frac{a-ic}{b-id}\right) =2$.
\end{enumerate}
\end{Remark}

\subsection{Geometric interpretations}

We give a list of geometric constructions giving the Springborn addition and difference.

The first one is not geometric, but algebraic, and can be obtained by a direct computation:
\begin{proposition}\label{prop:algebraic_springborn}
We have 
\begin{align*}
    \frac{a}{b}\oplus_S \frac{c}{d} &= \frac{1}{2}\left(\frac{a+ic}{b+id}+\frac{a-ic}{b-id}\right) =  \mathrm{Re}\left(\frac{a+ic}{b+id}\right) \\
    \frac{a}{b}\ominus_S \frac{c}{d} &= \frac{1}{2}\left(\frac{a+c}{b+d}+\frac{a-c}{b-d}\right). 
\end{align*}
\end{proposition}

A simple computation gives the following:
\begin{proposition}\label{Prop:harmonic_points}
The Farey sum $\tfrac{a}{b}\oplus_F\tfrac{c}{d}$ divides the segment $[\tfrac{a}{b},\tfrac{c}{d}]$ into two parts of ratio ${d}:{b}$, while the Springborn sum $\tfrac{a}{b}\oplus_S\tfrac{c}{d}$ divides it into parts of ratio ${d^2}:{b^2}$.
In particular, the four points $\left(\tfrac{a}{b},\tfrac{c}{d},\tfrac{a}{b}\oplus_S \tfrac{c}{d}, \tfrac{a}{b}\ominus_S \tfrac{c}{d}\right)$ are harmonic, i.e. their cross ratio is $-1$. 
\end{proposition}

Other geometric interpretations rely on the following explicit formula for the inner and outer homothety points of two circles:
\begin{proposition}\label{Prop-expl-formula-i-e}
For two circles $C_1$ and $C_2$, with centers $M_1$ and $M_2$ and radii $r_1$ and $r_2$ respectively, we have:
$$i(C_1,C_2) = \frac{r_1M_2+r_2M_1}{r_1+r_2} \;\;\text{ and }\;\; e(C_1,C_2) = \frac{r_1M_2-r_2M_1}{r_1-r_2}.$$
\end{proposition}

We can recover the Springborn operations using Ford circles.
The \emph{Ford circle} $F_{a/b}$ associated to a rational number $\frac{a}{b}$ is the circle with center $\left(\frac{a}{b},\frac{1}{2b^2}\right)$ and radius $\frac{1}{2b^2}$.

\begin{proposition}\label{Prop:geom-interpretation-ford}
For two rational numbers $\tfrac{a}{b}, \tfrac{c}{d}$, we have
$$
    \frac{a}{b}\ominus_S \frac{c}{d} = e(F_{a/b},F_{c/d}) \;\; \text{ and }\;\;
    \frac{a}{b}\oplus_S \frac{c}{d} = \mathrm{Re}(i(F_{a/b},F_{c/d})).
$$
\end{proposition}

In particular, if $d_F(\tfrac{a}{b}, \tfrac{c}{d})=1$, the two Ford circles are tangent and the $x$-coordinate of the contact point is given by $\tfrac{ab+cd}{b^2+d^2}$, as already stated in the original paper by Ford \cite{Ford}.

\begin{figure}[h!]
    \centering
    \includegraphics[height=5cm]{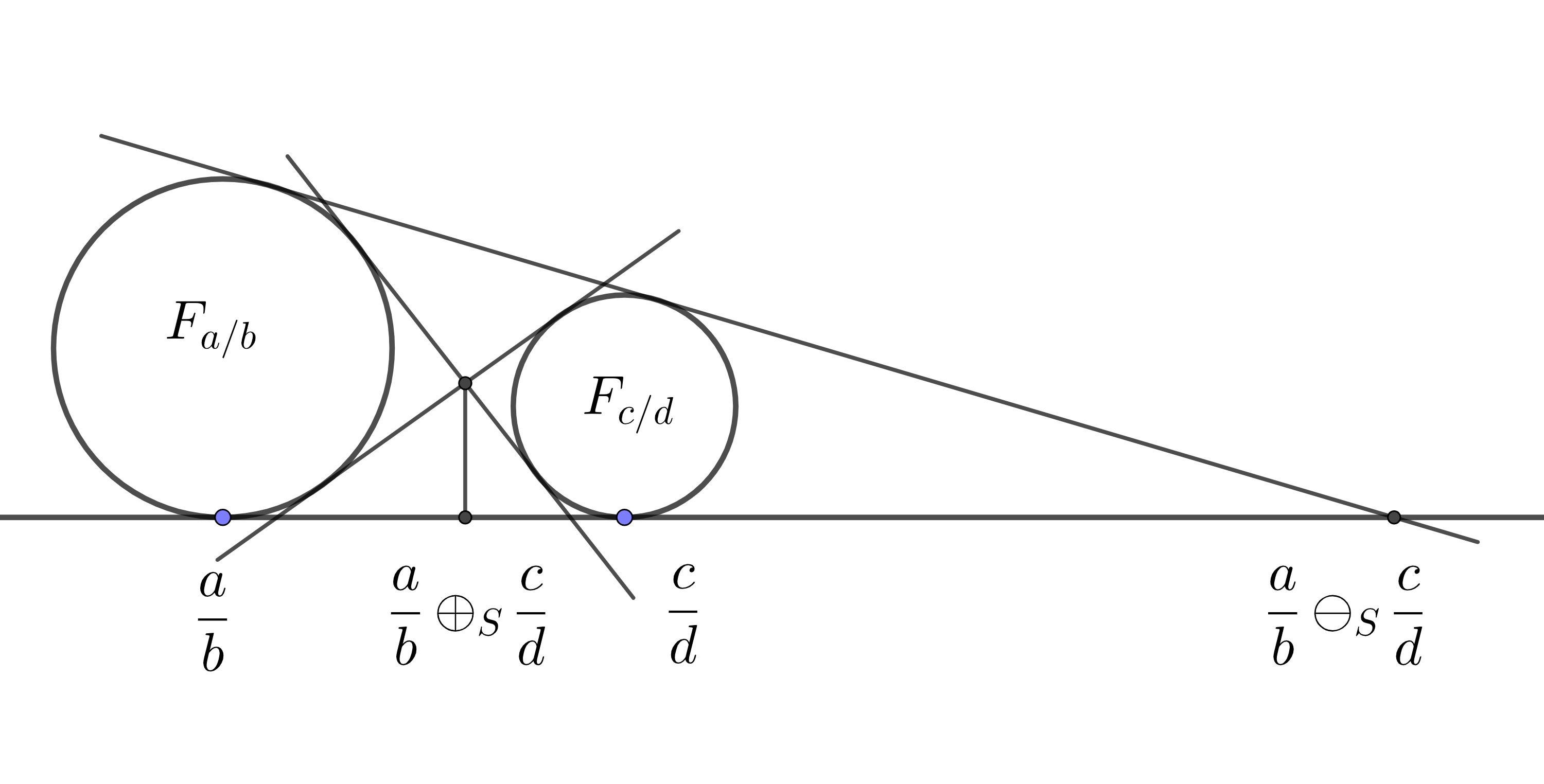}
    \caption{Springborn operations and Ford circles}
    \label{Fig:Springborn-Ford}
\end{figure}

\begin{proof}
We simply use the explicit formulas for inner and outer homothety center. We get
$$i(F_{a/b},F_{c/d}) = \frac{\tfrac{1}{2b^2}(\tfrac{c}{d}+\tfrac{i}{2d^2})+\tfrac{1}{2d^2}(\tfrac{a}{b}+\tfrac{i}{2b^2})}{\tfrac{1}{2b^2}+\tfrac{1}{2d^2}} = \frac{ab+cd}{b^2+d^2}+\frac{i}{b^2+d^2}.$$
Similarly, we get $e(F_{a/b},F_{c/d}) = \frac{ab-cd}{b^2-d^2}$.
\end{proof}

It turns out that in the proof of Proposition \ref{Prop:geom-interpretation-ford}, we have only used the $x$-coordinate of the center of the Ford circles (and the radius). So we can move them vertically without changing the result.

\begin{proposition}\label{Prop:geom-caract-Cab}
Denote by $C_{a/b}$ the circle with center $\tfrac{a}{b}$ and radius $\tfrac{1}{2b^2}$. Then 
$$\frac{a}{b}\oplus_S \frac{c}{d} = i(C_{a/b},C_{c/d}) \;\;\text{ and }\;\; \frac{a}{b}\ominus_S \frac{c}{d} = e(C_{a/b},C_{c/d}).$$
\end{proposition}

\begin{figure}[h!]
    \centering
    \includegraphics[height=5cm]{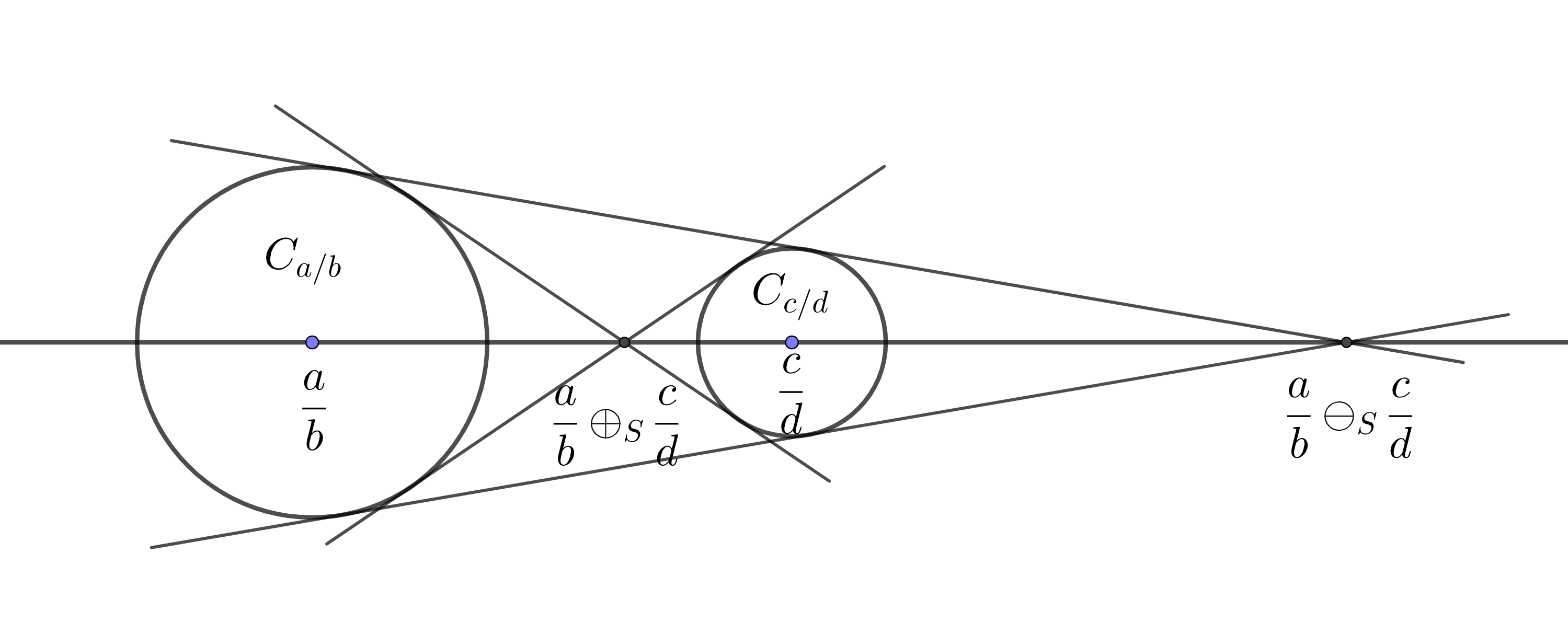}
    \caption{Springborn operations from Ford-like circles orthogonal to the real line}
    \label{Fig:Springborn-Ford-down}
\end{figure}

We are grateful to Boris Springborn, who mentioned an equivalent formulation using hyperbolic involutions:

\begin{proposition}\label{prop:homothety_geometric} Let $C_1$, $C_2$ be two circles centered on the real line, whose associated closed disks are disjoint. The outer homothety center $e(C_1,C_2)$ is the image of $\infty$ under the unique orientation-reversing isometry of $\mathbb{H}^2$ exchanging $C_1$ and $C_2$.

Similarly, the inner homothety center $i(C_1,C_2)$ is the image of $\infty$ under the unique orientation-preserving isometry of $\mathbb{H}^2$ exchanging $C_1$ and $C_2$.
\end{proposition}

\begin{proof}
The space of orientation-reversing isometries of $\mathbb{H}^2$ consists of inversions in geodesics (circle inversions or axial reflections from Euclidean viewpoint).
Suppose that there is an inversion $I_e$ with respect to a circle $C_e$ exchanging $C_1$ and $C_2$. Denote by $M_e$ the midpoint of $C_e$. Since the image $P'$ of any point $P$ under the inversion $I_e$ is on $PM_e$, we see that the two tangent lines from $M_e$ to $C_2$ are also tangent to $C_1$. Hence $M_e=e(C_1,C_2)$. The radius is then uniquely determined. Conversely, this data determines $C_e$ uniquely. Since an inversion exchanges the center of the circle with infinity, we get $e(C_1,C_2) = M_e = I_e(\infty)$.

For the inner homothety center, consider the geodesic $\gamma_3$ between $C_1$ and $C_2$ (which exists since $C_1$ and $C_2$ have disjoint interiors). This geodesic is part of a circle $C_3$ which is orthogonal to $C_1$ and $C_2$. Denote by $I_3$ the inversion in $C_3$. This inversion fixes $C_1$ and $C_2$, so also $C_e$. Hence $C_3$ is orthogonal to $C_e$.
Consider the composition $I_i = I_e\circ I_3$. Since $C_e$ is orthogonal to $C_3$, we also have $I_i = I_3\circ I_e$. Clearly $I_i$ is an orientation-preserving isometry of $\H^2$. It also exchanges $C_1$ with $C_2$. Finally, 
$$I_i(\infty) = I_3(I_e(\infty)) = I_3(e(C_1,C_2)) = i(C_1,C_2),$$
where we used that $i(C_1,C_2), e(C_1,C_2)$ together with $C_3\cap \R$ form four harmonic points, which can be proven by explicit computation.
\end{proof}

\begin{figure}[h!]
\centering
\begin{tikzpicture}[scale=7]

\path[use as bounding box] (-0.32,-0.25) rectangle (2,0.5);
\def\q{0.42}

\draw[->] (-0.2,0) -- (1.6,0) node[right] {};
\draw[->] (0,-0.2) -- (0,0.5) node[above] {};

\pgfmathsetmacro{\cA}{(\q+1)/2}
\pgfmathsetmacro{\rA}{(1-\q)/2}

\fill[gray!20] (\cA,0) circle (\rA);
\draw[blue, thick] (\cA,0) circle (\rA);

\node at (\cA,0) {$C_{1}$};

\pgfmathsetmacro{\a}{1+\q*\q}
\pgfmathsetmacro{\b}{1+\q}
\pgfmathsetmacro{\cB}{(\a+\b)/2}
\pgfmathsetmacro{\rB}{(\b-\a)/2}

\fill[gray!20] (\cB,0) circle (\rB);
\draw[blue, thick] (\cB,0) circle (\rB);

\node at (\cB,0) {$C_{2}$};

\pgfmathsetmacro{\cC}{(1+(1+\q))/2}
\pgfmathsetmacro{\rC}{((1+\q)-1)/2}

\path[name path=semiA] (\cC,0) circle (\rC);

\begin{scope}
\clip (-1,0) rectangle (2,2);
\draw[thick] (\cC,0) circle (\rC);
\end{scope}

\pgfmathsetmacro{\cD}{(\q+(1+\q*\q))/2}
\pgfmathsetmacro{\rD}{((1+\q*\q)-\q)/2}

\path[name path=semiB] (\cD,0) circle (\rD);

\begin{scope}
\clip (-1,0) rectangle (2,2);
\draw[thick] (\cD,0) circle (\rD);
\end{scope}

\path[name intersections={of=semiA and semiB, by={P1,P2}}];

\coordinate (P) at (P1);

\fill (P) circle (0.5pt);
\node[above right] at (P) {$P$};

\draw[dashed, thick] (P|-0,0.5)-- (P) -- (P |- 0,0);

\fill[red] (P |- 0,0) circle (0.5pt);

\node[below, xshift=-6pt, yshift=-2pt, text=red] at (P|-0.05,-0.05)
{\footnotesize{$i(C_1,C_2)$}};
\end{tikzpicture}
    \caption{The inner homothety center of two non-intersecting circles is defined via the following hyperbolic geometric construction. Connect their left (and, correspondingly, right) endpoints with geodesics, and take their intersection point $P$. The inner homothety center $i(C_1,C_2)$ is the orthogonal projection of $P$ on the real line.} \label{fig:Springborn_illustration}
\end{figure}

\begin{coro}
For two rational numbers $\frac{a}{b}$ and $\frac{c}{d}$, the inner homothety center $i$ of their corresponding $q$-disks can be obtained via a following (hyperbolic) geometric construction. Take two geodesics $\gamma_1$ and $\gamma_2$, connecting $[\frac{a}{b}]_q^{\flat}$ with $[\frac{c}{d}]_q^{\flat}$, and  $[\frac{a}{b}]_q^{\sharp}$ with  $[\frac{c}{d}]_q^{\sharp}$, respectively. Let $P:=\gamma_1 \cap \gamma_2$. Then, the geodesic line from $\infty$ to $P$ intersects the absolute in the point $i$.
\end{coro}

\begin{proof}
We prove a more general statement that holds for any two disjoint disks $C_1$ and $C_2$, orthogonal to the real line, see Figure \ref{fig:Springborn_illustration}.

Apply Proposition \ref{prop:homothety_geometric} to the two disks $C_1$ and $C_2$. The only orientation-preserving isometry exchanging these circles is the rotation of angle $\pi$ around the midpoint of the common perpendicular of $C_1$ and $C_2$. This rotation preserves the geodesics $\gamma_1$ and $\gamma_2$, hence $P$ is the rotation center. Hence, the infinity is sent to the inner homothety center under such an involution.
\end{proof}

\subsection{Reversion and iteration of Springborn operations}

We can ask whether any rational number $\tfrac{x}{y}$ can be represented as the Springborn sum of two others. We answer this question positively and analyze under which condition one can iterate the operation. In this subsection, we consider Springborn operations for pairs $(\tfrac{a}{b}, \tfrac{c}{d})$ satisfying the relation $\gcd(ab+cd, b^2+d^2)=ad-bc$. Note that these pairs are not necessarily regular.

\begin{proposition}
The equation $\tfrac{a}{b}\oplus_S \tfrac{c}{d}=\tfrac{x}{y}$, where $a, b, x, y$ are given and $c, d$ are unknowns, has as solution
$$c= ax-\frac{b(1+x^2)}{y}\;\;\text{ and }\;\; d=ay-bx \;.$$
Hence the equation admits a solution if and only if $y\mid b(1+x^2)$ and $\gcd(c,d)=1$. 

Similarly, $\tfrac{a}{b}\ominus_S \tfrac{c}{d}=\tfrac{x}{y}$ has solution $c=-ax+\frac{b(x^2-1)}{y}$ and $d=bx-ay$, which is well-defined if and only if $y\mid b(x^2-1)$ and $\gcd(c,d)=1$ and $b\neq d$.
\end{proposition}

\begin{proof}
Put $d_F=ad-bc$. We want to solve
$$
\left\{
\begin{aligned}
ab + cd &= x d_F \\
b^2 + d^2 &= y d_F
\end{aligned}
\right..
$$
Fix $d\in\mathbb{Z}$ such that $y\mid b^2+d^2$. Then
$$c=\frac{1}{d}\left(x\tfrac{b^2+d^2}{y}-ab\right)=\frac{x(b^2+d^2)-aby}{yd}.$$
It follows that $d_F=ad-bc=\frac{(b^2+d^2)(ay-bx)}{yd}$. Hence $d_F=\frac{b^2+d^2}{y}$ if and only if $d=ay-bx$. We then get $c=ax-\frac{b(1+x^2)}{y}$ by a direct computation.

Therefore, we see that $y\mid b(1+x^2)$ is a necessary condition to solve the system. Note that this also implies $y\mid b^2+d^2$ for $d=ay-bx$. If $y\mid b(1+x^2)$, the system has a solution if and only if $\tfrac{c}{d}$ is a reduced fraction, i.e. if and only if $\gcd(c,d)=1$.

The same argument holds for the Springborn difference.
\end{proof}

\begin{example}
    Consider $x=5, y=2$ and $a,b$ arbitrary (coprime). Then $d=2a-5b$ and $c=5a-13b$ and a direct computation gives that $\gcd(c,d)=\gcd(a,b)=1$. Hence $\tfrac{5}{2}$ can be written as Springborn sum in infinitely many ways.
\end{example}

\begin{proposition}
Every rational number is the Springborn sum or Springborn difference of two other numbers.
\end{proposition}
\begin{proof}
Let $\tfrac{x}{y}\in\Q$ be a reduced fraction. We use the previous proposition with $b=y$. Then $c=x(a-x)-1$ and $d=y(a-x)$. In particular $\gcd(c,d)=\gcd(x^2-ax+1,y)$ because of $\gcd(c,a-x)=1$. Since $x$ and $y$ are coprime, by Bézout's identity, there are integers $\alpha,\beta$ such that $\alpha x+ \beta y=1$. Choose $a=\alpha$. Then $\gcd(a,b)=\gcd(\alpha,y)=1$ by Bézout's identy, and $\gcd(c,d)=\gcd(x^2+\beta y,y)=\gcd(x^2,y)=1$.
Therefore we have two reduced fractions and $$\frac{\alpha}{y}\oplus_S\frac{x^2+\beta y}{y(x-\alpha)}=\frac{x}{y}.$$
Similary for the Springborn difference, $b=y$ and $a=-\alpha$ gives
$$\frac{-\alpha}{y}\ominus_S\frac{x^2-\beta y}{y(x+\alpha)}=\frac{x}{y}.$$
\end{proof}

We will now show that, then restricted to the neighboring pairs in the Farey graph, the Springborn sum is far from being surjective. For this, let us first 
remind the symbolic coding on rationals defined via a geodesic flow on the hyberbolic plane, following the work of Series \cite{Series}. An (oriented) geodesic from $\infty$ to $x \in \Q \cap (0,1)$ starts in the basic triangle of the Farey triangulation $\Delta_0=\{0,1,\infty\}$,  and then makes left and right turns while crossing over to neighboring triangles $\Delta_j, j=1, \ldots, k$, till it reaches $x$. A corresponding word $\omega(x)$ (of length $k$) encoding the sequence of turns (with $L$ for left and $R$ for right) is a symbolic coding of a rational $x$. Such a symbolic coding has many beautiful properties, in particular one can recover a continued fraction of $x$ by aggregating the numbers of consequent left and right turns. Indeed, for an even-length expansion $x=[0; a_1,\dots,a_{2n}]$ with $a_{2n} \geq 1$, the corresponding word is $L^{a_1}R^{a_2}L^{a_3}\ldots L^{a_{2n-1}}R^{a_{2n}-1}$. 

Define two operations on the words in the alphabet $\{L,R\}$. The swap $s$, that replaces all L’s by R’s and vice-versa; and the backwards reading b that replaces the word by itself read from right to left. 
Note that these are two commuting involutions. We denote the image of the word $\omega$ under their composition by $\widetilde{\omega}:=b \circ s (\omega)= s \circ b (\omega)$. 

\begin{example}
For $\tfrac{1}{2}=[0;1,1]$, the corresponding word $\omega(\tfrac{1}{2})=L$, whereas for $\tfrac{2}{5}=[0;2,1,1]$ one has  $\omega(\tfrac{2}{5})=LLR$ and for $\tfrac{5}{13}=[0;2,1,1,2]$, one has 
$\omega(\tfrac{5}{13})=LLRLR$. 
Moreover, the action of the involutions on the last word is the following : $s(LRRLR)=RLLRL$ and $b(LRRLR)=RLRRL$. 
\end{example}

\smallskip

The following proposition shows how the symbolic words of the Farey and Springborn sums of two neighboring rationals are related.

\begin{proposition}
Let $\tfrac{a}{b}, \tfrac{c}{d} \in (0,1)$ be two Farey neighbors. Suppose that their Farey sum $F:=\tfrac{a}{b} \oplus_F \tfrac{c}{d}$ has a symbolic coding $L \omega$. Then the symbolic word of their Springborn sum $S:=\tfrac{a}{b} \oplus_S \tfrac{c}{d}$ is exactly $L \omega \widetilde{\omega}$.
\end{proposition}
\begin{proof}
Geometrically, the Farey sum $F$ is a tangency point of a unique circle tangent to both Ford circles $F_{a/b}$ and $F_{c/d}$ and the absolute. Further, the Springborn sum $S$ is the orthogonal projection to the real line of the tangency point $P$ of the two Ford circles of $F_{a/b}$ and $F_{c/d}$, as already shown by Ford \cite{Ford} in $1938$. This implies that the corresponding vertical geodesics $\gamma_S$ and $\gamma_F$ (going into $S$ and $F$ correspondingly) pass by the same triangles in the Farey triangulation till they reach their common triangle $\Delta=\{\frac{a}{b}, \frac{a+c}{b+d}, \frac{c}{d}\}$, see Figure \ref{fig:springborn_sum_square}.
Moreover, since a half-turn around the point $P$ preserves the Farey triangulation and exchanges the endpoints of the geodesic $\gamma_S$, then the sequence of triangles crossed by $\gamma_S$ has a central symmetry. This implies the needed condition on the word $\omega(S)$. Let us remark that the prefix $L$ corresponds to a step from the basic triangle $\Delta_0$ to the triangle $\Delta_1$ while no turn is needed from the last triangle $\Delta_k$ (for the sequence corresponding to $S$, the number of triangles $k-1$ is necessarily even). 
\end{proof}

\begin{figure}[h!]
\centering
\begin{tikzpicture}[scale=4]
  \draw (-0.75,0) -- (1,0);
  \draw (0,0.5) circle (0.5);
  \draw (0.5,1/8) circle (1/8);
  \draw (1/3,1/18) circle (1/18);
  \draw[dotted] (0.5,0) arc(0:180:0.25);
  \draw[dotted] (0.5,0) arc(0:180:1/12);
  \draw[dotted] (1/3,0) arc(0:180:1/6);
  \draw (1/3,1.1) -- (1/3,0) node[midway,sloped] {$>$};
  \draw[orange] (2/5,1.1) -- (2/5,0);
  \node[orange] at(2/5,0.197) {\scriptsize{$\bullet$}};
  \node[orange] at(0.5,0.3) {\scriptsize{$P$}};
  \node[orange] at(2/5,1) {$\lor$};

  \node at(0,-0.1) {$\frac{a}{b}$};
  \node at(0.5,-0.1) {$\frac{c}{d}$};
  \node at(0.3,-0.06) {\scriptsize{$F$}};
  \node[orange] at(0.4,-0.06) {\scriptsize{$S$}};
  \node at(0,0) {\scriptsize{$\bullet$}};
  \node at(0.5,0) {\scriptsize{$\bullet$}};
  \node at(1/3,0) {\scriptsize{$\bullet$}};
  \node[orange] at(2/5,0) {\scriptsize{$\bullet$}};
\end{tikzpicture}
\caption{Springborn and Farey sums $S$ and $F$ of two Farey neighbors.}
\label{fig:springborn_sum_square}
\end{figure}

\medskip

We now examine under which conditions we can iterate the Springborn sum or difference.

\begin{proposition} \label{prop:iteration}
Let $\left(\frac{a}{b},\frac{c}{d}\right)$ be an inner regular pair with $\gcd(b,d)=1$. If in addition $b \mid a^2+1$ and $d\mid c^2+1$, then the pairs $\left(\frac{a}{b},\frac{a}{b}\oplus_S \frac{c}{d}\right)$ and $\left(\frac{a}{b}\oplus_S \frac{c}{d},\frac{c}{d}\right)$ are inner regular.

Similarly, if $\left(\frac{a}{b},\frac{c}{d}\right)$ is outer regular, $\gcd(b,d)=1$,  $b\mid a^2-1$ and $d\mid c^2-1$, then one can iterate the Springborn difference. The process might stop, but only when it reaches $\infty=\tfrac{1}{0}$.
\end{proposition}

\begin{proof}
For the Springborn addition, we have to show the following:
\begin{enumerate}
    \item $\mathrm{gcd}(b,(b^2+d^2)/d_F)=1$
    \item $\mathrm{gcd}\left(ab+\frac{(ab+cd)(b^2+d^2)}{d_F^2},b^2+\frac{(b^2+d^2)^2}{d_F^2}\right)=d_F\left(\frac{a}{b},\frac{(ab+cd)/d_F}{(b^2+d^2)/d_F}\right)$
    \item $(b^2+d^2)/d_F$ divides $\left(\frac{ab+cd}{d_F}\right)^2+1$
\end{enumerate}
and the same items with $\frac{a}{b}$ exchanged by $\frac{c}{d}$. Since all hypotheses are symmetric in this exchange, we will get automatically these items from the list above.
In addition, we have to prove similar items for the Springborn difference. 

Since $\mathrm{gcd}(b,d)=1$, we see that $b^2+d^2$ and $b^2-d^2$ are relatively prime to both $b$ and $d$. Hence, $(b^2+d^2)/d_F$ and $(b^2-d^2)/d_F$ too, which gives (1).

Let us compute the new Farey determinants.
$$d_F\left(\frac{a}{b},\frac{(ab+cd)/d_F}{(b^2+d^2)/d_F}\right)=\frac{a(b^2+d^2)-b(ab+cd)}{d_F}=\frac{d(ad-bc)}{d_F}=d.$$
Similarly we get
$$d_F\left(\frac{a}{b},\frac{(ab-cd)/d_F}{(b^2-d^2)/d_F}\right)=\frac{a(b^2-d^2)-b(ab-cd)}{d_F}=\frac{-d(ad-bc)}{d_F}=-d.$$
We use Lemma \ref{Lemma:simplify-Springborn-cond} to check the second point. Since $\gcd(d,d_F)=\gcd(d,bc)=1$, we are left with $ab d_F^2+(ab+cd)(b^2+d^2)$ and $b^2d_F^2+(b^2+d^2)^2$. Modulo $d$, the first one reduces to $ab^3(c^2+1)$ and the second to $b^4(c^2+1)$. Since $\mathrm{gcd}(c,d)=1$, we see that $d$ divides both expressions if and only if $d$ divides $c^2+1$, which gives (2).
In the case of the Springborn difference, very similar arguments hold and give $d$ divides $c^2-1$.

Finally for (3), we compute
$$\left(\frac{ab+cd}{d_F}\right)^2+1 = \frac{1}{d_F^2}((ab+cd)^2+(ad-bc)^2)=\frac{1}{d_F^2}(a^2+c^2)(b^2+d^2),$$
which is divisible by $(b^2+d^2)/d_F$. A similar computation holds for the Springborn difference:
$$\left(\frac{ab-cd}{d_F}\right)^2-1 = \frac{1}{d_F^2}((ab-cd)^2-(ad-bc)^2)=\frac{1}{d_F^2}(a^2-c^2)(b^2-d^2).$$
\end{proof}

\begin{example}
Starting from $(\tfrac{0}{1},\tfrac{1}{1})$ with Springborn addition, we get the Markov fractions described in the initial paper by Springborn \cite{Springborn}. Up to 3 iterations this gives the following $9$ numbers that we represent via a graph : 

\begin{center}
   \begin{tikzpicture}
    \draw (0,0) arc[start angle=180, end angle=0, x radius=4, y radius=3];

    \draw (0,0) arc (180:0:2);
    \draw (4,0) arc (180:0:2);
    
    \draw (0,0) arc (180:0:1);
    \draw (2,0) arc (180:0:1);
    \draw (4,0) arc (180:0:1);
    \draw (6,0) arc (180:0:1);
    
    \draw (0,0) arc (180:0:0.5);
    \draw (1,0) arc (180:0:0.5);
    \draw (2,0) arc (180:0:0.5);
    \draw (3,0) arc (180:0:0.5);
    \draw (4,0) arc (180:0:0.5);
    \draw (5,0) arc (180:0:0.5);
    \draw (6,0) arc (180:0:0.5);
    \draw (7,0) arc (180:0:0.5);
    
    \node at(0,-0.5) {$\frac{0}{1}$};
    \node at(1,-0.5) {$\frac{5}{13}$};
    \node at(2,-0.5) {$\frac{2}{5}$};
    \node at(3,-0.5) {$\frac{12}{29}$};
    \node at(4,-0.5) {$\frac{1}{2}$};
    \node at(5,-0.5) {$\frac{17}{29}$};
    \node at(6,-0.5) {$\frac{3}{5}$};
    \node at(7,-0.5) {$\frac{8}{13}$};
    \node at(8,-0.5) {$\frac{1}{1}$};
\end{tikzpicture}
   \end{center}
\end{example}
We draw the edges of the graph as curves in order to better represent its combinatorial structure. The same is done in the following two examples.
\begin{example}
Starting from $(\tfrac{0}{1},\tfrac{1}{2})$ with Springborn difference, we get fractions of the form $\frac{n}{n+1}$. Up to 3 iterations this gives:
\begin{center}
   \begin{tikzpicture}
    \draw (0,0) arc[start angle=180, end angle=0, x radius=4, y radius=3];

    \draw (0,0) arc (180:0:2);
    \draw (4,0) arc (180:0:2);
    
    \draw (0,0) arc (180:0:1);
    \draw (2,0) arc (180:0:1);
    \draw (4,0) arc (180:0:1);
    \draw (6,0) arc (180:0:1);
    
    \draw (0,0) arc (180:0:0.5);
    \draw (1,0) arc (180:0:0.5);
    \draw (2,0) arc (180:0:0.5);
    \draw (3,0) arc (180:0:0.5);
    \draw (4,0) arc (180:0:0.5);
    \draw (5,0) arc (180:0:0.5);
    \draw (6,0) arc (180:0:0.5);
    \draw (7,0) arc (180:0:0.5);
    
    \node at(0,-0.5) {$\frac{0}{1}$};
    \node at(1,-0.5) {$\frac{4}{5}$};
    \node at(2,-0.5) {$\frac{3}{4}$};
    \node at(3,-0.5) {$\frac{6}{7}$};
    \node at(4,-0.5) {$\frac{2}{3}$};
    \node at(5,-0.5) {$\frac{7}{8}$};
    \node at(6,-0.5) {$\frac{4}{5}$};
    \node at(7,-0.5) {$\frac{6}{7}$};
    \node at(8,-0.5) {$\frac{1}{2}$};
\end{tikzpicture}
\end{center}

Note that the order is not respected, and that the sequence is not injective.
\end{example}

\begin{example}
Starting from $(\tfrac{0}{1},\tfrac{1}{3})$ with Springborn difference, we get all companions of $\frac{0}{1}$ as described in Springborn's original paper. Up to 3 iterations this gives:
\begin{center}
   \begin{tikzpicture}
    \draw (0,0) arc[start angle=180, end angle=0, x radius=4, y radius=3];

    \draw (0,0) arc (180:0:2);
    \draw (4,0) arc (180:0:2);
    
    \draw (0,0) arc (180:0:1);
    \draw (2,0) arc (180:0:1);
    \draw (4,0) arc (180:0:1);
    \draw (6,0) arc (180:0:1);
    
    \draw (0,0) arc (180:0:0.5);
    \draw (1,0) arc (180:0:0.5);
    \draw (2,0) arc (180:0:0.5);
    \draw (3,0) arc (180:0:0.5);
    \draw (4,0) arc (180:0:0.5);
    \draw (5,0) arc (180:0:0.5);
    \draw (6,0) arc (180:0:0.5);
    \draw (7,0) arc (180:0:0.5);
    
    \node at(0,-0.5) {$\frac{0}{1}$};
    \node at(1,-0.5) {$\frac{21}{55}$};
    \node at(2,-0.5) {$\frac{8}{21}$};
    \node at(3,-0.5) {$\frac{144}{377}$};
    \node at(4,-0.5) {$\frac{3}{8}$};
    \node at(5,-0.5) {$\frac{377}{987}$};
    \node at(6,-0.5) {$\frac{21}{55}$};
    \node at(7,-0.5) {$\frac{144}{377}$};
    \node at(8,-0.5) {$\frac{1}{3}$};
\end{tikzpicture}
\end{center}
\end{example}

\section{Springborn operations for $q$-rationals}\label{Sec:q-Springborn}

We use our geometric picture of $q$-rationals, seen as circles, and apply the Springborn operations to them. For regular pairs, we get an explicit reduced formula.

\subsection{Springborn operations for regular pairs}

Recall that to each rational $x$, we associate the hyperbolic geodesic $[x]$, from $[x]_q^\sharp $ to $[x]^\flat_q$. Let us work in the upper-half plane model. So $[x]$ is a half circle, which we can complete to a full circle (using complex conjugation). We still denote by $[x]$ the full circle. The union of all $[x]$ for $x\in \mathbb{QP}^1$ is denoted by $\mathcal{Q}$.

Using Definition \ref{Def:i-and-e} we can then speak about the inner and outer homothety centers $i([x],[y])$ and $e([x],[y])$ of two $q$-rationals, where $x,y\in\Q$, $x\neq y$. We extend to the case of $y=\infty$ by $i([x],[\infty])=[x]_q^\sharp $ and $e([x],[\infty])=[x]_q^\flat$.

We can now relate precisely the homothety centers construction and its algebraic counterpart.

\begin{theorem}\label{Thm:main}
Let $(\tfrac{a}{b},\tfrac{c}{d})\in\Q^2$. If the pair is inner regular, then
$$\left[\frac{a}{b}\oplus_S\frac{c}{d}\right]^\sharp _q = i\left(\left[\frac{a}{b}\right],\left[\frac{c}{d}\right]\right).$$
If the pair is outer regular, then
$$\left[\frac{a}{b}\ominus_S\frac{c}{d}\right]^\flat_q = e\left(\left[\frac{a}{b}\right],\left[\frac{c}{d}\right]\right).$$
\end{theorem}

Note that an important special case is when the Farey determinant $d_F(\tfrac{a}{b},\tfrac{c}{d})$ is 1 or 2. Then the pair is automatically inner and outer regular.

The proof is a combination of the characterization of regular pairs from Theorem~\ref{Thm:charact-regularity}, and the symmetries of the $q$-Farey tesselation $\mathcal{Q}$. 

\begin{proof}
Consider an outer regular pair $(\tfrac{a}{b},\tfrac{c}{d})$. From Theorem \ref{Thm:charact-regularity}, we know that there is an inversion $I$ in $\PGL_2(\Z)$ exchanging $\tfrac{a}{b}$ with $\tfrac{c}{d}$. So its $q$-deformation $I_q$ is a symmetry of $\mathcal{Q}$.
By Proposition \ref{prop:homothety_geometric}, we then get that $$e\left(\left[\frac{a}{b}\right],\left[\frac{c}{d}\right]\right)=I_q(\infty).$$
Since $I_q$ is a symmetry of $\mathcal{Q}$ and $\infty=[\infty]_q^\sharp$, we see that $I_q(\infty)$ is the left-most point of some circle belonging to $\mathcal{Q}$ (orientation-reversing isometries exchange left and right $q$-deformations). 

The exact same argument, using again Theorem \ref{Thm:charact-regularity}, works for an inner regular pair. The only difference is that orientation-preserving symmetries of $\mathcal{Q}$ preserve right $q$-rationals, so that the inner homothety center is the right-most point of some circle belonging to $\mathcal{Q}$.

To get the label of the $q$-disks, we can use the explicit expression of the involution from Equations \eqref{Eq:matrix-exchange} and \eqref{Eq:matrix-exchange2}. For the inner regular pair, we have
$$M_I=\begin{pmatrix}ab+cd & -a^2-c^2\\ b^2+d^2 & -ab-cd\end{pmatrix}.$$
The label is given by the image of $\infty$ under $M_I$, i.e. the first column of the matrix.
\end{proof}

\begin{Remark}
Another way to get the label is to compute the limit of $i([\tfrac{a}{b}],[\tfrac{c}{d}])$ as $q\to 1$, using the explicit formula \ref{Prop-expl-formula-i-e}. The center and radius of $[\tfrac{a}{b}]$ satisfy (see Equation \eqref{Eq:diam-near-1})
$$M\left(\left[\frac{a}{b}\right]\right) = \frac{a}{b} + o(1) \;,\;\;     r\left(\left[\frac{a}{b}\right]\right) = \frac{\lvert 1-q\rvert}{2b^2} + O((1-q)^2).$$
\end{Remark}

\begin{Remark}
It is puzzling that the property described in Theorem \ref{Thm:main} is also valid for some non-regular pairs. An example is given by the pair $\left(\frac{1}{3},\tfrac{2}{9}\right)$, which is not regular (inner nor outer), but for which $i([\tfrac{1}{3}]_q,[\tfrac{2}{9}]_q)=[\tfrac{7}{30}]_q^\sharp  = [\tfrac{1}{3}\oplus_S \tfrac{2}{9}]_q^\sharp $.
Another example is given by the pair $(\tfrac{2}{7},\tfrac{3}{7})$ which is not regular, but for which $e([\tfrac{2}{7}]_q,[\tfrac{3}{7}]_q)=[\infty]_q^\flat = [\tfrac{2}{7}\ominus_S \tfrac{3}{7}]_q^\flat$.

\noindent Understanding these exceptional pairs seems a challenging task.
\end{Remark}

\begin{Remark}
Our original proof of Theorem \ref{Thm:main} was a bit different and used some known results from elementary number theory, which we find interesting to share.

\noindent The idea is to show that if a pair $(\tfrac{a}{b}, \tfrac{c}{d})$ is inner regular, then there is another pair $(\tfrac{a'}{b'}, \tfrac{c'}{d'})$ of Farey determinant 1 which has the same Springborn sum. Using the $\PSL_2(\Z)$-symmetry, we can assume that $(\tfrac{a}{b}, \tfrac{c}{d})=(\tfrac{1}{0},\frac{k}{n})$ is a standard pair with $n\mid k^2+1$. Finding $\tfrac{a'}{b'}$ and $\tfrac{c'}{d'}$ becomes then solving a system of Diophantine equations
$$
\left\{
\begin{aligned}
    a'b'+c'd' &= k\\
    b'^2+d'^2 &= n\\
    a'd'-b'c' &= 1
\end{aligned}
\right..
$$
To solve the system, we need to write $n$ as a primitive sum of two squares (meaning that $\mathrm{gcd}(b',d')=1$), and then we get $(a', c')$ as the Bézout coefficients. To fulfill the first equation, we need the following well-known bijection (so well-known that we were not able to find an exact reference):
$$\{(b,d)\in\N^2\mid b^2+d^2=n, \mathrm{gcd}(b,d)=1\} \overset{1:1}{\longleftrightarrow}  \{0<k<n\mid k^2\equiv -1 \!\mod n\}.$$
The case of outer regularity is similar, but a bit more complicated. A pair $(\tfrac{a}{b}, \tfrac{c}{d})$ is outer regular if there is another pair $(\tfrac{a'}{b'}, \tfrac{c'}{d'})$ of Farey determinant \emph{1 or 2} which has the same Springborn difference. To solve the analogous system, we are led to write $n$ or $2n$ as a difference of two squares. This leads to the following bijection which seems much less known:
\begin{gather*} 
\{(b,d)\in\N^2\mid b^2\!-\!d^2=n, \mathrm{gcd}(b,d)=1\} \cup \{(b,d)\in \N^2\mid b^2\!-\!d^2=2n, \mathrm{gcd}(\tfrac{b+d}{2},\tfrac{b-d}{2})=1\} \\
    \overset{1:1}{\longleftrightarrow}\\
    \{0 < k\leq \tfrac{n}{2}\mid k^2\equiv 1 \!\mod n\}.
\end{gather*}
\end{Remark}

\subsection{General non-reduced expression}
We determine an explicit expression for the inner and outer homothetic centers of two $q$-rationals, seen as circles.

\begin{proposition}\label{prop:homotheticformulas}
Let $\frac{a}{b},\frac{c}{d}\in\mathbb{Q}$. Denote by $\varepsilon_1$ (resp. $\varepsilon_2$) the integers associated to $a/b$ (resp. $\frac{c}{d}$), see Definition \eqref{defi:epsilon}. Then, we have
$$
i\left(\left[\frac{a}{b}\right],\left[\frac{c}{d}\right]\right) = \frac{q^{\varepsilon_2}A^\sharp B^{\flat}+q^{\varepsilon_1}C^{\flat} D^\sharp }{q^{\varepsilon_2}B^\sharp B^{\flat}+q^{\varepsilon_1}D^\sharp D^{\flat}}.$$
Similarly, for the outer homothetic center, we have
$$ e\left(\left[\frac{a}{b}\right],\left[\frac{c}{d}\right]\right) = \frac{q^{\varepsilon_2}A^\sharp B^{\flat}-q^{\varepsilon_1}C^{\sharp } D^\flat}{q^{\varepsilon_2}B^\sharp B^{\flat}-q^{\varepsilon_1}D^\sharp D^{\flat}}. $$
\end{proposition}

\begin{proof}
We combine the explicit formula for homothety centers (Proposition \ref{Prop-expl-formula-i-e}) with the formula in Proposition \ref{prop:qradius} for the (Euclidean) diameter of a $q$-rational.
This gives :
\begin{align*}
     i\left(\left[\frac{a}{b}\right],\left[\frac{c}{d}\right]\right) &= \frac{r_2M_1+r_1M_2}{r_1+r_2}\\
    &=\frac{r_2([\tfrac{a}{b}]^\sharp-r_1)+r_1([\tfrac{c}{d}]^{\flat}+r_2)}{r_1+r_2} \\
    &= \frac{\frac{q^{\varepsilon_2}\lvert q-1\rvert}{2D^\sharp D^\flat}\frac{A^\sharp}{B^\sharp}+\frac{q^{\varepsilon_1}\lvert q-1\rvert}{2B^\sharp B^\flat}\frac{C^\flat}{D^\flat}}{\frac{q^{\varepsilon_1}\lvert q-1\rvert}{2B^\sharp B^\flat}+\frac{q^{\varepsilon_2}\lvert q-1\rvert}{2D^\sharp D^\flat}} \\
    &= \frac{q^{\varepsilon_2}A^\sharp B^{\flat}+q^{\varepsilon_1}C^{\flat} D^\sharp}{q^{\varepsilon_2}B^\sharp B^{\flat}+q^{\varepsilon_1}D^\sharp D^{\flat}}\, .
\end{align*}
The same argument holds for the Springborn difference.
\end{proof}

\begin{Remark}\label{rem:sym_sharp_and_flat}
Using the commutativity of the homothetic centers construction, we also get
$$i\left(\left[\frac{a}{b}\right],\left[\frac{c}{d}\right]\right) = \frac{q^{\varepsilon_2}A^\flat B^{\sharp}+q^{\varepsilon_1}C^{\sharp} D^\flat}{q^{\varepsilon_2}B^\sharp B^{\flat}+q^{\varepsilon_1}D^\sharp D^{\flat}} \, .$$
Both formulas give the same result since $A^\sharp B^\flat - A^\flat B^\sharp = q^{\varepsilon_1}(1-q)$ and the same for $\frac{c}{d}$. 

For the difference, we also get
$$e\left(\left[\frac{a}{b}\right],\left[\frac{c}{d}\right]\right)  = \frac{q^{\varepsilon_2}A^\flat B^{\sharp}-q^{\varepsilon_1}C^{\flat} D^\sharp}{q^{\varepsilon_2}B^\sharp B^{\flat}-q^{\varepsilon_1}D^\sharp D^{\flat}} \, .$$
\end{Remark}

The formulas given above for homothetic centers are $q$-deformations of the formulas for Springborn operations. As in the classical case, they are not reduced in general. This is a challenging problem to find the reduced forms. We partially solve this problem for regular pairs.

\subsection{Reduced form}

In this subsection, we determine the reduced version of the $q$-deformed Springborn operations for regular pairs. 

\begin{theorem}\label{Thm:q-gcd}
Let $(\tfrac{a}{b},\tfrac{c}{d})\in\mathbb{Q}^2$ be an inner regular pair in the sense of Definition \ref{Def:reg}. Denote by $\varepsilon_1=\varepsilon(\tfrac{a}{b})$ and $\varepsilon_2=\varepsilon(\tfrac{c}{d})$. Then 
\begin{equation*}
    \gcd(q^{\varepsilon_2}A^\sharp B^\flat+q^{\varepsilon_1}C^\flat D^\sharp,q^{\varepsilon_2}B^\sharp B^\flat+q^{\varepsilon_1}D^\flat D^\sharp,q^{\varepsilon_2}A^\sharp A^\flat+q^{\varepsilon_1}C^\flat C^\sharp )\equiv_q d_F^{\sharp \flat} \equiv_q d_F^{\flat\sharp }.
\end{equation*}

Similarly for an outer regular pair, we have
\begin{equation*}
    \gcd(q^{\varepsilon_2}A^\sharp B^\flat-q^{\varepsilon_1}C^\sharp  D^\flat,q^{\varepsilon_2}B^\sharp B^\flat-q^{\varepsilon_1}D^\flat D^\sharp ,q^{\varepsilon_2}A^\sharp A^\flat-q^{\varepsilon_1}C^\flat C^\sharp )\equiv_q d_F^{\sharp \sharp } \equiv_q d_F^{\flat\flat}.
\end{equation*}
\end{theorem}

The proof strategy goes as follows: we first prove that the equalities are invariant under transformations by $T$ and $S$. Second, we check them on the representatives given in Proposition \ref{Prop:represents-reg}.

\begin{proof}
We start by showing invariance of the identities under transformations $T$ and $S$. By Proposition \ref{Prop:invariance-q-Farey-det}, the Farey determinants are invariant under the $\PSL_2(\Z)$-action (up to a some power of $q$). 
\medskip
Consider a fraction $\tfrac{a}{b}\in\mathbb{Q}$. Under $T$, the parameter $\varepsilon=\varepsilon(\tfrac{a}{b})$ and the (left or right) quantization $\frac{A}{B}$ change into
$$\varepsilon \mapsto \left \{ \begin{array}{cl}
\varepsilon+1 & \text{ if } \tfrac{a}{b}> 0 \\
\varepsilon-1 & \text{ if } \tfrac{a}{b}< 0 \\
\varepsilon=0 & \text{ if } \tfrac{a}{b} = 0.\\
\end{array} \right.
\;\; \text{ and }\;\;\;
\frac{A}{B} \mapsto \left \{ \begin{array}{cl}
\frac{qA+B}{B} & \text{ if } \tfrac{a}{b}\geq 0 \\
\frac{A+B/q}{B/q} & \text{ if } \tfrac{a}{b}< 0 .
\end{array} \right.$$
The second part follows from $B^\sharp (0)=1$ if $\tfrac{a}{b}\geq 0$, $B^\sharp (0)=0$ if $\tfrac{a}{b}< 0$, $B^\flat(0)=1$ if $\tfrac{a}{b}> 0$ and $B^\flat(0)=0$ if $\tfrac{a}{b}\leq 0$. This also shows that for $\tfrac{a}{b}=0$, we are in the first case for the right version, and in the second case for the left version.

With these transformation rules, we prove that the two greatest common divisors do not change under $T$, modulo some power of $q$.
For this, we distinguish three cases.\\
~\\
\underline{First case:} $\tfrac{ac}{bd}>0$. We focus on the case when both $\frac{a}{b}$ and $\frac{c}{d}$ are positive, the negative case works exactly the same.\\
\noindent Consider the first term in the greatest common divisor for the pair $\left(T\cdot\frac{a}{b},T\cdot\frac{c}{d}\right)$:
\begin{align*}
q^{\varepsilon_2+ 1}(qA^{\sharp}+ B^{\sharp})B^{\flat} + q^{\varepsilon_1+ 1}(qC^{\flat} + D^{\flat})D^{\sharp} &= q^{2}(q^{\varepsilon_2}A^{\sharp}B^{\flat} + q^{\varepsilon_1}C^{\flat}D^\sharp) \\
&+ q(q^{\varepsilon_2}B^{\sharp}B^{\flat} + q^{\varepsilon_1}D^{\sharp}D^{\flat}).
\end{align*}
\noindent The second term in the gcd changes by a factor $q$ under $T$. The third one changes by
\begin{align*}
    &q^{\varepsilon_2+ 1}(qA^{\sharp} + B^{\sharp})(qA^{\flat} + B^{\flat}) + q^{\varepsilon_1+ 1}(qC^{\sharp} + D^{\sharp})(qC^{\flat} + D^{\flat})\\
    &= q^{3}(q^{\varepsilon_2}A^{\sharp}A^{\flat}+ q^{\varepsilon_1}C^{\sharp}C^{\flat}) + q(q^{\varepsilon_2}B^{\sharp}B^\flat + q^{\varepsilon_1}D^{\sharp}D^{\flat}) \\
    &+ q^{2}(q^{\varepsilon_2}A^{\sharp}B^{\flat} + q^{\varepsilon_1}C^{\flat}D^{\sharp}) + q^2(q^{\varepsilon_2}A^{\flat}B^{\sharp} + C^{\sharp}D^{\flat}).
\end{align*}
\noindent Then taking the greatest common divisor between these three terms, we get
$$
q\gcd(q(q^{\varepsilon_2}A^{\sharp}B^{\flat} + q^{\varepsilon_1}C^{\flat}D^\sharp),q^{\varepsilon_2}B^{\sharp}B^{\flat} + q^{\varepsilon_1}D^{\sharp}D^{\flat},q^{2}(q^{\varepsilon_2}A^{\sharp}A^{\flat}+ q^{\varepsilon_1}C^{\sharp}C^{\flat}) + q(q^{\varepsilon_2}A^{\flat}B^{\sharp} + C^{\sharp}D^{\flat})).
$$
\noindent But according to Remark \ref{rem:sym_sharp_and_flat}, $q^{\varepsilon_2}A^{\flat}B^{\sharp} + q^{\varepsilon_1}C^{\sharp}D^{\flat} = q^{\varepsilon_2}A^{\sharp}B^{\flat} + q^{\varepsilon_1}C^{\flat}D^{\sharp}$, so in the end, the greatest common divisor of the pair changed by $T$ is the same (up to a factor $q$) as the one for the initial pair.\\
~\\
\underline{Second case:} $\tfrac{ac}{bd}<0$. We focus on the case when $\frac{a}{b}>0$ and $\frac{c}{d}<0$, the symmetric case is similar. \\
\noindent The first term in the greatest common divisor for the pair shifted by $T$ is 
\begin{align*}
q^{\varepsilon_2- 1}(qA^{\sharp}+ B^{\sharp})B^{\flat} + q^{\varepsilon_1+1}(C^{\flat} + q^{-1}D^{\flat})q^{-1}D^{\sharp} &= (q^{\varepsilon_2}A^{\sharp}B^{\flat} + q^{\varepsilon_1}C^{\flat}D^\sharp) \\
&+ q^{-1}(q^{\varepsilon_2}B^{\sharp}B^{\flat} + q^{\varepsilon_1}D^{\sharp}D^{\flat}).
\end{align*}
\noindent The second term changes by $q^{-1}$ under $T$. The third one changes by 
\begin{align*}
    &q^{\varepsilon_2- 1}(qA^{\sharp} + B^{\sharp})(qA^{\flat} + B^{\flat}) + q^{\varepsilon_1+ 1}(C^{\sharp} + q^{-1}D^{\sharp})(C^{\flat} + q^{-1}D^{\flat})\\
    &= q(q^{\varepsilon_2}A^{\sharp}A^{\flat}+ q^{\varepsilon_1}C^{\sharp}C^{\flat}) + q^{-1}(q^{\varepsilon_2}B^{\sharp}B^\flat + q^{\varepsilon_1}D^{\sharp}D^{\flat}) \\
    &+ 2(q^{\varepsilon_2}A^{\sharp}B^{\flat} + q^{\varepsilon_1}C^{\flat}D^{\sharp}).\\
\end{align*}
\noindent Therefore the greatest common divisor for the pair shifted by $T$ is the same as the gcd for the initial pair.\\ 
~\\
\underline{Third case:} $\tfrac{ac}{bd}=0$. Here we can explicitly compute the transformation using $\tfrac{a}{b}=0$. \\
~\\
A similar argument holds for the transformation $S$.
Under $S$, the parameter $\varepsilon=\varepsilon(\tfrac{a}{b})$ and the quantization $\frac{A}{B}$ change into
$$\varepsilon \mapsto \left \{ \begin{array}{cl}
\varepsilon+1 & \text{ if } \tfrac{a}{b}> 0 \\
\varepsilon-1 & \text{ if } \tfrac{a}{b}< 0 \\
\varepsilon=0 & \text{ if } \tfrac{a}{b} = 0 \text{ (by convention)} .\\
\end{array} \right.
\;\; \text{ and }\;\;\;
\frac{A}{B} \mapsto \left \{ \begin{array}{cl}
\frac{-B}{qA} & \text{ if } \tfrac{a}{b}> 0 \\
\frac{-B/q}{A} & \text{ if } \tfrac{a}{b}\leq 0 \\
\end{array} \right.$$
The evolution of $\varepsilon$ can be deduced from the direct computation of the $q$-diameter $\ell_q\left(\frac{-b}{a}\right)$ in terms of $\ell_q\left(\frac{a}{b}\right)$, see Proposition \ref{prop:qradius}.

The same distinction of cases shows invariance of the two greatest common divisors.

\medskip
To finish the proof, we have to check the identities on some representatives of the equivalence classes of regular pairs under the modular group action. Such representatives are given in Proposition \ref{Prop:represents-reg}.

Consider first the case of inner regular pairs. A representative is $\left(\tfrac{1}{0},\tfrac{k}{n}\right)$, where $n\mid k^2+1$. Since $[\tfrac{1}{0}]^\sharp =\tfrac{1}{0}$, $[\tfrac{1}{0}]^\flat=\tfrac{1}{1-q}$ and $\varepsilon(\infty)=0$, we get 
$$d_F^{\sharp \flat} = N^\flat \;\;\text{ and }\;\; d_F^{\flat\sharp } = N^\sharp -(1-q)K^\sharp .$$
By Proposition \ref{Prop:id-flat-sharp}, we see that these two Farey determinants are equal, up to some power of $q$. 
It remains to compute the greatest common divisor
$$G_1=\gcd(q^{\varepsilon_2}(1-q)+K^\flat N^\sharp ,N^\sharp  N^\flat, q^{\varepsilon_2}+K^\sharp  K^\flat).$$
Since $q^{\varepsilon_2}(1-q)=K^\sharp  N^\flat-K^\flat N^\sharp $, the first term of $G_1$ equals $K^\sharp  N^\flat$. The third term of $G_1$ equals
\begin{align*}
    q^{\varepsilon_2}+K^\sharp  K^\flat &= \frac{K^\sharp  N^\flat-K^\flat N^\sharp }{1-q} +K^\sharp  K^\flat\\
    &= \frac{K^\flat((1-q)K^\sharp -N^\sharp )+K^\sharp N\flat}{1-q} \\
    &= \frac{-q^\alpha K^\flat N^\flat+K^\sharp N\flat}{1-q},
\end{align*}
where we used Proposition \ref{Prop:id-flat-sharp} in the last line. 
Since $N^\flat(1)\neq 0$, we have $\gcd(N^\flat,q-1)=1$. Therefore, we see that $N^\flat$ divides $G_1$. Since $\gcd(K^\sharp ,N^\sharp )=1$, we finally get that
$$G_1 \equiv_q N^\flat = q^\alpha d_F^{\sharp \flat}$$

Consider now the case of outer regular pairs. A representative is $\left(\tfrac{1}{0},\tfrac{k}{n}\right)$, where $n\mid k^2-1$. We then get
$$d_F^{\sharp \sharp } = N^\sharp  \;\;\text{ and }\;\; d_F^{\flat\flat} = \frac{N^\flat-(1-q)K^\flat}{q^2-q+1}.$$
By Proposition \ref{Prop:id-flat-sharp}, we see that $d_F^{\sharp \sharp }=d_F^{\flat\flat}$, up to some power of $q$. 
It remains to compute the greatest common divisor
$$G_2=\gcd(q^{\varepsilon_2}(1-q)-K^\sharp  N^\flat,N^\sharp  N^\flat, q^{\varepsilon_2}-K^\sharp  K^\flat).$$
Again, the first term of $G_2$ equals $-K^\flat N^\sharp $ and the third term equals
\begin{align*}
    q^{\varepsilon_2}-K^\sharp  K^\flat &= \frac{K^\sharp ((q-1)K^\flat+N^\flat)-K^\flat N^\sharp }{1-q} \\
    &= \frac{q^\alpha(q^2-q+1) K^\sharp  N^\sharp -K^\flat N^\sharp }{1-q},
\end{align*}
where we used Proposition \ref{Prop:id-flat-sharp}. Since $N^\sharp (1)\neq 0$, we have $\gcd(N^\sharp ,q-1)=1$, hence $N^\sharp $ divides $G_2$. Since $\gcd(K^\flat,N^\flat)=1$, we finally get
$$G_2=q^\beta N^\sharp \equiv_q d_F^{\sharp \sharp },$$
which concludes the proof.
\end{proof}

Combining Theorem \ref{Thm:main} with Proposition \ref{prop:homotheticformulas} and Theorem \ref{Thm:q-gcd}, we get explicit expressions for the $q$-rationals under Springborn operations. Apart from the $\mathrm{PSL}_2(\Z)$-structure of $q$-rationals (notably expressed via Farey sum and differences), this is the first instance of explicit formulas of this kind.

\begin{coro}\label{coro:reduced_formulas}
Let $(\tfrac{a}{b}, \tfrac{c}{d})$ be an inner regular pair. Put $\tfrac{k}{n}=\tfrac{a}{b}\oplus_S \tfrac{c}{d}$, i.e. $k=\frac{ab+cd}{ad-bc}$ and $n=\frac{b^2+d^2}{ad-bc}$. Then
\begin{align}\label{Eq:numerator-denominator-ids}
    K^\sharp &\equiv_q \frac{q^{\varepsilon_2}A^{\sharp}B^{\flat} + q^{\varepsilon_1}C^{\flat}D^{\sharp}}{d_F^{\sharp\flat}}, \\
    N^\sharp &\equiv_q \frac{q^{\varepsilon_2}B^{\sharp}B^{\flat} + q^{\varepsilon_1}D^{\flat}D^{\sharp}}{d_F^{\sharp\flat}}.
\end{align}
A similar formula holds for an outer regular pair.
\end{coro}

\begin{proof}
Let us denote by $\Delta := \gcd(q^{\varepsilon_2}A^\sharp B^\flat + q^{\varepsilon_1}C^\flat D^\sharp, q^{\varepsilon_1}B^\sharp B^\flat + q^{\varepsilon_1}D^\sharp D^\flat)$. By Theorem~\ref{Thm:q-gcd}, we have $d_F^{\sharp,\flat} \mid \Delta$, up to powers of $q$. We will check that actually $\Delta \equiv_q d_F^{\sharp\flat}$. This follows from the $q$-version of the last identity~\eqref{eq:technicalrelations}:
\begin{align*}
    &(q^{\varepsilon_2}A^\sharp A^\flat + q^{\varepsilon_1}C^\sharp C^\flat)(q^{\varepsilon_1}B^\sharp B^\flat + q^{\varepsilon_1}D^\sharp D^\flat) \\ &= (q^{\varepsilon_2}A^\sharp B^\flat + q^{\varepsilon_1}C^\flat D^\sharp)(q^{\varepsilon_2}A^\flat B^\sharp + q^{\varepsilon_1}C^\sharp D^\flat) + q^{\varepsilon_1+\varepsilon_2}d_F^{\sharp\flat}d_F^{\flat\sharp}, 
\end{align*}
    which, using that there exists $k$ such that $d_F^{\flat\sharp} = q^{k}d_F^{\sharp\flat}$ and that $X^\sharp Y^\flat - X^\flat Y^\flat = q^{\varepsilon}(1-q)$, can be written also as 
    \[
    (q^{\varepsilon_2}A^\sharp A^\flat + q^{\varepsilon_1}C^\sharp C^\flat)(q^{\varepsilon_1}B^\sharp B^\flat + q^{\varepsilon_1}D^\sharp D^\flat) = (q^{\varepsilon_2}A^\sharp B^\flat + q^{\varepsilon_1}C^\flat D^\sharp)^2 + q^{\varepsilon_1+\varepsilon_2+k}(d_F^{\sharp\flat})^2. 
    \]
    \noindent We can divide by $(d_F^{\sharp\flat})^2$ in this equality, and we conclude that $\Delta/d_F^{\sharp\flat}$ must be a power of $q$. \\
    ~\\
    \noindent For the outer case, let us denote again by $\Delta := \gcd(q^{\varepsilon_2}A^\sharp B^\flat - q^{\varepsilon_1}C^\sharp D^\flat, q^{\varepsilon_1}B^\sharp B^\flat - q^{\varepsilon_1}D^\sharp D^\flat)$. By similar arguments, we get that  $\Delta/d_F^{\sharp\sharp}$ must divide $q^2-q+1$, up to powers of $q$. Let us consider then the $q$-version of the first relation in~\eqref{eq:technicalrelations}: 
    \[
    B^\sharp\frac{(q^{\varepsilon_2}A^\sharp B^\flat - q^{\varepsilon_1}C^\sharp D^\flat)}{d_F^{\sharp\sharp}} = A^\sharp\frac{q^{\varepsilon_1}B^\sharp B^\flat - q^{\varepsilon_1}D^\sharp D^\flat}{d_F^{\sharp\sharp}} + q^{\varepsilon_1}D^\flat.
    \]
    \noindent If $\Delta/d_F^{\sharp\sharp} = q^2-q+1$, then $D^\flat$ is divisible by $q^2-q+1$. The third point property in Lemma~\ref{Lem:special-values} shows that this is not possible. 
\end{proof}

\subsection{Application: mid-point formula}

A first application of our main Theorem \ref{Thm:main} is an explicit formula for the midpoint between two rationals of Farey determinant 1:
\begin{theorem}\label{Coro:q-midpoint}
Consider two rational numbers $0 < \tfrac{a}{b} < \tfrac{c}{d}$ of Farey determinant 1. Then the arithmetic mean of $\frac{a}{b}$ and $\frac{c}{d}$ has left $q$-deformation 
\begin{equation}\label{Eq:midpoint-dF-1}
\left[\frac{1}{2}\left(\frac{a}{b} + \frac{c}{d}\right)\right]_q^\flat = \frac{A^\sharp D^\flat + q^{\varepsilon}B^\flat C^\sharp}{B^\sharp D^\flat + q^{\varepsilon}B^\flat D^\sharp},    
\end{equation}
\noindent where $\varepsilon = \abs{\varepsilon(\tfrac{a}{b}) - \varepsilon(\tfrac{c}{d})}+1$ is the difference between the integers associated to $\frac{a}{b}$ and $\frac{c}{d}$ as in \eqref{defi:epsilon}. This expression is in reduced form. A formula for the right $q$-deformation of the arithmetic mean can be deduced, using the transition map \ref{Prop:duality}.  
\end{theorem}

\begin{proof}
Combining Proposition \ref{prop:algebraic_springborn}, Lemma \ref{Lemma:dF-for-Farey-op} and Theorem \ref{Thm:main}, we get
\begin{equation}\label{Eq:midpt-e}
\left[\frac{1}{2}\left(\frac{a}{b}+\frac{c}{d}\right)\right]_q^\flat = \left[\left(\frac{a+c}{b+d}\right)\ominus_S\left(\frac{a-c}{b-d}\right)\right]_q^\flat=e\left(\left[\frac{a}{b}\oplus_F\frac{c}{d}\right],\left[\frac{a}{b}\ominus_F\frac{c}{d}\right]\right). 
\end{equation}
\noindent Let $\frac{a+c}{b+d} = [a_1,\cdots,a_{2m}] = [\![c_1,\cdots,c_k]\!]$ be the even positive and negative continued fractions of the Farey sum of $\frac{a}{b}$ and $\frac{c}{d}$. We distinguish cases according to the relationship between $\frac{a}{b}$ and $\frac{c}{d}$.\\
~\\
Consider first a case when $\frac{a}{b}$ is a child of $\frac{c}{d}$. The situation in the left and right $q$-Farey trees is as follows. 
\begin{center}
\begin{tikzpicture}
\begin{scope}
\vspace{-1cm}
	\draw (0,0) arc (0:180:2) node[above,midway] {\scriptsize{$q^{a_{2m-1}-1}$}};
	\draw (0,0) node[below] {$\frac{C^\flat}{D^\flat}$};
	\draw (-4,0) node[below] {$[\tfrac{a-c}{b-d}]^\flat$};
	\draw (0,0) arc (0:180:1) node[above,midway] {\scriptsize{$1$}};
	\draw (-2,0) arc (0:180:1) node[above,midway] {\footnotesize{$q^{a_{2m-1}}$}};
	\draw (-2,0) node[below] {$\frac{A^\flat}{B^\flat}$};
    \draw (0,0) arc (0:180:0.5) node[above,midway] {\scriptsize{$1$}};
    \draw (-1,0) arc (0:180:0.5) node[above,midway] {\scriptsize{$q$}};
    \draw (-1,0) node[below] {$[\tfrac{a+c}{b+d}]^\flat$};
\end{scope}
\begin{scope}[shift={(6,0)}]
\vspace{-1cm}
	\draw (0,0) arc (0:180:2) node[above,midway] {\scriptsize{$q^{c_k-3}$}};
	\draw (0,0) node[below] {$\frac{C^\sharp}{D^\sharp}$};
	\draw (-4,0) node[below] {$[\tfrac{a-c}{b-d}]^\sharp$};
	\draw (0,0) arc (0:180:1) node[above,midway] {\scriptsize{$q^{c_k-2}$}};
	\draw (-2,0) arc (0:180:1) node[above,midway] {\scriptsize{$1$}};
	\draw (-2,0) node[below] {$\frac{A^\sharp}{B^\sharp}$};
    \draw (0,0) arc (0:180:0.5) node[below,midway] {\scriptsize{$q^{c_k-1}$}};
    \draw (-1,0) arc (0:180:0.5) node[below,midway] {\scriptsize{$1$}};
    \draw (-1,0) node[below] {$[\tfrac{a+c}{b+d}]^\sharp$};
\end{scope}
\end{tikzpicture}
\end{center}
\noindent Therefore, $\varepsilon(\tfrac{a+c}{b+d}) = \varepsilon(\tfrac{a}{b}) +1$, and $\varepsilon(\tfrac{a-c}{b-d}) = \varepsilon(\tfrac{a}{b}) - a_{2m-1}$. Then

\begin{align*}
    \left[\frac{a+c}{b+d}\right]^{\sharp} = \frac{A^\sharp + q^{c_k-1}C^\sharp}{B^\sharp + q^{c_k-1}D^\sharp} &\text{ , }& \left[\frac{a+c}{b+d}\right]^{\flat} = \frac{qA^\flat + C^\flat}{qB^\flat + D^\flat}, \\
\left[\frac{a-c}{b-d}\right]^{\sharp} = \frac{A^\sharp - q^{c_k-2}C^\sharp}{B^\sharp - q^{c_k-2}D^\sharp} &\text{ , }&  \left[\frac{a-c}{b-d}\right]^{\flat} = \frac{q^{-a_{2m-1}}(A^\flat - C^\flat)}{q^{-a_{2m-1}}(B^\flat - D^\flat)}.
\end{align*}

\noindent Denote $\varepsilon_1 = \varepsilon(\tfrac{a-c}{b-d})$ and $\varepsilon_2 = \varepsilon(\tfrac{a+c}{b+d})$. Then by Proposition \ref{prop:homotheticformulas}, 
$$
e\left(\left[\frac{a\!+\!c}{b\!+\!d}\right],\left[\frac{a\!-\!c}{b\!-\!d}\right]\right) = \frac{q^{\varepsilon_1}(A^\sharp \!+\! q^{c_k-1}C^\sharp)(qB^\flat \!+\! D^\flat) - q^{\varepsilon_2-a_{2m-1}}(A^\sharp - q^{c_k-2}C^\sharp)(B^\flat \!-\! D^\flat)}{q^{\varepsilon_1}(B^\sharp \!+\! q^{c_k-1}D^\sharp)(qB^\flat \!+\! D^\flat)\! -\! q^{\varepsilon_2-a_{2m-1}}(B^\sharp \!-\! q^{c_k-2}D^\sharp)(B^\flat \!-\! D^\flat)}.
$$
\smallskip

Using $\varepsilon_1 + 1 = \varepsilon_2 - a_{2m-1}$ and $\varepsilon_1 + c_k-1 = \varepsilon_2 + c_k-2 -a_{2m-1}$, the formula simplifies:
$$
e\left(\left[\frac{a\!+\!c}{b\!+\!d}\right],\left[\frac{a\!-\!c}{b\!-\!d}\right]\right) = \frac{q^{\varepsilon_1}(1+q)A^\sharp D^\flat + q^{\varepsilon_1 + c_k-1}(1+q)C^\sharp B^\flat}{q^{\varepsilon_1}(1+q)B^\sharp D^\flat + q^{\varepsilon_1 + c_k-1}(1+q)D^\sharp B^\flat},
$$
\noindent which concludes, since $\varepsilon(\tfrac{a}{b}) = \varepsilon(\tfrac{c}{d})+c_k-2$.\\
~\\
\noindent The case where $\frac{a}{b}$ is the left parent of $\frac{c}{d}$ is similar, with $\frac{a}{b} < \frac{a+c}{b+d} < \frac{c}{d} < \frac{a-c}{b-d}$.
\end{proof}

\begin{Remark}
Equation \eqref{Eq:midpt-e} also holds for an outer regular pair $(\tfrac{a}{b},\tfrac{c}{d})$, since in that case $d_F(\tfrac{a+c}{b+d},\tfrac{a-c}{b-d})\in\{1,2\}$ (see Remark \ref{Rk:reg-characterizations}). This means that there exists an explicit formula for $\left[\frac{1}{2}\left(\frac{a}{b} + \frac{c}{d}\right)\right]_q^\flat$ when $(\tfrac{a}{b},\tfrac{c}{d})$ is outer regular, but the general expression is complicated.
\end{Remark}

\begin{question}
Is there a combinatorial interpretation of the identities \eqref{Eq:numerator-denominator-ids} -- \eqref{Eq:midpoint-dF-1}?
\end{question}

In the final section, we will address this question in the particular case of Markov fractions.

\subsection{Application: size of $q$-rationals}

Another application of our main Theorem \ref{Thm:main} shows that the diameters of $q$-rationals, given by the jump function $\ell_q(\tfrac{a}{b})$ from Equation \eqref{Eq:diam-def}, decrease with the Farey sum of Farey neighbors:

\begin{theorem}\label{Thm:diam-decrease}
Let $\tfrac{a}{b}<\tfrac{c}{d}$ be such that $d_F := \abs{ad-bc} = 1$. Then
$$\ell_q\left(\frac{a+c}{b+d}\right) < \min \left(\ell_q\left(\frac{a}{b}\right), \ell_q\left(\frac{c}{d}\right)\right).$$
\end{theorem}

\begin{proof}
Since $T_q$ scales all diameters by $q$, we can use it to suppose $0<\tfrac{a}{b}<\tfrac{c}{d}$. We then verify that
$$\frac{a}{b}\ominus_S \frac{a+c}{b+d} > \frac{a+c}{b+d}.$$
Indeed, we have 
\begin{align*}
    \frac{a}{b}\ominus_S \frac{a+c}{b+d} &= \frac{(a+c)(b+d) - ab}{(b+d)^2 - b^2} > \frac{a+c}{b+d}\\
    &\Leftrightarrow (b+d)((a+c)(b+d) - ab) > (a+c)((b+d)^2 - b^2) \\
    &\Leftrightarrow bc-ad > 0,
\end{align*}
which is true since $\tfrac{a}{b} < \tfrac{c}{d}$.

Since the pair $\tfrac{a}{b},\tfrac{a+c}{b+d}$ is outer-regular, we can apply Theorem \ref{Thm:main} to get
$$e\left(\left[\frac{a}{b}\right],\left[\frac{a+c}{b+d}\right]\right) = \left[\frac{a}{b}\ominus_S \frac{a+c}{b+d}\right]_q^\flat > \left[\frac{a+c}{b+d}\right]_q^\sharp,$$ since the $q$-deformation preserves the order on rationals, see Corollary \ref{Coro:well-orderedness}. 
$$\ell_q\left(\frac{a+c}{b+d}\right) < \ell_q\left(\frac{a}{b}\right).$$
Similarly, we can verify that $\tfrac{c}{d}\ominus_S \tfrac{a+c}{b+d}<\tfrac{a+c}{b+d}$, which gives the second inequality.
\end{proof}
\begin{figure}[h!]
    \centering
    \includegraphics[height=4cm]{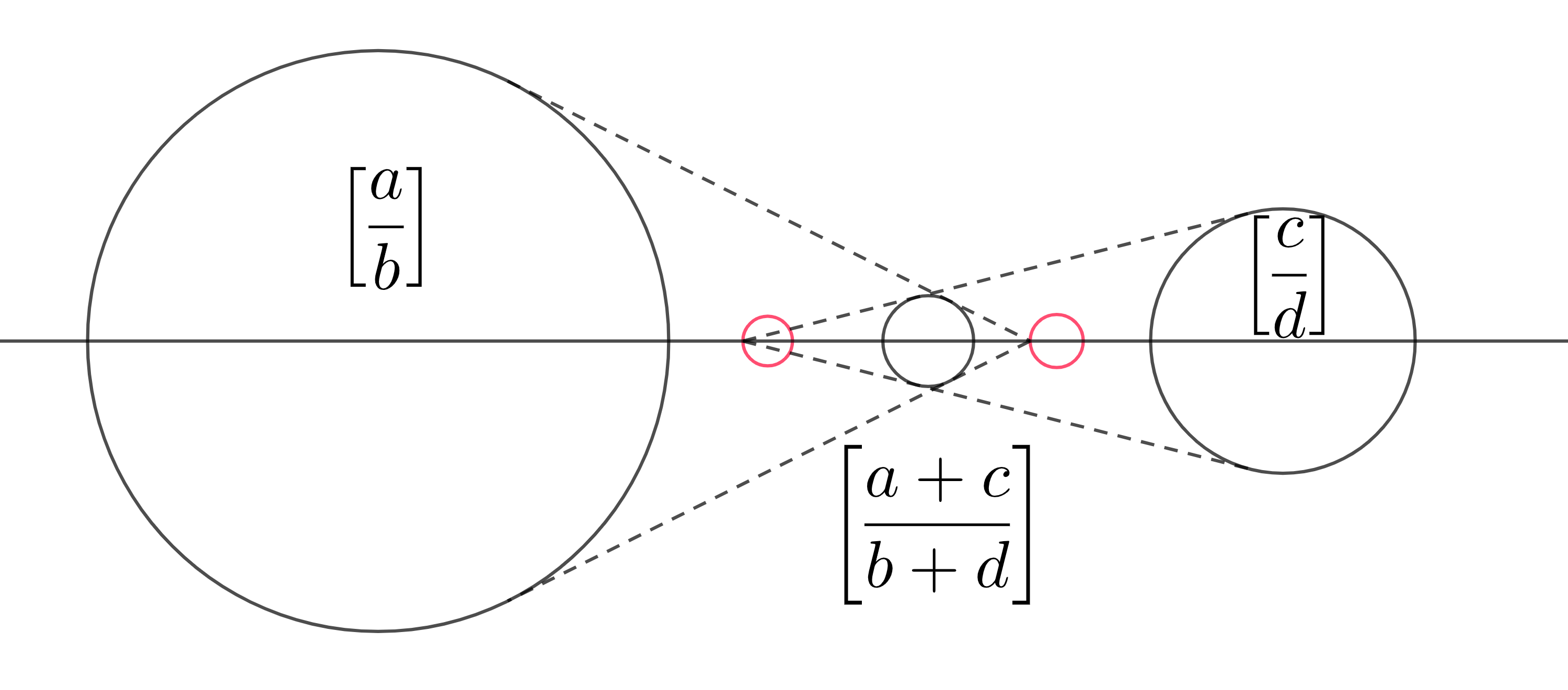}
    \caption{The diameter, or jump, of $q$-rationals decreases with the Farey sum.}
    \label{Fig:diam-decrease}
\end{figure}

\begin{Remark}
It is not true that the diameter is decreasing with the depth in the Farey tree, even for rationals in $[0,1]$: $\tfrac{1}{6}$ is deeper in the Farey tree than $\tfrac{3}{8}$, but for $q$ close to 1 we have
$$\ell_q\left(\frac{1}{6}\right) \underset{q\to 1}\sim \frac{\lvert 1-q\rvert}{36} >\frac{\lvert 1-q\rvert}{64} \underset{q\to 1}\sim \ell_q\left(\frac{3}{8}\right).$$
It is not true neither that the diameter decreases with increasing denominator: for $q$ close to 0, we have
$$\ell_q\left(\frac{1}{6}\right) \underset{q\to 0}\sim q^5 < q^4 \underset{q\to 0}\sim \ell_q\left(\frac{3}{8}\right).$$
\end{Remark}

As a consequence of the decreasing diameter of $q$-rationals, we can show that all $q$-rationals lie in a certain cone:
\begin{coro}
Let $K_q$ be the cone between the common outer tangents to $[0]$ and $[1]$. Then $[x]\subset K_q$ for all $x\in \R$, see Figure \ref{Fig:q-cone}.
\end{coro}

\begin{figure}[h!]
    \centering
    \includegraphics[height=4cm]{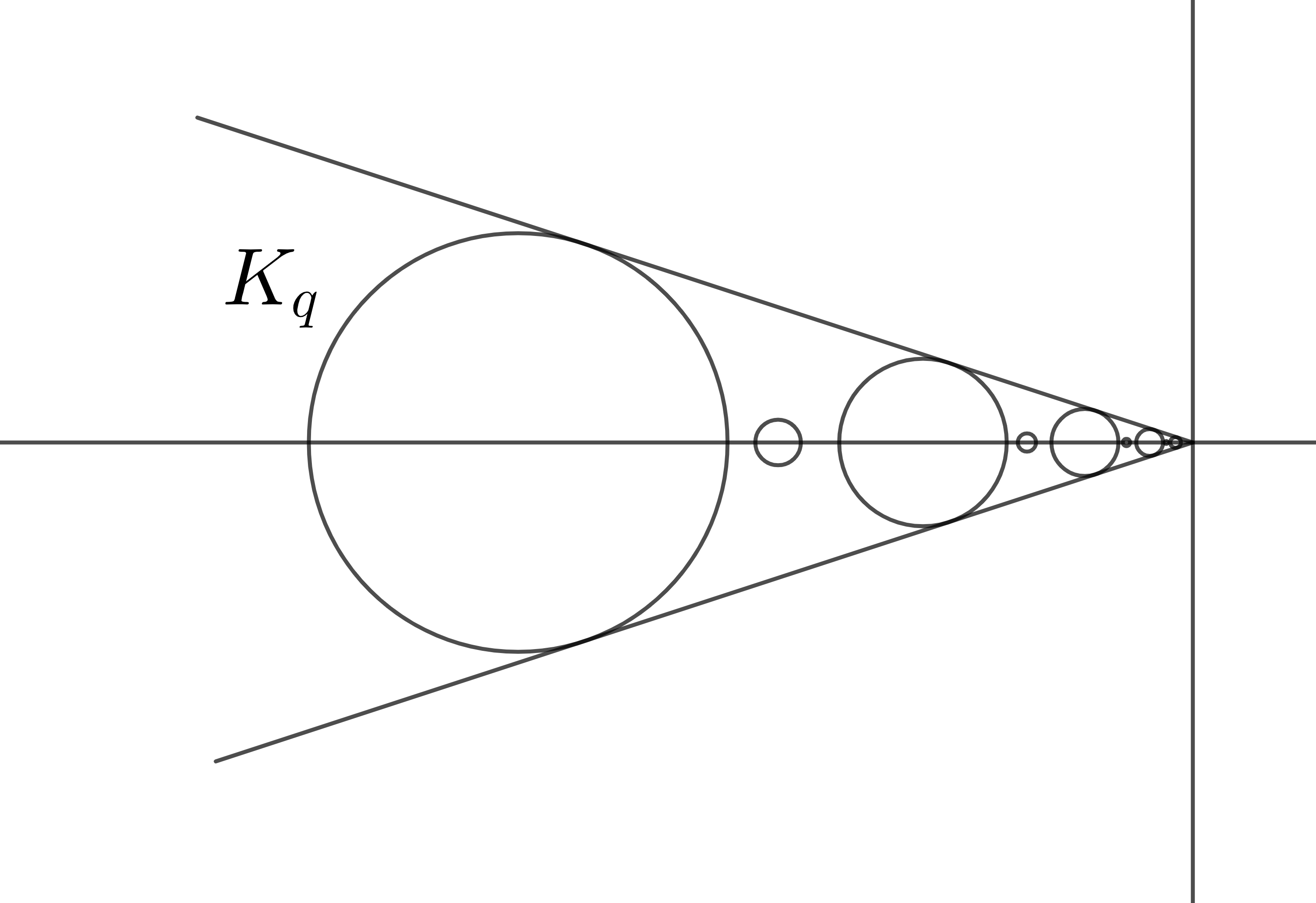}
    \caption{All $q$-rationals lie in the same cone $K_q$, spanned by the common outer tangents to all integers.}
    \label{Fig:q-cone}
\end{figure}

\begin{proof}
Consider two rationals $\tfrac{a}{b}< \tfrac{c}{d}$ of Farey distance 1, and denote by $Z([\tfrac{a}{b}],[\tfrac{c}{d}])$ the zone delimited between the two circles $[\tfrac{a}{b}], [\tfrac{c}{d}]$ and their common outer tangents, see Figure \ref{Fig_zoneZ}. From Theorem \ref{Thm:diam-decrease}, we know that
$$\left[\frac{a+c}{b+d}\right]\subset Z\left(\left[\frac{a}{b}\right],\left[\frac{c}{d}\right]\right).$$
In addition, we have $Z([\tfrac{a}{b}],[\tfrac{a+c}{b+d}])\subset Z([\tfrac{a}{b}],[\tfrac{c}{d}])$ and $Z([\tfrac{a+c}{b+d}],[\tfrac{c}{d}])\subset Z([\tfrac{a}{b}],[\tfrac{c}{d}])$.
The Corollary than follows by an induction in the depth of the Farey tree, with initialisation at the integers $[n]$, which are all inside (and tangent) to the cone $K_q$.
\end{proof}

\begin{figure}[h!]
    \centering
    \includegraphics[height=4cm]{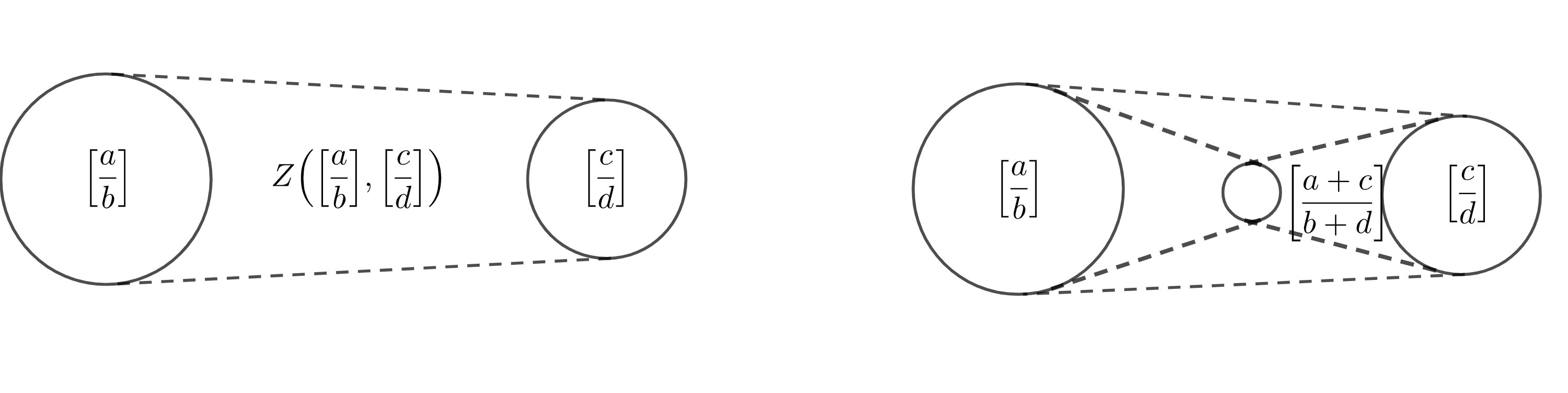}
    \caption{Zone $Z([\tfrac{a}{b}],[\tfrac{c}{d}])$ delimited by two circles and their outer tangents (left) and inclusion of these zones for Farey addition (right).}
    \label{Fig_zoneZ}
\end{figure}

\section{An example : Markov fractions}\label{Sec:Markov}

Let us consider Markov fractions, defined by Springborn in \cite{Springborn}. They form inner regular pairs, so they satisfy Theorem \ref{Thm:main}. This enables us to show that the $q$-deformations of rational Markov triples are solutions of a $q$-deformed Markov equation. We give an additional combinatorial interpretation of Theorem \ref{Thm:main} in this context. 

\subsection{Definition : iterating the Springborn addition}\label{subs:iterating the Springborn sum}

Markov fractions are defined in a recursive tree, starting with the inital two rationals $0$ and $\frac{1}{2}$, and iterating the Springborn addition. A given Markov fraction $x_1$ has thus two parents $x_0$ and $x_2$, and the triple $(x_0,x_1,x_2)$ with $x_1 = x_0\oplus_S x_2$ is called a rational Markov triple. In the rational Markov tree, a vertex represents a rational Markov triple and each zone is labelled by a Markov fraction.

\begin{figure}[h!]
\centering
\begin{tikzpicture}[scale=2]
    \draw (0,0)--(0,-0.7) node[midway,sloped] {$<$};
    \draw (0,0)--(1/2, {sqrt(3)/2}) node[midway,sloped] {$>$};
    \draw (0,0) -- (-1/2,{sqrt(3)/2}) node[midway,sloped] {$<$};
  
    \node at (1/2,0) 
    {\scalebox{1.3}{$\displaystyle \frac{0}{1}$}};
    \node at (-1/2,0) {\scalebox{1.3}{$\displaystyle \frac{1}{2}$}};
    \node at (0,0.8) {\scalebox{1.2}{$\displaystyle \frac{2}{5}$}};
    
    \draw (1/2, {sqrt(3)/2}) -- (1/2+0.8, {sqrt(3)/2-0.2}) node[midway,sloped] {$>$};
    \draw (1/2, {sqrt(3)/2}) -- (1/2, {sqrt(3)/2+0.75}) node[midway,sloped] {$>$};
      
    \node at ({1/2+0.3},{sqrt(3)/2+0.3}) {\scalebox{1.1}{$\displaystyle \frac{5}{13}$}};
    
    \draw (1.3, {sqrt(3)/2-0.2}) -- (1.3, {sqrt(3)/2-0.6}) node[midway,sloped] {$>$};
    \draw (1.3, {sqrt(3)/2-0.2}) -- (1.6, {sqrt(3)/2+0.1}) node[midway,sloped] {$>$};
    \node at (1.5, {sqrt(3)/2-0.4}){$\frac{13}{34}$};
    
   \draw (1/2, {sqrt(3)/2+0.75}) -- (1/2+0.4, {sqrt(3)/2+0.75}) node[midway,sloped] {$>$};
   \draw (1/2, {sqrt(3)/2+0.75}) -- (1/2-0.2, {sqrt(3)/2+1}) node[midway,sloped] {$<$};
    \node at (0.6,{sqrt(3)/2+1} ){$\frac{75}{194}$};
    
    \draw (-1/2,{sqrt(3)/2}) -- ({-1/2-0.8}, {sqrt(3)/2-0.2}) node[midway,sloped] {$<$};
    \draw (-1/2,{sqrt(3)/2}) -- (-1/2, {sqrt(3)/2+0.75}) node[midway,sloped] {$>$};
    \node at ({-1/2-0.3},{sqrt(3)/2+0.3}) 
    {\scalebox{1.1}{$\displaystyle \frac{12}{29}$}};

     \draw (-1.3, {sqrt(3)/2-0.2}) -- (-1.3, {sqrt(3)/2-0.6}) node[midway,sloped] {$>$};
    \draw (-1.3, {sqrt(3)/2-0.2}) -- (-1.6, {sqrt(3)/2+0.1}) node[midway,sloped] {$<$};
    \node at (-1.5, {sqrt(3)/2-0.4}){$\frac{70}{169}$};
    
     \draw (-1/2, {sqrt(3)/2+0.75}) -- (-1/2-0.4, {sqrt(3)/2+0.75}) node[midway,sloped] {$<$};
   \draw (-1/2, {sqrt(3)/2+0.75}) -- (-1/2+0.2, {sqrt(3)/2+1}) node[midway,sloped] {$>$};
     \node at (-0.6,{sqrt(3)/2+1} ){$\frac{179}{433}$};
\end{tikzpicture}
\label{fig:Springborn_tree}
\caption{The (oriented) rational Markov tree with the Springborn local rule, starting with $\frac{0}{1}$ and $\frac{1}{2}$. Every vertex of such tree corresponds to a triple of Markov fractions - two parents and a child. A couple of parents $\left(\frac{a}{b}, \frac{c}{d}\right)$ has one child $\frac{r}{s}$ with $r:=\frac{ab+cd}{ad-bc}$ and $s:=\frac{b^2+d^2}{ad-bc}$.}
\end{figure}

\begin{lemma}
Any two rational numbers in a rational Markov triple form an inner regular pair.
\end{lemma}

\begin{proof}
The initial pair $\left(\frac{0}{1},\frac{1}{2}\right)$ is inner regular, and $\gcd(1,2) = 1$. Moreover, $1 \mid 0^2+1$ and $2\mid 1^2 + 1$. By iteration conditions determined in Proposition \ref{prop:iteration}, we have the result.
\end{proof}

Springborn shows that rational Markov triples satisfy some defining equations. For a rational Markov triple $ \left(\frac{a_0}{b_0},\frac{a_1}{b_1},\frac{a_2}{b_2}\right)$,
\begin{equation} \label{eq:markovsystem}
    \begin{cases}
    b_1^2 + b_2^2 + b_0^2 = 3b_1b_2b_0\\
    b_0 = b_1a_2 - a_1b_2\\
    b_2 = b_0a_1 - a_0b_1\
    \end{cases}.
\end{equation}

Markov fractions fit in the classical theory of Markov numbers via a particular choice of a system of Cohn matrices. For each integer $n\in \mathbb{Z}$, one can define a family of matrices $C^t(n)$, indexed by rational numbers between $0$ and $1$, such that $\mathrm{tr}(C^t(n)) = 3m^t$ is a Markov number (see \cite[Chapter 4]{Aigner_markov}).

\begin{proposition}
Let $t\in \Q\cap [0,1]$ and let $\frac{a^t}{m^t}$ be the Markov fraction associated to $t$. Then the second column of the Cohn matrix $C^t(3)$ is $\begin{pmatrix}
    m^t\\
    -a^t\\
\end{pmatrix}$. 
\end{proposition}
A similar observation has been made by Veselov \cite[Theorem 3.1]{Veselov2}.

\begin{proof}
We proceed by induction on the $3$rd Cohn tree. The first triple is:
$$
A(3) = \begin{pmatrix}
    3 & 1\\
    -1 & 0\\
\end{pmatrix} \, \text{ , } \, A(3)B(3) = \begin{pmatrix}
17 & 5 \\
-7 & -2 \\
\end{pmatrix} \, \text{ , } \, B(3) = \begin{pmatrix}
    7 & 2 \\
    -4 & -1\\
\end{pmatrix}.
$$
\noindent Let $(M_0,M_1,M_2)$ be a Cohn triple, write $M_i = \begin{pmatrix}
    d_i & b_i \\
    -c_i & -a_i\\
\end{pmatrix}$ for $i = 0,1,2$, so 
$$
M_1 = M_0M_2 = \begin{pmatrix}
    d_1 & d_0b_2 - b_0a_2 \\
    -c_1 & a_0a_2 - c_0b_2\\
\end{pmatrix}= 
\begin{pmatrix}
        d_1 & b_1\\
    -c_1 & -a_1
\end{pmatrix}.
$$

\noindent Suppose that $(\tfrac{a_0}{b_0},\tfrac{a_1}{b_1},\tfrac{a_2}{b_2})$ is a rational Markov triple. Consider the left child $M_3 = M_0M_1$. Its second column is 
$$
b_3 = d_0b_1 - b_0a_1 \, \text{ , }\,  -a_3 = -c_0b_1 + a_0a_1. 
$$
\noindent Let us show that $\frac{a_3}{b_3} = \frac{a_0}{b_0}\oplus_S \frac{a_1}{b_1}$. By design of the Cohn tree, $b_3$ is a Markov number, part of the triple $(b_0,b_3,b_1)$, so $b_3 = \frac{b_0^2 + b_1^2}{b_2}$. Since $(\tfrac{a_0}{b_0},\tfrac{a_1}{b_1},\tfrac{a_2}{b_2})$ is a rational Markov triple is satisfies \eqref{eq:markovsystem} so $b_2 = b_0a_1 - a_0b_1$. Then $b_3 = \frac{b_0^2+b_1^2}{a_0b_1-a_1b_0}$ is the denominator of $\frac{a_0}{b_0}\oplus_S \frac{a_1}{b_1}$. \\
\noindent We want to show $a_3 = \frac{a_0b_0 + a_1b_1}{b_2}$. 
\begin{align*}
a_0b_0 + a_1b_1 &= a_0(b_1a_2 - a_1b_2) + a_1b_1\\
                    &=  b_1(c_0b_2-a_1) - a_0a_1b_2 + a_1b_1\\
                    &= b_2(b_1c_0 - a_0a_1)\\
                    &= b_2a_3.
\end{align*} 
\noindent The right child can be handled with similar arguments.
\end{proof}

By considering the right and left $q$-deformations of Markov fractions, we obtain $q$-deformations of Markov numbers on denominators. We prove that these $q$-Markov numbers satisfy a $q$-version of the classical Markov equation $x^2+y^2+z^2=3xyz$. Note that our $q$-Markov numbers are very close to the ones of Aval, Labbé, Lapointe and Steiner in \cite{labbe_2022,labbe_2024,Aval_labbe} (the difference is that they consider the Cohn tree $C^t(2)$ while we have $C^t(3)$). Choosing to work with $C^t(3)$ makes the $q$-Markov numbers compatible with Springborn operations and allows us to apply Theorem~\ref{Thm:main}. 

Several other approaches to $q$-deform Markov numbers were explored in the literature, some based on $q$-rational numbers as in \cite{Kogiso_Markov,Oguz_markov,EJMGO}, and some in a completely different context in \cite{BJOMY_markov}.

\begin{theorem}
\label{thm:qdeformedequations}
If $\left(\tfrac{a_0}{b_0},\tfrac{a_1}{b_1},\tfrac{a_2}{b_2}\right)$ is a rational Markov triple, denote by $\frac{A^{\sharp}_i}{B^{\sharp}_i}$ the right quantized Markov fractions associated with it, and by $\frac{A_i^{\flat}}{B_i^{\flat}}$ the left one, for $i\in \{0,1,2\}$. Then 
\begin{equation}\label{eq:qMarkoveq}
    B_1^{\sharp} B_1^{\flat} + q^{\varepsilon_0+3}B_2^{\sharp}B_2^{\flat} + B_0^{\flat}(B_1^{\sharp}A_2^{\sharp} -q^3A_1^{\sharp}B_2^{\sharp}) = [3]_qB_1^{\sharp}B_2^{\sharp}B_0^{\flat}.
\end{equation}
\noindent where $\varepsilon_0=\varepsilon(\tfrac{a_0}{b_0})$ from Definition \ref{defi:epsilon}. Moreover, 
\[
\begin{array}{c c c c}
    B_0^{\sharp} \equiv_q d_F^{\sharp\sharp}\left(\frac{a_1}{b_1},\frac{a_2}{b_2}\right), & B_0^{\flat} \equiv_q  d_F^{\sharp\flat}\left(\frac{a_1}{b_1},\frac{a_2}{b_2}\right),  & B_2^{\sharp} \equiv_q d_F^{\sharp\sharp}\left(\frac{a_0}{b_0},\frac{a_1}{b_1}\right), & B_2^{\flat} \equiv_q d_F^{\flat\sharp}\left(\frac{a_0}{b_0},\frac{a_1}{b_1}\right).\\
    \end{array}
\]
\end{theorem}

We prove this theorem in Section \ref{sec:proofmarkov}.

\begin{Remark}
Equation \eqref{eq:qMarkoveq} above is a deformation of the classical Markov equation. Note that $B_1^\sharp A_2^\sharp - q^3A_1^\sharp B_2^\sharp $ is a deformation of $b_0$.
\end{Remark}

\begin{Remark}
For a fraction $0 < a/b < 1$, the numerators and denominators of its $q$-deformation are chosen such that they are coprime polynomials in $q$, and the denominator has constant coefficient $1$. For $\tfrac{a_0}{b_0} = 0$, one has to choose $A_0^{\flat} = 1 - q^{-1}$, $B_0^{\flat} = 1$ and $\varepsilon(0) = -1$ in order to make the equation above work.
\end{Remark}

\subsection{Counting in fence posets}

The proof of Theorem \ref{thm:qdeformedequations} relies on a combinatorial interpretation of $q$-Markov fractions, in terms of fence posets. 

\begin{notation}
Let $\frac{a}{b}$ be a Markov fraction, different from $\frac{0}{1}$ and $\frac{1}{2}$. By convention, its canonical continued fraction expansion is $\frac{a}{b} = [0,\alpha_1,\cdots,\alpha_n]$, with $n$ even and $\alpha_i \geq 1$ for all $i$. \\
\noindent The canonical continued fraction expansion of $\frac{0}{1}$ is $[0]$ and the one of $\frac{1}{2}$ is $[0,2]$. 
\end{notation}

\begin{lemma}\label{lemma:contfracsum}
Let $\left(\frac{a_0}{b_0},\frac{a_1}{b_1},\frac{a_2}{b_2}\right)$ be a rational Markov triple. Write the canonical continued fraction expansions of $\frac{a_0}{b_0}$ and $\frac{a_2}{b_2}$ as 
$$
\frac{a_0}{b_0} = [0,\alpha_1,\cdots,\alpha_n] \text{ and } \frac{a_2}{b_2} = [0,\beta_1,\cdots,\beta_m],
$$
\noindent Then 
\begin{itemize}
    \item $\alpha_1 = \alpha_n = \beta_1 = \beta_m = 2$,
    \item the sequences $(\alpha_1,\cdots,\alpha_n)$ and $(\beta_1,\cdots,\beta_m)$ are palindromic,
    \item the continued fraction expansion of the Springborn sum $\frac{a_1}{b_1}$ is given by 
$$
\frac{a_1}{b_1} = [0,\alpha_1,\cdots,\alpha_n,2,1,\beta_1-1,\beta_2,\cdots,\beta_m].
$$
\end{itemize}
If $\frac{a_2}{b_2}= \frac{1}{2}$, we have $\frac{a_1}{b_1} = [0,\alpha_1,\cdots,\alpha_n,2,2]$.
\end{lemma}

This has been independently observed by Veselov in \cite{Veselov2}.

\begin{proof}
Consider the triangular snake graph model for Markov fractions described by Springborn in Section 5 of \cite{Springborn}. This model can be interpreted as a synthesis between two different combinatorial situations :
\begin{itemize}
    \item[(i)] the snake graph model for Markov numbers, see for example Aigner's book \cite{Aigner_markov};
    \item[(ii)] the triangulation model for rational numbers, see Section 2 in \cite{MGO_Farey}.
\end{itemize}
Indeed, in the Springborn model, each Markov fraction corresponds to a snake graph built with triangles. When gluing pairs of adjacent triangles one get a usual snake graph made of squares (or tiles), which is exactly the snake graph usually associated to the Markov number in the denominator of the Markov fraction. See Figure \ref{fig:snakegraphs} below. On the other hand, in the Springborn triangular snake graph, vertices are labeled with fractions, computed using the Farey summation formula, exactly as in the polygon triangulation model in \cite{MGO_Farey}, see Figure \ref{fig:triangulatedpolygon}. Combining these two models, we can deduce the continued fractions of Markov fractions. 

\begin{figure}[h!]
    \centering
    \begin{tikzpicture}[scale=0.8]
    \begin{scope}
    \draw (0,0)--(0.5,0.87)--(1.5,0.87)--(2,1.73)--(4,1.73)--(3.5,0.87)--(2.5,0.87)--(2,0)--(0,0);
    \draw (0.5,0.87)--(1,0)--(1.5,0.87)--(2,0);
    \draw (1.5,0.87)--(2.5,0.87)--(2,1.73);
    \draw (2.5,0.87)--(3,1.73)--(3.5,0.87);
    \node at(0,-0.5) {$\frac{1}{0}$};
    \node at(1,-0.5) {$\frac{0}{1}$};
    \node at(0.5,1.3) {$\frac{1}{1}$};
    \node at(1.3,1.3) {$\frac{1}{2}$};
    \node at(2,-0.5) {$\frac{1}{3}$};
    \node at(2.7,0.5) {$\frac{2}{5}$};
    \node at(2,2.2) {$\frac{3}{7}$};
    \node at(3,2.2) {$\frac{5}{12}$};
    \node at(3.5,0.5) {$\frac{7}{17}$};
    \node at(4,2.2) {$\frac{12}{29}$};
    \end{scope}
    \begin{scope}[shift={(5,0)}]
    \draw (0,0)--(2,0);
    \draw (0,0)--(0,1);
    \draw (0,1)--(3,1);
    \draw (1,2)--(3,2);
    \draw (0,0)--(1,0);
    \draw (1,0)--(1,2);
    \draw (2,0)--(2,2);
    \draw (3,1)--(3,2);
    \draw[blue] (0,0)--(3,2);
    \end{scope}
    \end{tikzpicture}
    \caption{Snake graphs associated to the Markov fraction $\frac{12}{29}$.}
    \label{fig:snakegraphs}
\end{figure}

\begin{figure}[h!]
    \centering
    \begin{tikzpicture}[scale=0.7]
    \draw (0,0)--(4,0);
    \draw (-1,2)--(5,2);
    \draw (-1,2)--(0,0);
    \draw (0,0)--(0,2);
    \draw (0,0)--(1,2);
    \draw (1,0)--(1,2);
    \draw (2,0)--(1,2);
    \draw (2,0)--(2,2);
    \draw (2,0)--(3,2);
    \draw (3,2)--(4,0);
    \draw (4,0)--(5,2);
    \node at(-1,2.4) {$\frac{1}{0}$};
    \node at(0,2.4) {$\frac{1}{1}$};
    \node at(0,-0.5) {$\frac{0}{1}$};
    \node at(1,2.4) {$\frac{1}{2}$};
    \node at(1,-0.5) {$\frac{1}{3}$};
    \node at(2,-0.5) {$\frac{2}{5}$};
    \node at(2,2.4) {$\frac{3}{7}$};
    \node at(3,2.4) {$\frac{5}{12}$};
    \node at(4,-0.5) {$\frac{7}{17}$};
    \node at(5,2.4) {$\frac{12}{29}$};
    \end{tikzpicture}
    \caption{Triangulated polygon associated to $\frac{12}{29}$.}
    \label{fig:triangulatedpolygon}
\end{figure}

More precisely, let $\mu$ be a Markov fraction and $S_{\mu}$ its triangular snake graph. This graph has an even number of triangles, so let group them two by two in order to create quadrilateral tiles, starting with the first two adjacent triangles. The continued fraction of $\mu$ is obtained by reading $S_{\mu}$ from bottom to top, with the following rules :
\begin{itemize}
    \item[$\circ$] replace the first tile by $[0,2]$,
    \item[$\circ$] then for each tile, if it is below an other tile, replace these two tiles by $[2,2]$, and if it is not below an other tile, replace it by $[1,1]$,
    \item[$\circ$] except for the last tile which is replaced by $[2]$.
\end{itemize}

\noindent Now the usual snake graph model for Markov numbers ensures that the continued fraction we get is palindromic (by palindromicity of Christoffel words, see \cite{Aigner_markov}), and that for two neighbors $\mu_1$ and $\mu_2$ in the Markov tree, the snake graph of there Springborn sum $\mu_1\oplus\mu_2$ is the concatenation of $S_{\mu_1}$ and $S_{\mu_2}$, placing the second one on top of the first one. Hence the formula for the continued fraction.
\end{proof}

\begin{notation}
Let $n\in \N^*$ and let $\alpha = (\alpha_1,\cdots,\alpha_n)$ be a sequence of $n$ non-negative integers. The corresponding fence poset $F(\alpha)$ is the linear poset with $\alpha_1 + \cdots + \alpha_n + 1$ vertices, and covering relations described by the following Hasse diagram 
\begin{center}
    \begin{tikzpicture}[x=0.7cm,y=1cm,scale=0.6]
    \node (A1) at (0,0) {$\bullet$};
    \node (A2) at (1,1) {$\bullet$};
    \node[rotate=90] (A3) at (2,2) {$\ddots$};
    \node (A4) at (3,3) {$\bullet$};
    \node (A5) at(4,4) {$\bullet$};
    \node (A6) at(5,3) {$\bullet$};
    \node[rotate=-15] (A7) at(6,2) {$\ddots$};
    \node (A8) at(7,1) {$\bullet$};
    \node (A9) at(8,0) {$\bullet$};

    \node at(10,2) {$\cdots$};

    \node (B1) at (12,0) {$\bullet$};
    \node (B2) at (13,1) {$\bullet$};
    \node[rotate=90] (A3) at (14,2) {$\ddots$};
    \node (B4) at (15,3) {$\bullet$};
    \node (B5) at(16,4) {$\bullet$};
    \node (B6) at(17,3) {$\bullet$};
    \node[rotate=-15] (B7) at(18,2) {$\ddots$};
    \node (B8) at(19,1) {$\bullet$};
    \node (B9) at(20,0) {$\bullet$};
    
    \draw (A1)--(A2);
    \draw (A4)--(A5)--(A6);
    \draw (A8)--(A9);
    \draw (B1)--(B2);
    \draw (B4)--(B5)--(B6);
    \draw (B8)--(B9);

    \node[gray] (N1) at(1.3,2.4) {$\alpha_1$};
    \draw[<-,gray] (A1)to[out=90,in=-130] (N1);
    \draw[->,gray] (N1)to[out=60,in=190] (A5);

    \node[gray] (N2) at(6.7,2.4) {$\alpha_2$};
    \draw[<-,gray] (A5)to[out=-20,in=150] (N2);
    \draw[->,gray] (N2)to[out=-70,in=90] (A9);
    \node[gray] (M1) at(13.3,2.4) {$\alpha_{n-1}$};
    \draw[<-,gray] (B1)to[out=90,in=-130] (M1);
    \draw[->,gray] (M1)to[out=60,in=190] (B5);

    \node[gray] (M2) at(18.7,2.4) {$\alpha_n$};
    \draw[<-,gray] (B5)to[out=-20,in=150] (M2);
    \draw[->,gray] (M2)to[out=-70,in=90] (B9);
    \end{tikzpicture}
\end{center}

Following \cite{MGO-2020}, we associate to each rational number $0 < x < 1$ a fence poset $F_x$, such that if $x = [0,\alpha_1,\cdots,\alpha_n]$, then the poset $F_x$ is $F(0,\alpha_1-1,\cdots,\alpha_{n-1},\alpha_n-1)$. 
\end{notation}

Now we have tools to state the combinatorial interpretation of our $q$-Markov numbers. 

\begin{lemma}\label{lemma:fenceposetint}
Let $\left(\tfrac{a_0}{b_0},\tfrac{a_1}{b_1},\tfrac{a_2}{b_2}\right)$ be a rational Markov triple. Then $B_1^\sharp$ is the generating function of order ideals of $F_{\frac{a_1}{b_1}}$, $B_1^\flat$ is the generating function of order ideals in $F_{\frac{a_1}{b_1}}$ with the last vertex counting twice, and 
\begin{equation}
\begin{cases}
    B_1^{\sharp} = [3]_qB_0^{\sharp}B_2^{\sharp} - B_0^{\sharp}A_2^{\sharp} +q^3A_0^{\sharp}B_2^{\sharp} \\
    B_1^{\flat} = [3]_qB_0^{\sharp}B_2^{\flat} - B_0^{\sharp}A_2^{\flat} + q^3A_0^{\sharp}B_2^{\flat} = [3]_qB_0^{\flat}B_2^{\sharp} - B_0^{\flat}A_2^{\sharp} + q^3A_0^{\flat}B_2^{\sharp}\\
\end{cases}
\end{equation}
\end{lemma}

Recall that by convention, if $\frac{a_0}{b_0} = 0$, we choose $A_0^\flat = 1-q^{-1}$ and $B_0^\flat = 1$.

\begin{proof}
Denote $\frac{a_0}{b_0} = [0,\alpha_1,\cdots,\alpha_n]$ and $\frac{a_2}{b_2} = [0,\beta_1,\cdots,\beta_{m}]$ the canonical continued fraction expansions. Let us suppose first that $\frac{a_0}{b_0} \neq 0$. By combinatorial interpretation of $q$-rationals given in \cite{MGO-2020} and more precisely in \cite{Aval_labbe} for rationals in $(0,1)$, $B_0^{\sharp}$ (resp. $B_2^{\sharp}$) is the generating function of order ideals in the fence poset $F_0 := F_{\frac{a_0}{b_0}}$ (resp. $F_2 = F_{\frac{a_2}{b_2}}$) and $A_0^{\sharp}$ (resp. $A_2^{\sharp}$) is the generating function of ideals of $F_0$ (resp. $F_2$) containing the first $\alpha_1$ (resp. $\beta_1$) vertices. Moreover, by Lemma \ref{lemma:contfracsum}, the fence poset $F_1$ is the concatenation of posets $F_0$ and $F_2$.
\begin{center}
    \begin{tikzpicture}[scale=0.7]
    \node (A1) at(0,1) {$\bullet$};
    \node (A2) at(1,0) {$\bullet$};
    \node (A3) at(2,1) {$\bullet$};
    \node (A4) at(4,1) {$\cdots$};
    \node (A5) at(6,1) {$\bullet$};
    \node (A6) at(7,0) {$\bullet$};
    \node (A7) at(8,1) {$\bullet$};
    
    \node (B1) at(9,2) {\tiny{$\blacksquare$}};
    \node (B2) at(10,1) {\tiny{$\blacksquare$}};
    \node (B3) at(11,0) {\tiny{$\blacksquare$}};
    \node at(9,2.5) {$1$};
    \node at(10,1.5) {$2$};
    \node at(11,-0.5) {$3$};
    
    \node (C1) at(12,1) {$\bullet$};
    \node (C2) at(13,0) {$\bullet$};
    \node (C3) at(14,1) {$\bullet$};
    \node (C4) at(16,1) {$\cdots$};
    \node (C5) at(18,1) {$\bullet$};
    \node (C6) at(19,0) {$\bullet$};
    \node (C7) at(20,1) {$\bullet$};
    
    \draw (A1.center)--(A2.center)--(A3.center);
    \draw (A5.center)--(A6.center)--(A7.center)--(B1.center)--(B2.center)--(B3.center)--(C1.center)--(C2.center)--(C3.center);
    \draw (C5.center)--(C6.center)--(C7.center);
    
    \draw[<->] (0,2) -- (8,2) node[midway,above] {$F_{0}$};
    \draw[<->] (12,2) -- (20,2) node[midway,above] {$F_2$};
    \end{tikzpicture}
\end{center}

\noindent The order ideals of this poset $F_1$ can be split into three groups :
\begin{itemize}
    \item Group 1 : those which contain vertex $1$.
    \item Group 2 : those which do not contain vertex $1$ but contain vertex $2$;
    \item Group 3 : those which do not contain vertex $2$ ;
\end{itemize}

\noindent Ideals in the group $1$ correspond to the couples of ideals in $F_0$ containing the last $\alpha_n$ vertices and ideals in $F_2$, so the generating function of group $1$ is $q^3A_0^{\sharp}B_2^{\sharp}$ (because $F_0$ is symmetric by the palindromicity of the sequence $(\alpha_1,\cdots,\alpha_n))$. \\
~\\
\noindent Ideals in the group $2$ correspond to the couples of ideals of $F_0$ and ideals of $F_2$, so the generating function of group $2$ is $q^2B_0^{\sharp}B_2^{\sharp}$.\\
~\\
\noindent Ideals in the group $3$ correspond to the couples of ideals in $F_0$ and ideals of the poset $F(1,\beta_1-1,\beta_2,\cdots,\beta_m-1)$, which are counted by the $q$-numerator of the continued fraction $[2,\beta_1-1,\beta_2,\cdots,\beta_m] = T^{2}STS(a_2/b_2)$. We have
$$
T_q^{2}S_qT_qS_q = \begin{pmatrix}
-1 & 1+q\\
 -1 &  1\\
\end{pmatrix},
$$
\noindent so the ideals of the poset $(1,\beta_1-1,\beta_2,\cdots,\beta_m-1)$ are counted by $-A_2^{\sharp} + (1+q)B_2^{\sharp}$, and the generating function of the group $3$ is $B_0^{\sharp}(-A_2^{\sharp}+(1+q)B_2^{\sharp})$.\\
~\\
\noindent Finally, we get $B_1^\sharp = q^3A_0^{\sharp}B_2^{\sharp} + q^2B_0^{\sharp}B_2^{\sharp} -B_0^{\sharp}A_2^{\sharp} + (1+q)B_0^{\sharp}B_2^{\sharp}$.\\
~\\
\noindent On the other hand, the left denominator $B_1^{\flat}$ is the generating function of $F_1$ with the final vertex counting twice. \\
\noindent Therefore $B_1^{\flat}$ is given by the same formula as $B_1^{\sharp}$ but using the left deformation of $a_2/b_2$. By palindromicitiy of the sequence defining $F_1$, the left deformation $B_1^{\flat}$ is also the generating function of order ideals in $F_1$ with the first vertex counting twice, hence the second formula for $B_1^\flat$.\\
\noindent If now $\frac{a_0}{b_0} = 0$, the same arguments apply with $F_1 = F(0,1,1,\beta_1-1,\beta_2,\hdots,\beta_m)$, so the group 1 is empty, group 2 contributes with $q^2B_2^\sharp$ and group 3 contributes with $-A_2^\sharp +(1+q)B_2^\sharp$.
\end{proof}

\subsection{Proof of the $q$-deformed Markov relations}\label{sec:proofmarkov}

We prove here the following relations for rational Markov triples 
\begin{equation*}
    (\star) \left\lbrace\begin{array}{l c r}
    B_1^\sharp B_1^{\flat} + q^{\varepsilon_0+3}B_2^\sharp B_2^{\flat} + B_0^{\flat}(B_1^\sharp A_2^\sharp -q^3A_1^\sharp B_2^\sharp) = [3]_qB_1^\sharp B_2^\sharp B_0^{\flat} & & (r_0)\\
    B_0^\sharp \equiv_q B_1^\sharp A_2^\sharp - A_1^\sharp B_2^\sharp & & (r_1)\\
    B_0^{\flat} \equiv_q B_1^\sharp A_2^{\flat} - A_1^\sharp B_2^{\flat} \equiv_q B_1^\flat A_2^\sharp - A_1^\flat B_2^\sharp & & (r_1^{\flat})\\
    B_2^\sharp \equiv_q A_1^\sharp B_0^\sharp - B_1^\sharp A_0^\sharp & & (r_2)\\
    B_2^{\flat} \equiv_q A_1^{\flat}B_0^\sharp - B_1^{\flat}A_0^\sharp \equiv_q A_1^\sharp B_0^\flat - B_1^\sharp A_0^\flat & & (r_2^{\flat})\\
    \end{array} \right.
\end{equation*}

We proceed by induction on the rational Markov tree. It is straightforward to check these relations for the initial Markov triple $\left(\frac{0}{1},\frac{2}{5},\frac{1}{2}\right)$. \\
~\\
\noindent Suppose relations $(\star)$ hold for a rational Markov triple $\left(\tfrac{a_0}{b_0},\tfrac{a_1}{b_1},\tfrac{a_2}{b_2}\right)$, and consider the child $\tfrac{a_3}{b_3} = \tfrac{a_1}{b_1}\oplus_S \tfrac{a_2}{b_2}$, part of the triple $\left(\tfrac{a_1}{b_1},\tfrac{a_3}{b_3},\tfrac{a_2}{b_2}\right)$.\\
~\\
\noindent By Theorem \ref{Thm:main}, 
\[
A_3^\sharp \equiv_q \frac{q^{\varepsilon_2}A_1^{\flat}B_1^\sharp + q^{\varepsilon_1}A_2^\sharp B_2^{\flat}}{A_2^{\flat}B_1^\sharp - A_1^\sharp B_2^{\flat}} \text{ and } B_3^\sharp \equiv_q \frac{q^{\varepsilon_2}B_1^{\flat}B_1^\sharp + q^{\varepsilon_1}B_2^\sharp B_2^{\flat}}{A_2^{\flat}B_1^\sharp - A_1^\sharp B_2^{\flat}}.
\]
\noindent Because of Lemma \ref{lemma:contfracsum}, $\varepsilon_1 = \varepsilon_0+\varepsilon_2 + 3$. Besides, by relation $(r_1^\flat)$, $A_2^{\flat}B_1^\sharp - A_1^\sharp B_2^{\flat} \equiv_q B_0^\flat$. Since Markov fractions are in $(0,1)$, their $q$-deformations must have denominators with constant coefficient $1$, so we can normalize and get
\[
A_3^\sharp = \frac{A_1^{\flat}B_1^\sharp + q^{\varepsilon_0+3}A_2^\sharp B_2^{\flat}}{B_0^\flat} \text{ and } B_3^\sharp = \frac{B_1^{\flat}B_1^\sharp + q^{\varepsilon_0+3}B_2^\sharp B_2^{\flat}}{B_0^\flat}.
\]
\noindent Let us show relation $(r_1)$ for the child. 
\begin{align*}
    B_3^\sharp A_2^\sharp - A_3^\sharp B_2^\sharp &= \frac{B_1^{\flat}B_1^\sharp + q^{\varepsilon_0+3}B_2^\sharp B_2^{\flat}}{B_0^\flat}A_2^\sharp - \frac{A_1^{\flat}B_1^\sharp + q^{\varepsilon_0+3}A_2^\sharp B_2^{\flat}}{B_0^\flat}B_2^\sharp\\
    &= B_1^\sharp \frac{A_2^\sharp B_1^\flat - B_2^\sharp A_1^\flat}{B_0^\flat}\\
    &\equiv_q q^{\varepsilon_2}B_1^\sharp\\
\end{align*}
\noindent where the last equality comes from relation $(r_1^\flat)$.\\
\noindent The relation $(r_2)$ for the child follows from similar computations, using the other expression of $A_3^\sharp = \frac{A_1^\sharp B_1^\flat + q^{\varepsilon_0+3}A_2^\flat B_2^\sharp}{B_0^\flat}$ (recall Remark \ref{rem:sym_sharp_and_flat}).\\
\noindent First equality of relation $(r_1^\flat)$ for the child is straightforward. The second equality is true up to a power of $q$ by Theorem \ref{Thm:q-gcd}.
\noindent Relation $(r_2^\flat)$ is symmetric. \\
\noindent Now relation $(r_0)$ for the child comes from Lemma \ref{lemma:fenceposetint} applied to the triple $(\frac{a_3}{b_3},\frac{a_3}{b_3}\oplus_S \frac{a_2}{b_2},\frac{a_2}{b_2})$, combined with Theorem \ref{Thm:main} : 
$$
[3]_qB_3^\sharp B_2^\sharp - B_3^\sharp A_2^\sharp + q^3A_3^\sharp B_2^\sharp = \frac{B_3^{\flat}B_3^\sharp + q^{\varepsilon_1+3}B_2^\sharp B_2^{\flat}}{B_1^\flat}.
$$

\noindent To finish induction, it remains to check the same relations $(\star)$ for the other child of the triple $\left(\frac{a_0}{b_0},\frac{a_1}{b_1},\frac{a_2}{b_2}\right)$, with $\frac{a_0}{b_0}\oplus_S\frac{a_1}{b_1}$, and the arguments are symmetric to the previous case. This concludes the proof of Theorem \ref{thm:qdeformedequations}. 
~\\

\begin{question}
In analogy to $q$-binomials counting points in Grassmannians over finite fields, is there a geometric interpretation of $q$-deformed Markov fractions? Veselov's work \cite{Veselov} gives an interpretation of classical Markov fractions as slopes of exceptional bundles over $\mathbb{P}^2$.
\end{question}

\subsection{Companions of Markov fractions via Springborn's difference}
In \cite{Springborn}, Springborn is interested in the Diophantine approximations of the rational numbers, and especially in the question of bounding, for $x \in \Q$, the following constant from below :

$$
C(x):= \inf_{\frac{a}{b} \in \Q \setminus \left\{x\right\}} b^2 \left| x-\frac{a}{b} \right|.
$$

The ``worst'' cases are $C(x)=1$ - only in the case when $x$ is an integer, and $C(x)=1/2$ - only when $x$ is a half-integer, $x \in \Z+1/2$. In the title result of his paper, Springborn shows that the next worst case is that of companions of Markov fractions.

\begin{theorem}[\cite{Springborn}]
The approximation constant $C(x) \geq \frac{1}{3}$ if and only if $x$ is a Markov fraction or its companion. Furthermore, the equality is attained only for left or right (second) companions $c_2^{\pm} \left( \frac{a}{b}\right)$ of Markov fractions. 
\end{theorem}
While Markov fractions are defined by iterating the Springborn sum as we discussed in paragraph \ref{subs:iterating the Springborn sum}, their companions are defined by a following recursive procedure. 

\begin{definition}\label{def:Markov_companions}
For a Markov fraction $\frac{a}{b} \in \Q$, define two sequences of its (right and left) companions $c_k^+$ and $c_k^-, k=1,2, \ldots$ as follows: 
\begin{equation*}
    c_k^{\pm}:=\frac{a}{b} \pm \frac{u_{k-1}}{bu_k},
\end{equation*}
where $u_k$ is defined via a following recursive linear equation:
\begin{equation}\label{eq:sequence_u}
    u_0=0, u_1=1, u_{k+1}=3 b u_k-u_{k-1}.
\end{equation}
\end{definition}

In his work, Springborn gives geometric interpretations to Markov fractions and their companion sequences : Markov fractions correspond to simple geodesics in the modular torus with both ends in the cusps, and companions correspond to non-simple geodesics wit both ends in the cusps that do not intersect a pair of disjoint simple geodesics, cutting the topology of the modular torus. 

We propose a simple iterative procedure defining the companions of Markov fractions -- it happens that they can be defined via the Springborn difference!

\begin{proposition}
Let $c_k^{\pm}$ be a sequence of companions of some Markov fraction $\frac{a}{b} \in \mathbb{Q}$. Then, for any $k,l \geq 1$, $k\neq l$:
\begin{equation}\label{Eq:ck+}
    c_k^\pm \ominus_S c_l^\pm = c_{k+l}^\pm.
\end{equation}
Moreover, the corresponding sequence $u_k$ from Definition \ref{def:Markov_companions} is a sequence of special values of Chebyshev polynomials of the second kind $U_{k-1}$ (see Remark \ref{rem:Chebyshev_polynomials} below):
$$u_k = U_{k-1} \left(\frac{3b}{2}\right), \; \textit{for all} \; k \geq 1.
$$
\end{proposition}

\begin{proof}
For $u_k$ defined via the above Definition \ref{def:Markov_companions}, we first claim that \eqref{Eq:ck+} is equivalent to 
\begin{equation}\label{eq:relation_on_u}
\frac{u_{k-1}}{u_k} \ominus_S \frac{u_{l-1}}{u_l}=\frac{u_{k+l-1}}{u_{k+l}}.
    \end{equation}

For this, we show that $c_k^\pm = \tfrac{a}{b}\pm\tfrac{u_{k-1}}{bu_k}=\tfrac{au_k\pm u_{k-1}}{bu_k}$ is a reduced fraction. The relation
\begin{equation}\label{eq:easy}
u_k u_{k-2} = u_{k-1}^2 - 1,
\end{equation}
proven in \cite[Theorem 3.14]{Springborn}, shows that $\gcd(a u_k\pm u_{k-1}, u_k)=\mathrm{gcd} (u_{k-1}, u_k)=1$.
Applying twice the recurrence relation \eqref{eq:sequence_u}, we get
$$\gcd(a u_k\pm u_{k-1},b)=\gcd (\pm u_{k-1}-a u_{k-2},b)=
\gcd (a u_{k-2}\pm u_{k-3},b).$$
Depending on the parity of $k$, one finishes with either $\gcd (a u_1\pm u_0,b)=\gcd(a,b)=1$ or $\gcd (a u_2\pm u_1, b)=\gcd (ab\pm 1,b)=1$.

Using the definition of the Springborn difference, a direct computation shows that equation \eqref{Eq:ck+} is indeed equivalent to equation \eqref{eq:relation_on_u}. 

\smallskip
Second, we will prove \eqref{eq:relation_on_u} by induction on $N=k+l$. 

Let us first notice that for $l=1$, the statement is equivalent to \eqref{eq:easy}.
Then, suppose $k+l=N$ and that the statement is already proven for $k+l<N$. We calculate the studied Springborn difference and use the induction once for $N-2$:

\begin{align*}
    \frac{u_{k-1}}{u_k} \ominus_S \frac{u_{l-1}}{u_l} &=\frac{u_{k-1} \left( 3b u_{k-1} -u_{k-2}
    \right) - u_{l-1} \left( 3b u_{l-1} -u_{l-2}
    \right)}{u_k^2-u_l^2}\\
    &=\frac{\left(
    u_{k-1}^2-u_{l-1}^2
    \right) \left(
    3 b u_{k+l-2} - u_{k+l-3}
    \right)
    }{\left(u_k^2-u_l^2\right) u_{k+l-2}} \\
    &= \frac{u_{k-1}^2-u_{l-1}^2}{u_k^2-u_l^2}\, \frac{u_{k+l-1}}{u_{k+l-2}}\\
    &=\frac{u_{k+l-1}}{u_{k+l}}.
\end{align*}

In the last equality, we use that 
\begin{equation}\label{eq:general_form}
(u_k-u_l) (u_k+u_l)=u_{k}^2-u_l^2=u_{k-l} u_{k+l}
\end{equation}
for all $k$ and all $0 \leq l \leq k$, which is a generalisation of equation \eqref{eq:easy} (where $l=1$). 

Comparing the defining recursion \eqref{eq:sequence_u} for $u_k$ to the one defining the Chebychev polynomials of the second kind (see Remark \ref{rem:Chebyshev_polynomials} below), we see that $u_k = U_{k-1} \left(\frac{3b}{2}\right).$ The equality \ref{eq:general_form} is a standard recurrence relation on Chebyshev polynomials of the second kind.
This finishes the proof.
\end{proof}

\begin{example}
The companions of the fraction $\frac{0}{1}$, are the following : $c_1=0/1, c_2=1/3, c_3=3/8, c_4=8/21, c_5=21/55, c_6=55/144, c_7=144/377, \ldots$, see in particular Figure $5$ in \cite{Springborn} for the list of companions for the first thirteen Markov fractions.

The above proposition shows, in particular, that 
\begin{equation*}
    \frac{144}{377}=\frac{1}{3} \ominus_S \frac{21}{55} = \frac{3}{8} \ominus_S \frac{8}{21}.
\end{equation*}
\end{example}

\begin{Remark}\label{rem:Chebyshev_polynomials}
Let us remind the standard definition of Chebyshev polynomials of the second kind $U_k(x)$. They are defined via the following recurrence : 
$$U_{k+1}(x) = 2x U_k (x) - U_{k-1} (x),$$
with $U_0(x)=1, U_1(x)=2x$. 
They express the fact that the fraction $\frac{\sin ((n+1) \theta)}{\sin \theta}$ (as well as the function $\frac{\sinh ((n+1) \theta)}{\sinh \theta}$) is a polynomial in $\cos \theta$ (respectively, $\cosh \theta$).

Note that if $b \neq 0$, then $3b/2 \geq 1.5$ and we are in the hyperbolic case. Remind the following well-known identity:
$$ 
U_k^2(x)-U_l^2(x)=U_{k-l-1}(x) \cdot U_{k+l+1}(x).
$$
The proof of this identity is to substitute $U_n(x)=\frac{\sinh ((n+1) \theta)}{\sinh \theta}$ with $x=\cosh \theta$ (recall that $x \geq 3/2$ in our setting), and to use $\sinh^2 \varphi_1 - \sinh^2 \varphi_2 = \sinh (\varphi_1-\varphi_2) \sinh (\varphi_1 + \varphi_2)$. 
\end{Remark}

\bigskip

\bibliographystyle{plain}
\bibliography{ref-arxiv}

\end{document}